\documentclass[12pt]{amsart}

\usepackage{graphicx}
\usepackage{epstopdf}
\usepackage{amsmath}
\usepackage{enumerate}
\usepackage{tikz}
\usepackage{wrapfig}
\usepackage{pinlabel}
\usepackage{subfig}
\usepackage{mathabx}

\newtheorem{theorem}{Theorem}[section]
\newtheorem{lemma}[theorem]{Lemma}

\newtheorem{proposition}[theorem]{Proposition}

\theoremstyle{definition}
\newtheorem{definition}[theorem]{Definition}

\theoremstyle{remark}

\numberwithin{equation}{section}

\newcommand{\T}{\mathcal{T}}
\newcommand{\storus}{S^1 \times D^2}

\newcommand{\K}{K(S^1\times D^2)}
\newcommand{\Krel}{K(S^1\times D^2,2)}
\newcommand{\laurent}{\mathbb{Z}[A,A^{-1}]}

\newcommand{\G}{\mathcal{G}}
\newcommand{\A}{\mathcal{A}}

\newcommand{\C}{\mathcal{C}}
\newcommand{\s}{\mathcal{S}}

\newcommand{\f}{\mathcal{F}}
\newcommand{\vareps}{\varepsilon}
\newcommand{\ds}{\displaystyle}
\newcommand{\pair}[1]{\langle #1 \rangle}

\begin{document}

\title{The Kauffman bracket ideal for genus-1 tangles}
\author{Susan M. Abernathy}
\address{Mathematics Department\\
Louisiana State University\\
Baton Rouge, Louisiana}
\email{sabern1@tigers.lsu.edu}

\date{}


\begin{abstract}
Given a compact oriented 3-manifold $M$ in $S^3$ with boundary, an $(M,2n)$-tangle $\mathcal{T}$ is a 1-manifold with $2n$ boundary components properly embedded in $M$.  We say that $\mathcal{T}$ embeds in a link $L$ in $S^3$ if $\mathcal{T}$ can be completed to $L$ by a 1-manifold with $2n$ boundary components exterior to $M$.  The link $L$ is called a closure of $\mathcal{T}$.  We define the Kauffman bracket ideal of $\mathcal{T}$ to be the ideal $I_\mathcal{T}$ of $\mathbb{Z}[A,A^{- 1}]$ generated by the reduced Kauffman bracket polynomials of all closures of $\mathcal{T}$.  If this ideal is non-trivial, then $\mathcal{T}$ does not embed in the unknot.  We give an algorithm for computing a finite list of generators for the Kauffman bracket ideal of any $(S^1 \times D^2, 2)$-tangle, also called a genus-1 tangle, and give an example of a genus-1 tangle with non-trivial Kauffman bracket ideal. Furthermore, we show that if a single-component genus-1 tangle $\hat \T$ can be obtained as the partial closure of a $(B^3, 4)$-tangle $\mathcal{T}$, then $I_\mathcal{T} = I_{\hat \T}$.
\end{abstract} 



\maketitle

\section{Introduction}\label{intro}

Let $M$ be a compact oriented 3-dimensional submanifold of $S^3$ with boundary.  Then an $(M,2n)$-tangle $\T$ is a 1-manifold with $2n$ boundary components properly embedded in $M$.  We say that $\T$ embeds in a link $L \subset S^3$ if $\T$ can be completed by a 1-manifold with $2n$ boundary components exterior to $M$ to form the link $L$; that is, there exists some 1-manifold with $2n$ boundary components in $S^3 -Int(M)$ such that upon gluing this manifold to $\T$ along their boundary points, we have a link in $S^3$ which is isotopic to $L$.  We say that $L$ is a closure of $\T$.

This definition naturally gives rise to the following question: given an $(M,2n)$-tangle $\T$ and a link $L \subset S^3$, when does $\T$ embed in $L$?

This embedding question has been studied before in the case where $M = B^3$ (see \cite{kr,psw,rub}) and discussed in the case where $M = \storus$ in \cite{kr} and \cite{rub}.  In \cite{kr}, Krebes asked whether the genus-1 tangle pictured in Fig. \ref{fig:krebesexample}, denoted by $\A$, can be embedded into the unknot.  It was this question that first motivated our interest in the topic of tangle embedding.  We partially answer this question in \cite{ab} using methods different than those in this paper.

\begin{figure}
\includegraphics[height=1in]{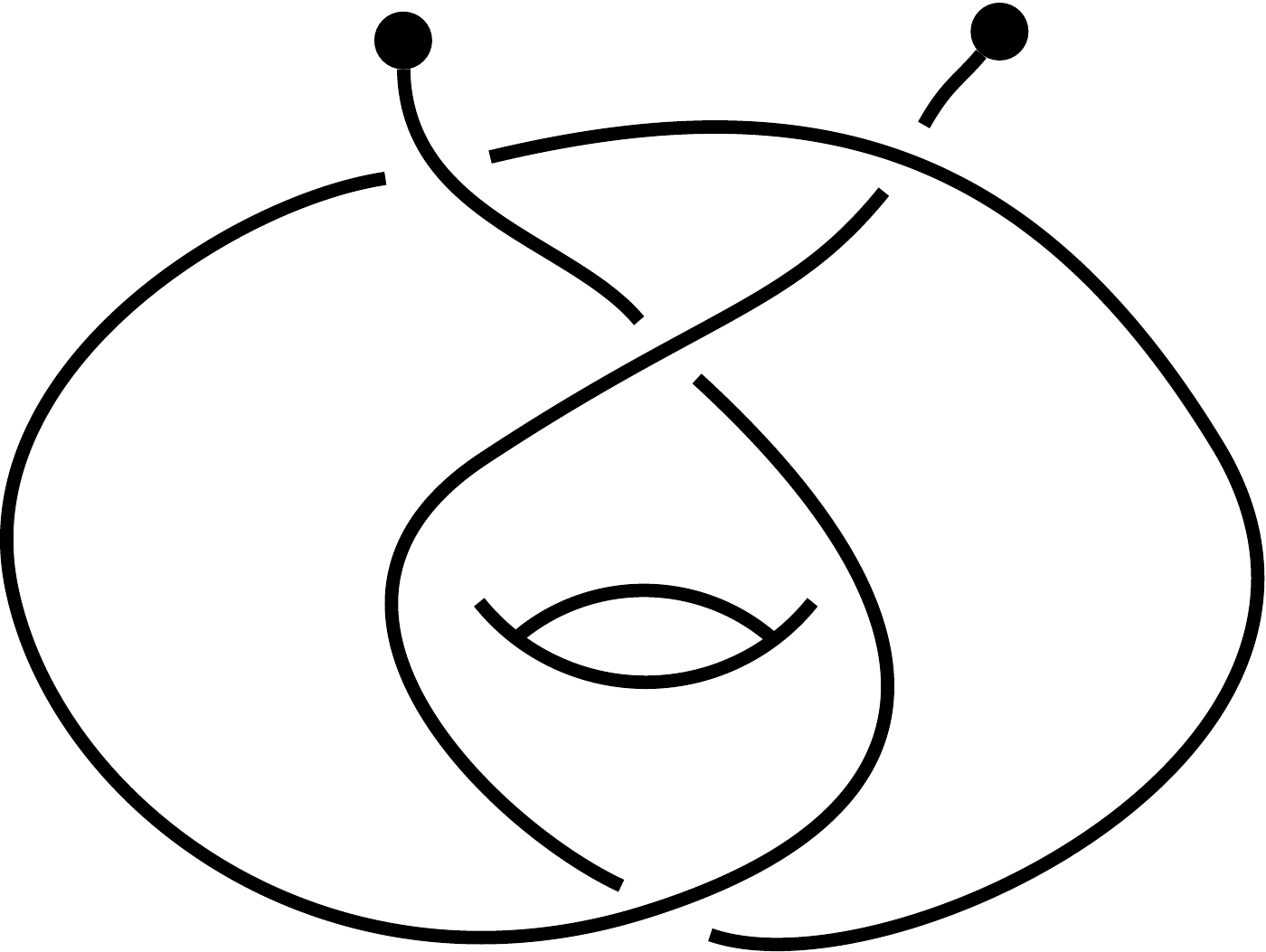}
\caption{Krebes's example, which we denote by $\A$.}\label{fig:krebesexample}
\end{figure}

Though our main concern in this paper is the case where $M$ is a solid torus, we first consider the case where $M=B^3$.  Suppose a $(B^3,2n)$-tangle $\T$ embeds in a link $L$.  Then the complement of $\T$ in $L$ is also a $(B^3,2n)$-tangle, since it is a 1-manifold with $2n$ boundary points properly embedded in the 3-ball $S^3 - Int(B^3)$.  Let $\s$ denote this complementary tangle.  We may view $L$ as the union of $\s$ and $\T$ along their boundary points.  In this case we refer to $L$ as the closure of $\T$ by $\s$, denoted by $\T^\s$.

In \cite{psw}, Przytycki, Silver and Williams examine the ideal $I_\T$ associated to a $(B^3,2n)$-tangle $\T$ generated by the reduced Kauffman bracket polynomials of certain closures of $\T$.  The Kauffman bracket polynomial of a link (diagram) $L$ is denoted by $\langle L \rangle$. From the definition given in Section \ref{subsection:kbsm}, it is clear that the Kauffman bracket polynomial of any non-empty link $L \subset S^3$ is a multiple of $\delta = -A^2 - A^{-2}$.  So we define the reduced Kauffman bracket polynomial to be $\langle L \rangle ^\prime = \langle L\rangle/\delta \in \laurent$.

The following theorem, proven in \cite{psw}, gives an obstruction to $(B^3,2n)$-tangles embedding in links.  A $2n$-Catalan tangle $\C$ is a crossingless $(B^3,2n)$-tangle with no trivial components.

\begin{theorem}[Przytycki, Silver, and Williams]\label{thm:psw}
If a $(B^3,2n)$-tangle $\T$ embeds in a link $L$, then the ideal $I_\T$ of $\laurent$ generated by the reduced Kauffman bracket polynomials of all diagrams $\langle \T^\C\rangle^\prime$, where $\C$ is any Catalan tangle, contains the polynomial $\langle L\rangle^\prime$.
\end{theorem}

In the case where $2n = 4$, there are only two Catalan tangles and thus $I_\T$ is generated by the reduced Kauffman bracket polynomials of the two tangles in Fig. \ref{fig:numdenom}.  These are the numerator $n(\T)$ and denominator $d(\T)$ closures of $\T$.

\begin{figure}
\centering
\subfloat[][$n(\T)$]{
\includegraphics[width=1.2in]{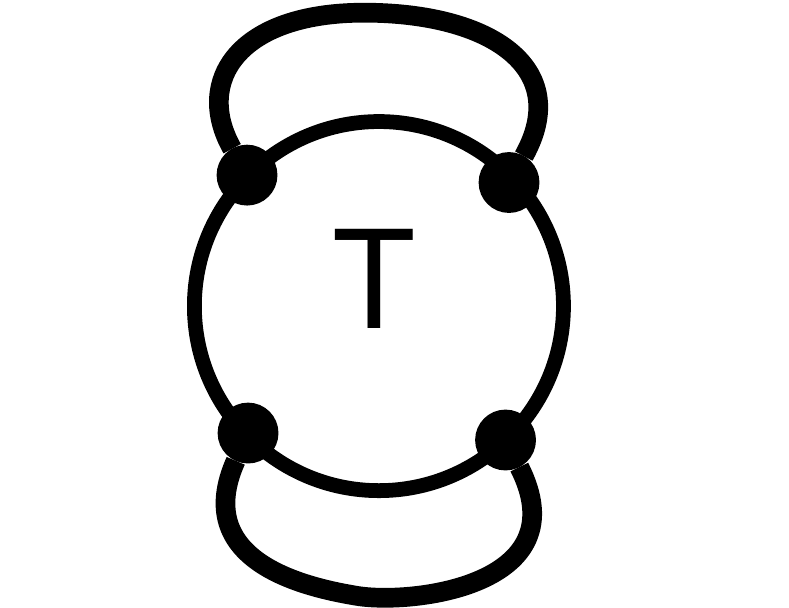}
\label{fig:num}}
\qquad\qquad
\subfloat[][$d(\T)$]{
\includegraphics[height=.65in]{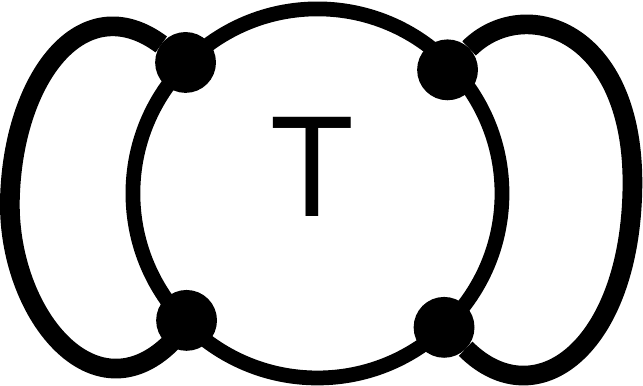}
\label{fig:denom}}
\caption{The numerator $n(\T)$ and the denominator $d(\T)$ of  a $(B^3,4)$-tangle $\T$.}\label{fig:numdenom}
\end{figure}

In \cite{psw}, it is noted that Theorem \ref{thm:psw} may be viewed in a skein theoretic light.  Recall that any closure of a $(B^3,2n)$-tangle $\T$ can be viewed as the union of $\T$ and a complementary $(B^3,2n)$-tangle $\s$ along their boundary points.  We may view both $\T$ and $\s$ as elements of the relative Kauffman bracket skein module $K(B^3,2n)$.  Then we can describe the closure of $\T$ by $\s$ in terms of a symmetric bilinear pairing  $\langle \text{ , } \rangle:K(B^3,2n) \times K(B^3,2n) \rightarrow K(S^3) = \laurent$ defined as follows:
\[\left\langle\text{ } \begin{minipage}{.6in}\includegraphics[width=.5in]{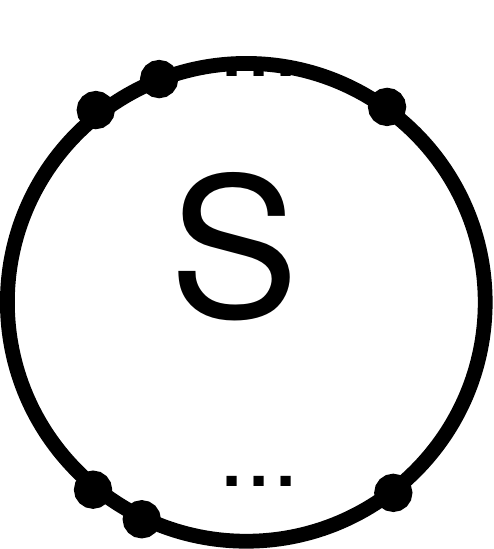}\end{minipage}, \begin{minipage}{.6in}\includegraphics[width=.5in]{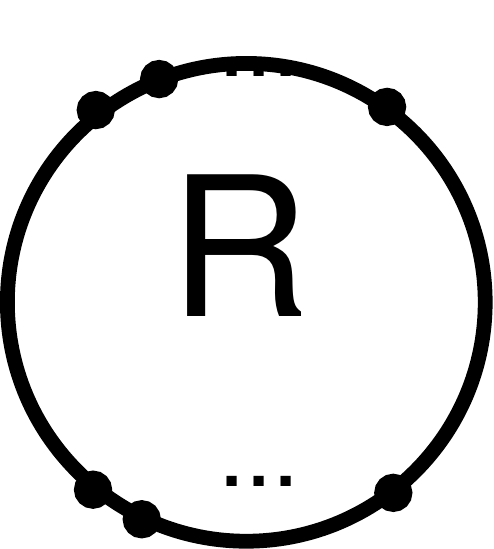}\end{minipage}\right\rangle = \left\langle \begin{minipage}{1.3in}\includegraphics[width=1.25in]{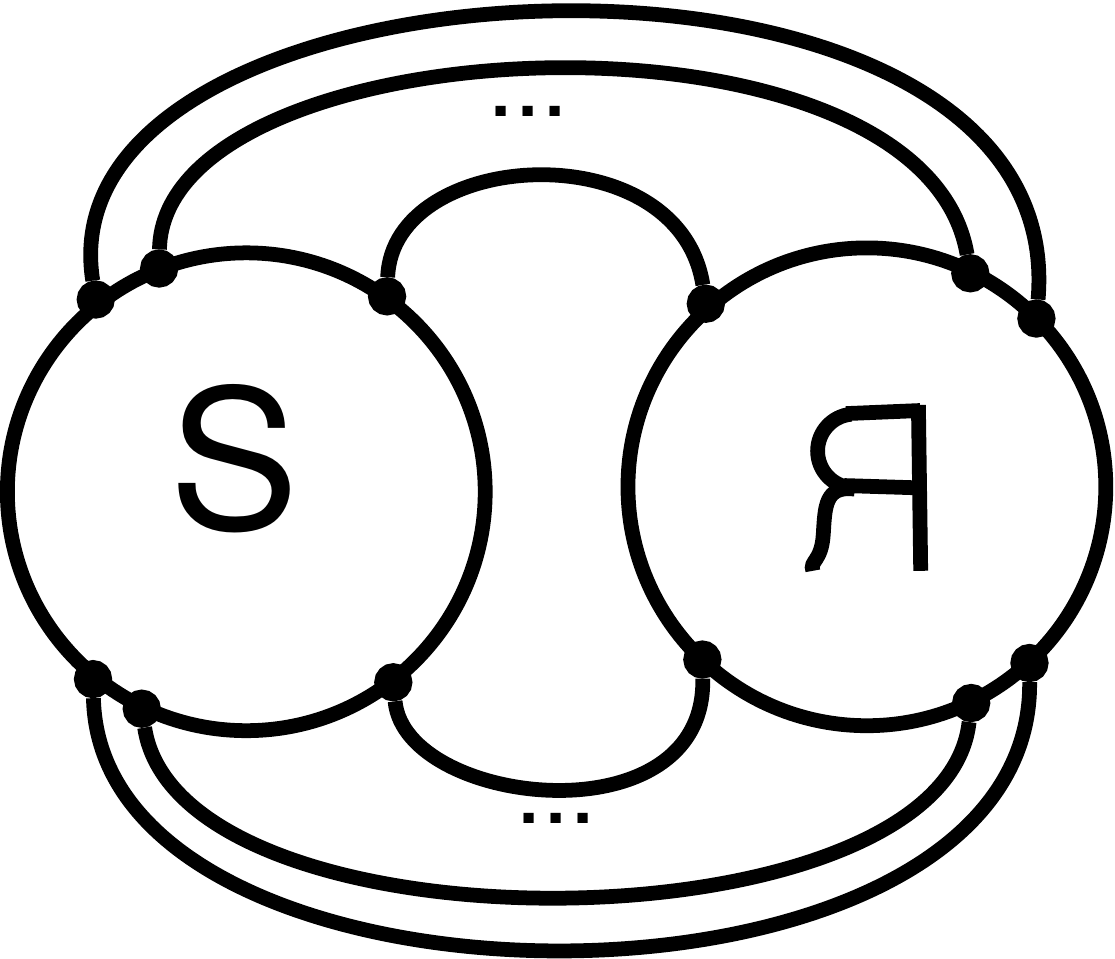}\end{minipage}\right\rangle.\]
Any closure of $\T$ may be written as $\langle \T, \s \rangle$ for some $(B^3,2n)$-tangle $\s$.  Since the set of all $2n$-Catalan tangles forms a basis for $K(B^3,2n)$, we see that any such tangle $\s$ can be written as a linear combination of Catalan tangles. So the ideal $I_\T$ is generated by pairings $\langle \T, \C\rangle/\delta$ where $\C$ is a Catalan tangle.  Furthermore, this means that an equivalent way to think about $I_\T$ is as the ideal generated by the reduced Kauffman bracket polynomials of all closures of $\T$.

We generalize this ideal to $(M,2n)$-tangles.  Given an $(M,2n)$-tangle $\T$, let $I_\T$ denote the ideal of $\laurent$ generated by the reduced Kauffman bracket polynomials of all closures of $\T$.  We call this the Kauffman bracket ideal of $\T$.  Note that if $M =B^3$, this is the same ideal defined in Theorem \ref{thm:psw}.  If $I_\T = \laurent$, we refer to $I_\T$ as a trivial ideal.  The following proposition is an immediate consequence of the definition.

\begin{proposition}\label{prop:kbideal}
If an $(M,2n)$-tangle $\T$ embeds in a link $L \subset S^3$, then $\langle L \rangle^\prime \in I_\T$.
\end{proposition}

If $\T$ embeds in the unknot, then $I_\T$ is trivial since the reduced Kauffman bracket polynomial of the unknot is one.  So, Proposition \ref{prop:kbideal} gives an obstruction tangle embedding; if $I_\T$ is non-trivial, then $\T$ does not embed in the unknot.

Our main concern in this paper is applying this obstruction to $(\storus, 2)$-tangles, which we refer to as genus-1 tangles.  We apply it first to Krebes's genus-1 tangle $\A$ in Fig. \ref{fig:krebesexample}. A brief examination shows that both the figure-eight knot and a $-1$-framed trefoil are closures of $\A$,  so $f = A^{-8} -A^{-4}+1-A^4+A^8$ and $g = A^{-8} +1-A^4$ are two generators of $I_\A$.  A short computation shows that $A^{-4}f + (1-A^{-4})g = 1$, and thus $I_\A$ is trivial.  So Proposition \ref{prop:kbideal} does not provide an obstruction to Krebes's example embedding in the unknot.

Obviously, we cannot always compute the Kauffman bracket ideal of a genus-1 tangle by simply examining some number of closures as we did with Krebes's tangle since the ideal has infinitely many generators by definition. One can give a finite list of generators for the ideal of a $(B^3,2n)$-tangle because the Catalan tangles are a finite basis for the relative Kauffman bracket skein module $K(B^3, 2n)$.  We generalize this method to the case of genus-1 tangles.  For this we must consider the relative Kauffman bracket skein module $K(\storus, 2)$ which is infinite dimensional.  Nevertheless, we outline an algorithm for finding an explicit finite list of generators for the Kauffman bracket ideal $I_\G$ of any genus-1 tangle $\G$. 

We use two bases for the Kauffman bracket skein module of $\storus$ relative to two points on the boundary.  The first basis is for the skein module over the ring $\laurent$ localized by inverting the quantum integers, and involves banded trivalent graphs.  We discuss banded trivalent graphs and define this basis in Section \ref{section:graphbasis}.  Gilmer  discussed this type of basis for a handlebody with colored points in a course on quantum topology in the fall of 2001. It is the generic version of the basis discussed in \cite[Theorem 4.11]{bhmv2}.  The second basis is for the skein module over $\laurent$ and is related to the orthogonal basis $\{Q_n\}$ defined in \cite{bhmv92}.  We discuss this basis in Section \ref{section:almostorthbasis}.

Then in Section \ref{sec:fingen}, we outline an algorithm for finding a finite list of generators for the Kauffman bracket ideal $I_\G$ of any genus-1 tangle $\G$.  

In Section \ref{section:1057}, we use this method to show that Proposition \ref{prop:kbideal} does provide an obstruction for the genus-1 tangle $\f$ pictured in Fig. \ref{fig:1057}, and we prove the following theorem.

\begin{theorem}\label{thm:1057}
 The Kauffman bracket ideal $I_\f$ of $\f$  is non-trivial.  In fact, $I_\f= \langle 11, 4 - A^4 \rangle$.  If a link $L \subset S^3$ is a closure of $\f$ and $J_L(\sqrt{t})$ is the Jones polynomial of $L$, then $J_L(\sqrt{t})|_{\sqrt{t}=5} = 0 \pmod{11}$.
\end{theorem}

\begin{figure}
\includegraphics[height=1.5in]{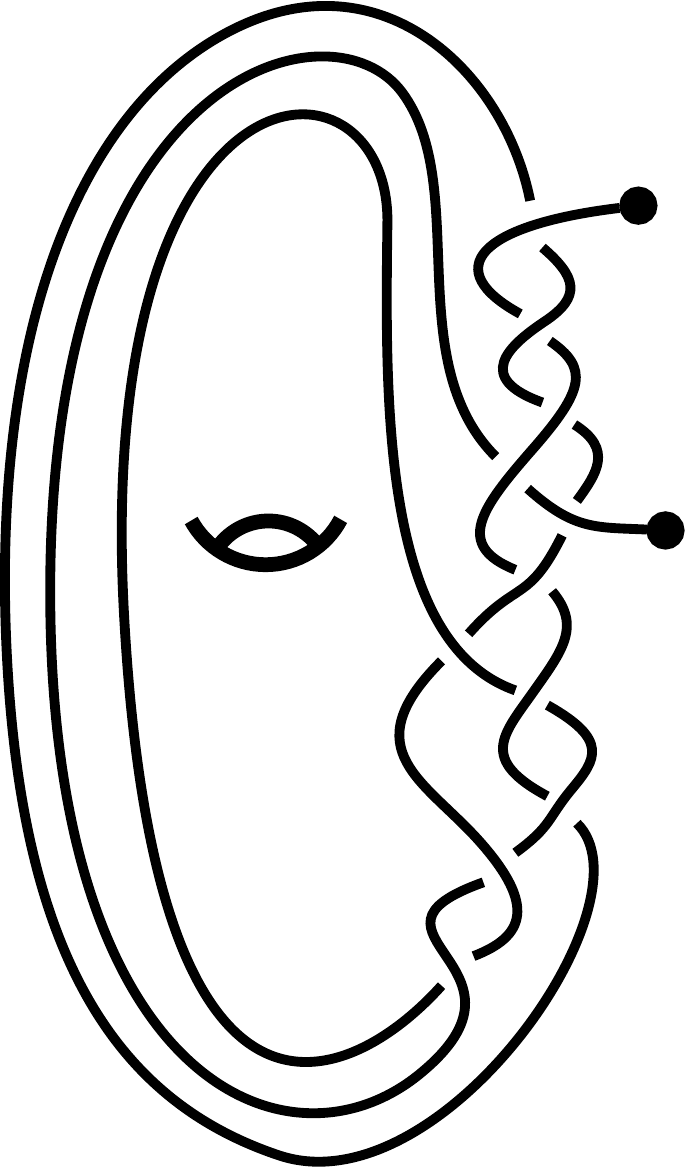}
 \caption{The genus-1 tangle $\f$.}\label{fig:1057}
\end{figure}

Of course, one could easily give an example where the Kauffman bracket ideal is non-trivial because the genus-1 tangle contains a local knot or has a $(B^3,4)$-subtangle with non-trivial ideal.  The genus-1 tangle $\f$ contains no local knots and does not appear to have any $(B^3,4)$-subtangles with non-trivial Kauffman bracket ideals.  To find this example, we used the concept of partial closures, which we discuss in Section \ref{section:partialclosures}.

The partial closure of a $(B^3,2n)$-tangle $\T$ is the genus-1 tangle obtained from $\T$ by gluing a copy of $D^2 \times I$ containing $n-1$ properly embedded arcs to $B^3$ as indicated in Fig. \ref{fig:partialclosure}.  We denote the partial closure by $\hat \T$.  

If a $(B^3,4)$-tangle consists of exactly two arcs embedded in $B^3$, then its partial closure either has a single component (if the partial closure joins boundary points from the two different arcs) or two components (if the partial closure joins boundary points of the same arc).  If it has a single component, then we have the following surprising result which we prove in Section \ref{section:partialclosures}.

\begin{theorem}\label{thm:partialclosures}
Let $\T$ be a $(B^3,4)$-tangle and let $\hat \T$ denote the genus-1 tangle which is the partial closure of $\T$.  If $\hat \T$ has a single component, then $I_{\hat\T} = I_\T$.
\end{theorem}

\begin{figure}
\includegraphics[height=1.5in]{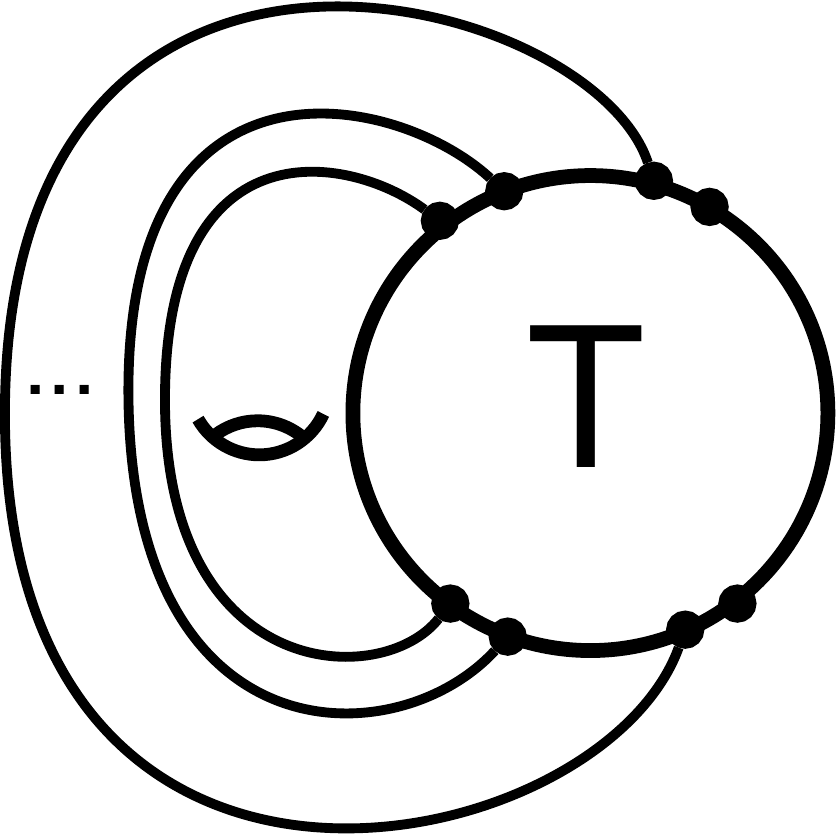}
 \caption{The partial closure $\hat \T$ of a $(B^3,2n)$-tangle $\T$.}\label{fig:partialclosure}
\end{figure}

This result influenced our search for an example of a genus-1 tangle with non-trivial Kauffman bracket ideal because any genus-1 tangle with one component which intersects some meridional disk of the solid torus exactly once can be viewed as the partial closure of a $(B^3,4)$-tangle.  Thus, its Kauffman bracket ideal can easily be computed using Theorem \ref{thm:psw}.  So, we should consider only those genus-1 tangles which intersect every meridional disk in the solid torus at least twice.  In particular, we considered partial closures of braids when looking for an example and used Mathematica to make our search more efficient.

Any braid $B$ on $n$ strands can be viewed as a $(B^3, 2n)$-tangle.  So, we can obtain a genus-1 tangle from $B$ by taking the partial closure of $B$.  Furthermore, certain closures of any genus-1 tangle obtained from a braid are easy to describe in Mathematica.

Since $B$ has an inverse element $B^{-1}$ in the braid group, it is easy to see that some closure of the $(B^3,2n)$-tangle consisting of $B$ concatenated with $B^{-1}$ is the unknot, and we have the following easy proposition.

\begin{proposition}\label{prop:braidideals}
For any $(B^3, 2n)$-tangle $B$ which is a braid, we have that $I_B = \laurent$.
\end{proposition}
Furthermore, any subtangle of a braid $B$ also has trivial Kauffman bracket ideal.

We do not consider 2-stranded braids, since any 2-stranded braid can be viewed as a $(B^3,4)$-tangles and thus satisfies Theorem \ref{thm:partialclosures}.  Furthermore, it is easy to see that any 2-stranded braid embeds in either the unknot or the 2-component unlink, depending on whether the braid has an odd or even number of twists.  So we consider only partial closures of braids with at least three strands.

We wrote a Mathematica program using Bar-Natan's KnotTheory package \cite{bn} to make detecting potential examples easier.  It computes the ideal generated by certain closures of the partial closure of certain braids.  This notebook is available on the author's website \footnote{\texttt{https://math.lsu.edu/\~{ }sabern1}}.  It proceeds as follows. Given the $n$th knot with $m$ crossings, we obtain a braid representative $br[m,n]$ of the knot.  From $br[m,n]$, we obtain a genus-1 tangle by taking its partial closure $\mathcal{G}$.  We then examine some particular closures of $\G$.

\begin{figure}
\includegraphics[width=1in]{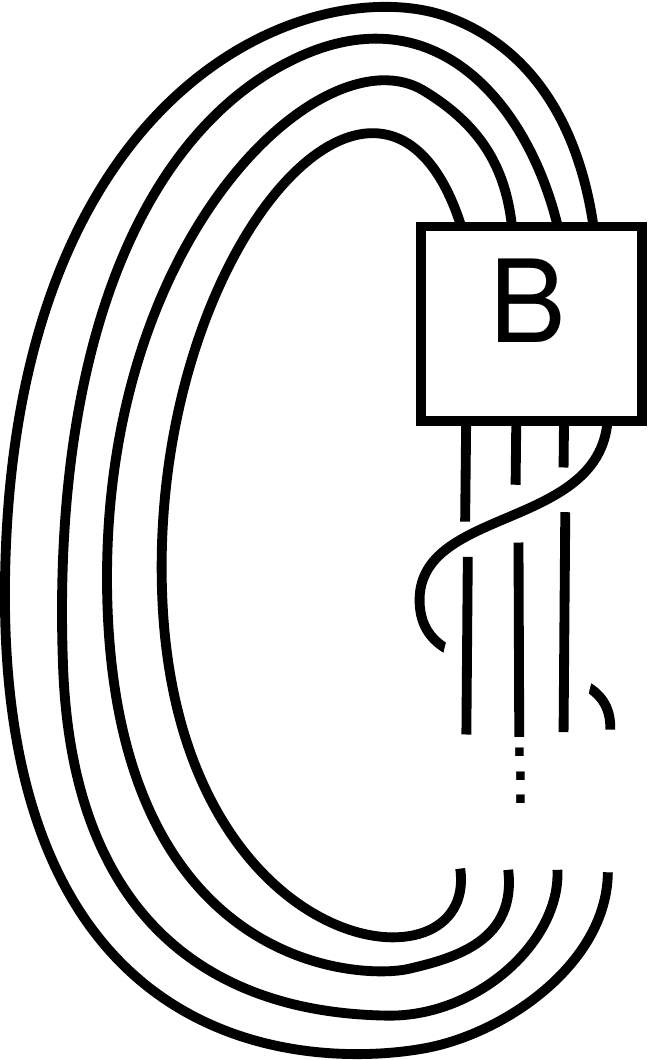}\hspace{.75in}\includegraphics[width=1in]{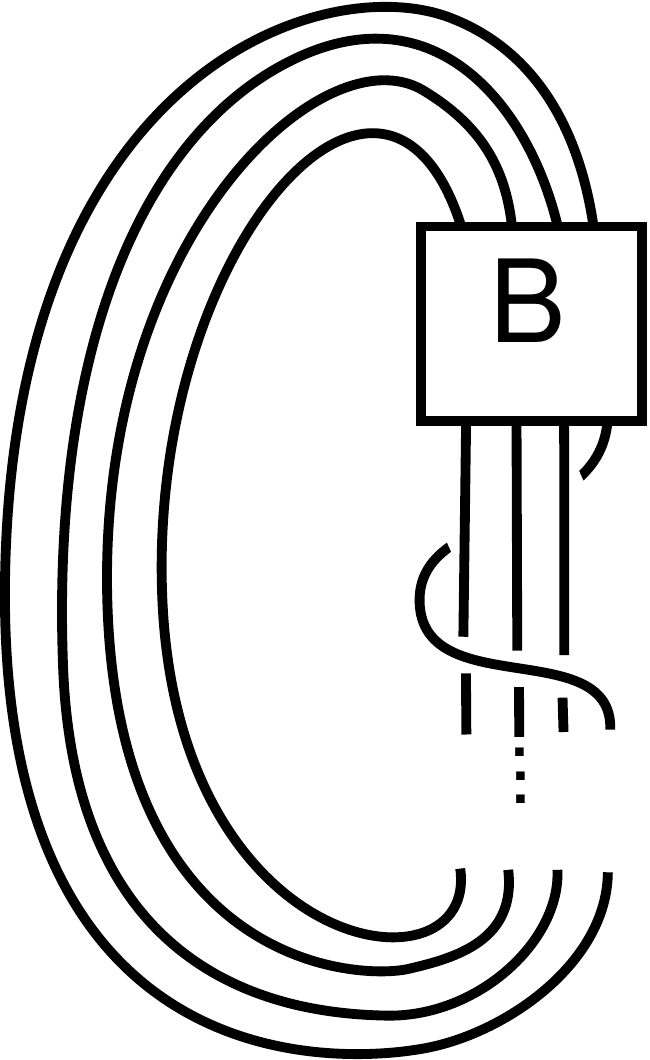}
\caption{Some closures of a genus-1 tangle obtained as the partial closure of a braid $B$ on four strands.}\label{fig:braidgens}
\end{figure}

The closures of $\mathcal{G}$ we consider are those in which the strand closing the tangle wraps around through hole $n$ of times either front to back or back to front, for some positive integer $n$, as in Fig. \ref{fig:braidgens}.  Such a closure can be viewed as the closure of the braid $br[m,n]$ concatenated $n$ times with one of the following braids:
\[P = \begin{minipage}{.6in}\includegraphics[width=.55in]{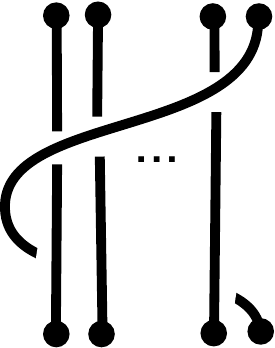}\end{minipage}\text{ or } N = \begin{minipage}{.6in}\includegraphics[width=.55in]{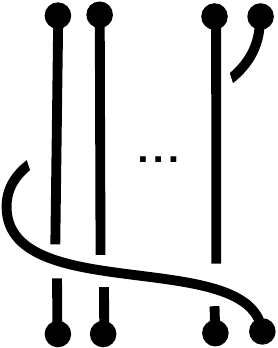}\end{minipage} .\]  We consider eleven closures of $\G$: $br[m,n]$ concatenated with each of $P$ and $N$ up to five times, along with $br[m,n]$ itself.

Our program then computes the Jones polynomials of these closures and rescales them as follows: if the smallest exponent of $t$ appearing in the Jones polynomial is negative, then we multiply the Jones polynomial by the power of $t$ necessary to make that smallest degree term a constant; if all exponents of $t$ in the Jones polynomial are positive, we do nothing.  These rescaled Jones polynomials lie in $\mathbb{Z}[t]$ and generate an ideal.  Our program computes a Groebner basis for this ideal.  The tangles for which this ideal was non-trivial formed our list of potential examples.

For a fixed integer $k$, our program does the computation described above for every knot up to 10 crossings whose braid representative has $k$ strands.  All knots whose braid representatives have three strands yielded a trivial Groebner basis.  However, the ideal was non-trivial for three knots whose braid representatives have four strands: $10_{57}$, $10_{117}$, and $10_{162}$.  We obtained the example $\f$ by taking the partial closure of the braid representative of the $10_{57}$ knot.  We chose $10_{57}$ because its braid representative has several twist regions which make the computation in Appendex \ref{app:linearcombo} slightly easier.

Now, a natural question is whether there exists a genus-1 tangle which is the partial closure of a braid on three strands (or more generally, a $(B^3, 6)$-tangle) with a trivial Kauffman bracket ideal but does not embed in the unknot.  Because our search resulted in no non-trivial examples for $k=3$, we must find another way to detect such an example.


\section{The Kauffman bracket and trivalent graphs}\label{section:graphbasis}

\subsection{Kauffman bracket skein modules}\label{subsection:kbsm}
First we recall the definition of the Kauffman bracket of a link diagram $D$.  The Kauffman bracket $\langle D \rangle$ of a link diagram $D$ is a polynomial in $\laurent$ given by the following relations, where $\delta = -A^{2} -A^{-2}$: 
\begin{enumerate}[(i)]
\item $\displaystyle \langle\begin{minipage}{.5in}\begin{center}\includegraphics[width=.4in]{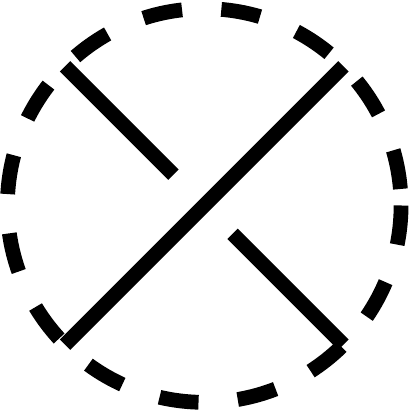}\end{center}\end{minipage} \rangle= A \langle\begin{minipage}{.5in}\begin{center}\includegraphics[width=.4in]{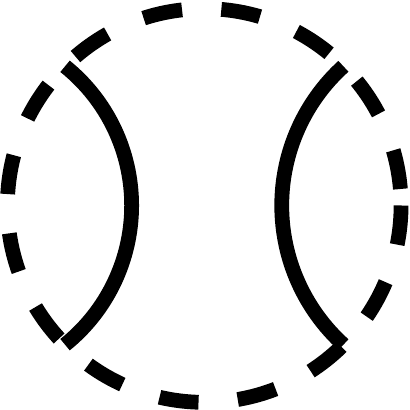}\end{center}\end{minipage} \rangle+ A^{-1}\langle\begin{minipage}{.5in}\begin{center}\includegraphics[width=.4in]{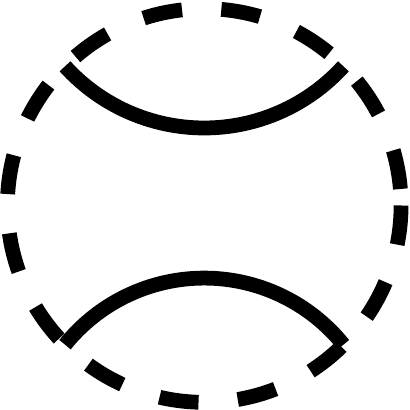}\end{center}\end{minipage} \rangle   $
\item $ \langle D^\prime \coprod \begin{minipage}{.4in}\begin{center}\includegraphics[width=.3in]{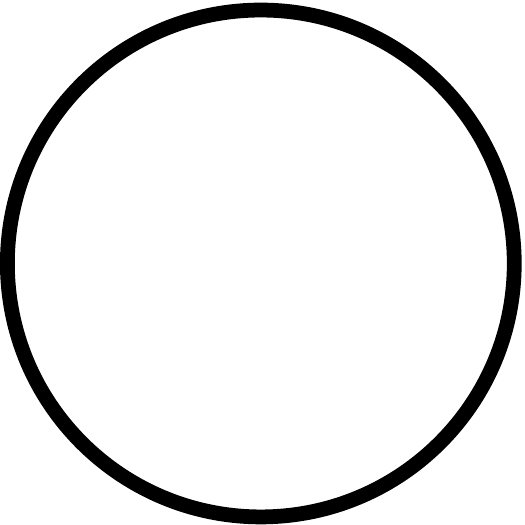}\end{center}\end{minipage} \rangle= \delta \langle D^\prime\rangle.$
\item $ \langle \begin{minipage}{.1in}\begin{center}\end{center}\end{minipage} \rangle= 1.$
\end{enumerate}
Furthermore, for any non-empty link $L$ and diagram $D$ of $L$, we define $\langle D \rangle ^\prime = \langle D\rangle/ \delta$ to be the reduced Kauffman bracket polynomial of $L$.

The Kauffman bracket skein module of a 3-manifold $M$, denoted by $K(M)$, is the $\laurent$-module generated by isotopy classes of framed links in $M$ modulo the Kauffman bracket relations.  Note that the isotopy class of the empty link is the identity in $K(M)$.

Given a 3-manifold $M$ with boundary and a set of $m$ framed points in $\partial M$, the relative Kauffman bracket skein module of $M$, denoted $K(M,m)$, is the $\laurent$-module generated by isotopy classes of framed links and arcs in $M$ which intersect $\partial M$ in the framed points.

Let $R$ denote $\laurent$ localized by inverting the quantum integers, $[k] =(A^{2n}-A^{-2n})/(A^2-A^{-2})$.  In addition to $K(M,m)$, we consider $K_R(M, m)$ the relative Kauffman bracket skein module of $M$ with coefficients in $R$.  When we refer to a skein element, we mean an element of $K_R(M,m)$.

We must make this distinction because when we compute the Kauffman bracket ideal of an $(M,2n)$-tangle, we are in fact using elements of and pairings defined on $K_R(M,2n)$ rather than $K(M,2n)$.  Since each 3-manifold $M$ we consider in this paper has the form $\Sigma \times I$ for some surface $\Sigma$, we have that $K_R(M,2n)$ is free on diagrams without crossings or contractable loops according to \cite[Theorem 3.1]{pr}.  Furthermore, according to \cite[Proposition 2.2]{pr}, we have that $K_R(M,2n) = K(M,2n) \otimes R$.  So we may view $K(M,2n)$ as a subset of $K_R(M,2n)$.


\subsection{Banded colored trivalent graphs}\label{subsection:graphs}

Recall that for each $n > 0$, the $n$th Temperley-Lieb algebra $TL_n = K_R(D^2 \times I, 2n)$ contains the $n$th Jones-Wenzl idempotent $f_n$ defined recursively as in Fig. \ref{idempdef}.  Here, $\Delta_n$ denotes the $n$th Chebyshev polynomial.  A small rectangle on an arc labelled $n$ represents the idempotent $f_n$.  For the rest of the paper, we drop the rectangles, and any arc labelled $n$ represents $n$ strands colored by $f_n$. 

\begin{figure}[h]
$\begin{minipage}{.3in}\includegraphics[height=.8in]{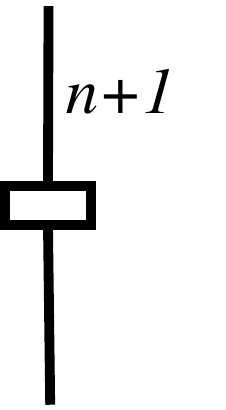}\end{minipage} = \hspace{.15in}\begin{minipage}{.4in}\includegraphics[height=.8in]{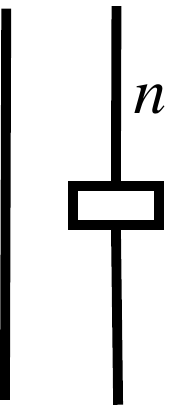}\end{minipage} - \ds\frac{\Delta_{n-1}}{\Delta_n} \hspace{.1in}\begin{minipage}{.4in}\includegraphics[height=.8in]{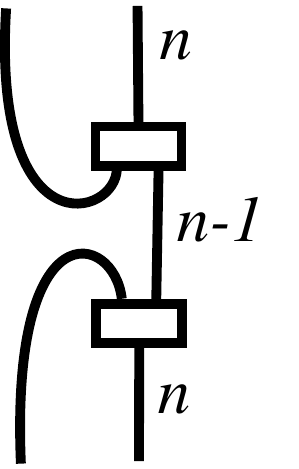}\end{minipage}$\\

\vspace{.2in}

where $\Delta_n = \hspace{.1in}\begin{minipage}{.9in}\includegraphics[height=.7in]{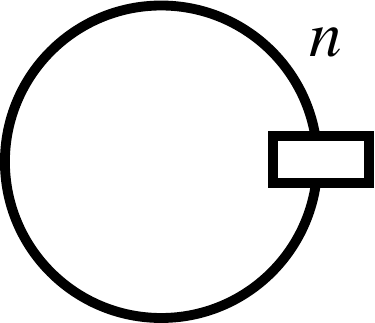}\end{minipage}$.
 \caption{Definition of the Jones-Wenzl idempotents.}\label{idempdef}
\end{figure}

A banded colored trivalent graph in a 3-manifold $M$ is a framed trivalent graph equipped with a cyclic orientation of the edges incident to each vertex.  The framing is given at the vertices by viewing each vertex as a disk with three bands attached (one for each edge).  Away from the vertices, the framing is simply the blackboard framing.

Additionally, each edge is colored by a non-negative integer $n$ which indicates the presence of the $n$th Jones-Wenzl idempotent.  For the rest of this paper, any unlabelled edge is assumed to be colored one.  At each vertex, the colors of the incident edges must form an admissible triple where admissibility is defined as follows.

\begin{definition}
 For non-negative integers $a$, $b$, and $c$, if $|a-b| \leq c\leq a + b$ and $a + b + c \equiv 0 (\text{mod } 2)$, then the triple $(a,b,c)$ is said to be admissible.
\end{definition}  In fact, such a vertex actually represents a linear combination of skein elements as in Fig. \ref{trivalentvertex}.   The inner colors, $i$, $j$, and $k$, must satisfy the following conditions: $i + j = a$, $i + k = b$, $j + k = c$.  For a more detailed treatment of the topic of banded colored trivalent graphs see \cite{kl,mv}.

\begin{figure}
$\begin{minipage}{1in}\includegraphics[width=1in]{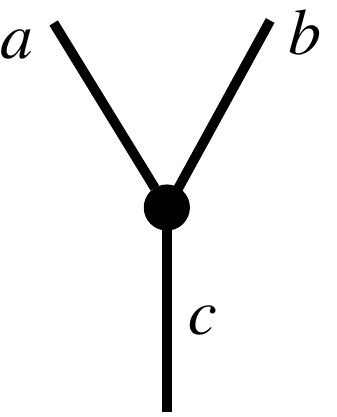}\end{minipage}\hspace{.25in} = \hspace{.25in} \begin{minipage}{1in}\includegraphics[width=1in]{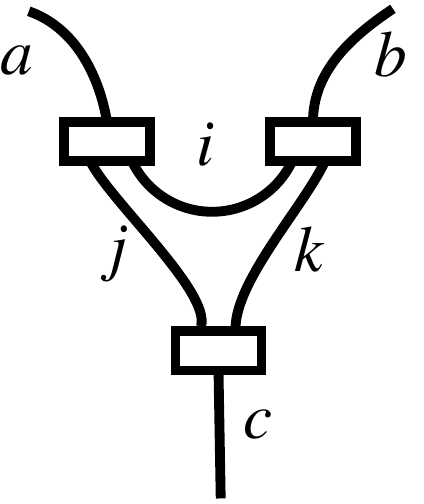}\end{minipage}$
 \caption{A trivalent vertex.}\label{trivalentvertex}
\end{figure}

We use the same notation as in \cite{gh} for the evalutations of two banded colored trivalent graphs that appear frequently:
 \[\begin{minipage}{.8in}\includegraphics[width=.8in]{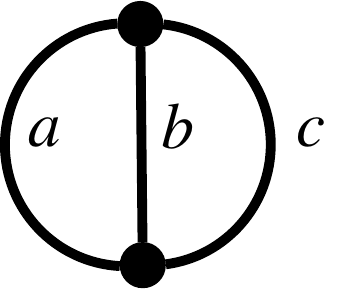}\end{minipage} = \theta (a,b,c) \text { and }\hspace{.1in} \begin{minipage}{1.1in}\includegraphics[width=1.1in]{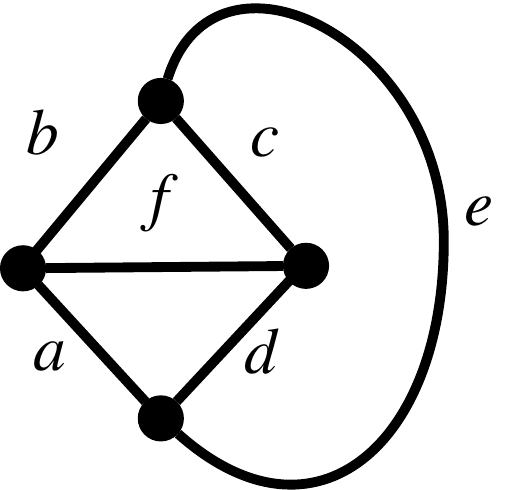}\end{minipage} = Tet \left[  \begin{array} {c c c} a & b & e\\ c & d & f \\ \end{array}\right].\]

 We use the following formulas when computing the Kauffman bracket ideal of a genus-1 tangle.  For details, see \cite{kl,mv}.

\begin{equation}\label{eq:bubble}
 \begin{minipage}{.7in}\includegraphics[height=1in]{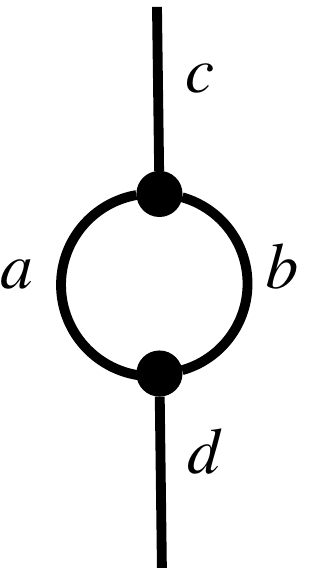}\end{minipage}=  \hspace{.1in} \delta^c_d \ds\frac{\theta(a,b,c)}{\Delta_c} \hspace{.25in}\begin{minipage}{1in}\includegraphics[height=1in]{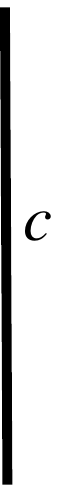}\end{minipage}
 \end{equation}

\begin{equation}\label{eq:lambdatwist}
\begin{minipage}{.7in}\includegraphics[height=.8in]{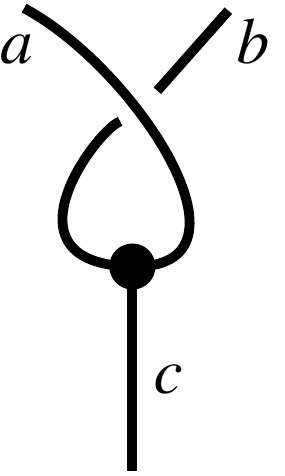}\end{minipage}=  \hspace{.1in} \lambda^{a \text{ }b}_c \hspace{.15in}\begin{minipage}{.8in}\includegraphics[height=.8in]{trivalentvertexbefore-eps-converted-to.pdf}\end{minipage}\text{ where }\lambda^{a \text{ }b}_c \text{ is as given in \cite{kl}}
 \end{equation}

\begin{equation}\label{eq:tetreduce}
\begin{minipage}{.7in}\includegraphics[height=.8in]{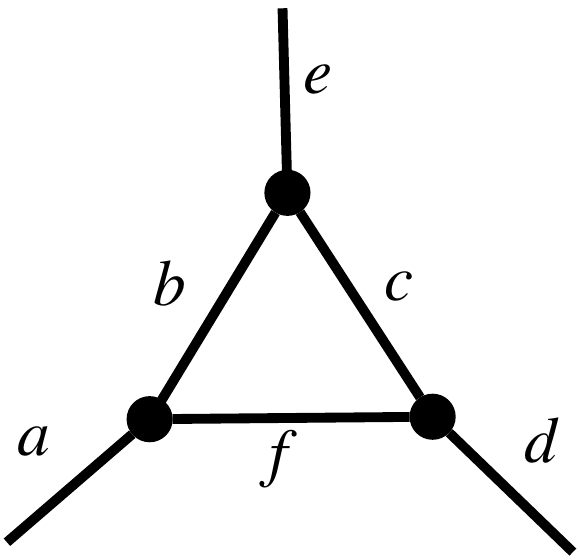}\end{minipage}=  \hspace{.1in}
 \left\{ \begin{array}{l l}
 \ds\frac{Tet \left[  \begin{array} {c c c} a & b & e\\ c & d & f \\ \end{array}\right]}{\theta(a,d,e)}\begin{minipage}{.9in}\includegraphics[height=.8in]{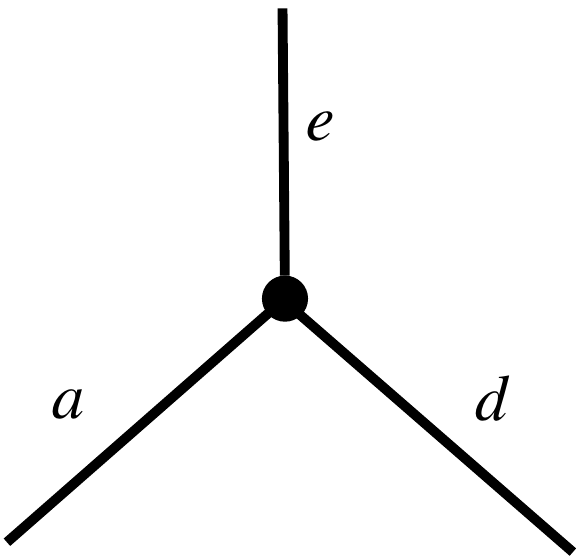}\end{minipage}, & \text{ if } (a,d,e) \text{ is admissible} \\
 0, & \text{otherwise}
 \end{array}\right.
  \end{equation}

\begin{equation}\label{eq:6j}
\begin{minipage}{1in}\includegraphics[height=.6in]{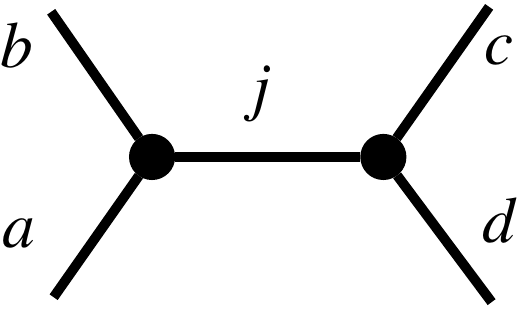}\end{minipage}=  \hspace{.1in} \ds\sum_i \left\{  \begin{array} {c c c} a & b & i\\ c & d & j \\ \end{array}\right\} \begin{minipage}{.6in}\includegraphics[height=1in]{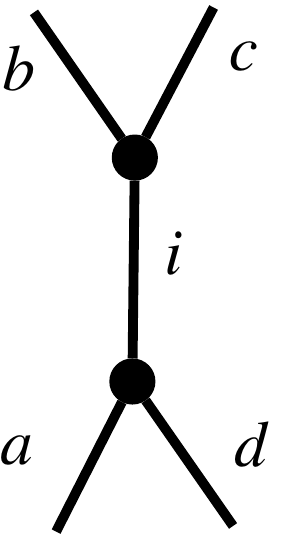}\end{minipage},
 \end{equation}
 where the sum is over all admissible values $i$  and $\left\{  \begin{array} {c c c} a & b & i\\ c & d & j \\ \end{array}\right\} = \ds\frac{Tet \left[  \begin{array} {c c c} a & b & i\\ c & d & j \\ \end{array}\right] \Delta_i}{\theta(a,d,i)\theta(b,c,i)}$

\begin{equation} \label{eq:removeloop}
\begin{minipage}{.6in}\includegraphics[height=.8in]{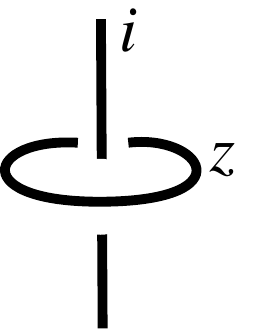}\end{minipage} =  \phi_i\text{   }\begin{minipage}{.3in}\includegraphics[height=.75in]{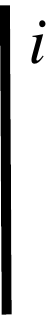}\end{minipage}\text{ where }\phi_i = -A^{2i+2}-A^{-2i-2}
 \end{equation}

\begin{theorem}[Fusion Formula]\label{thm:fusion}
$$\begin{minipage}{.5in}\includegraphics[height=.75in]{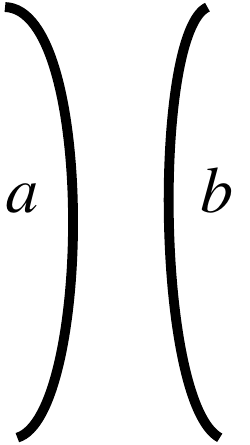}\end{minipage}=  \hspace{.1in} \ds\sum_i \frac{\Delta_i}{\theta(a,b,i)} \hspace{.1in}\begin{minipage}{1in}\includegraphics[height=.85in]{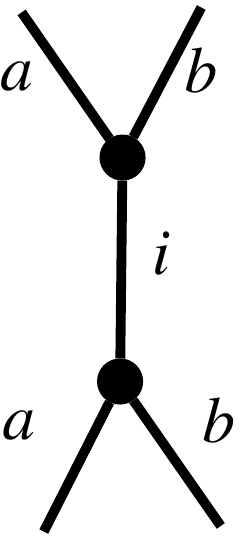}\end{minipage}$$
where the sum is over all $i$ such that $(a,b,i)$ is admissible.
\end{theorem}

For details of the following theorem, see \cite{gh}.

\begin{theorem}\label{thm:gh}
 If a sphere intersects a skein element in exactly 2 labelled arcs, then
 $$\begin{minipage}{.65in}\includegraphics[height=.75in]{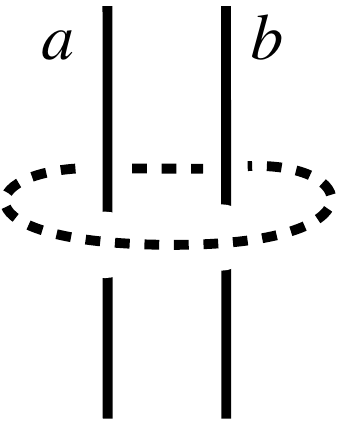}\end{minipage}=  \hspace{.1in} \ds\frac{\delta^a_b}{\Delta_a} \hspace{.1in}\begin{minipage}{.75in}\includegraphics[height=.85in]{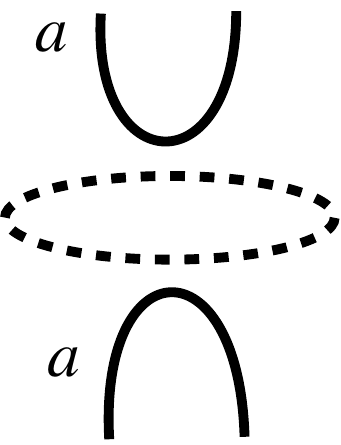}\end{minipage}.$$

If a sphere intersects a skein element in exactly 3 labelled arcs, then
 $$\begin{minipage}{.9in}\includegraphics[height=.8in]{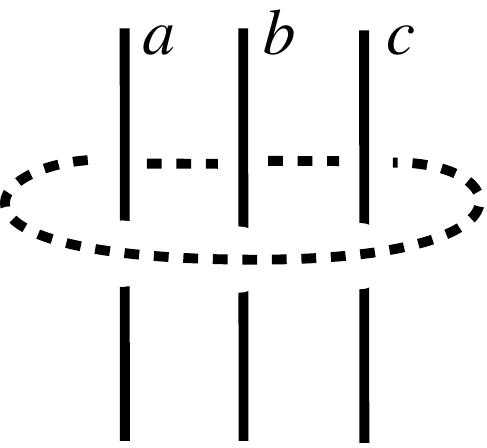}\end{minipage}=  \hspace{.1in}
  \left\{ \begin{array}{l l}
 \ds\frac{1}{\theta(a,b,c)}   \hspace{.1in} \begin{minipage}{.65in}\includegraphics[height=.9in]{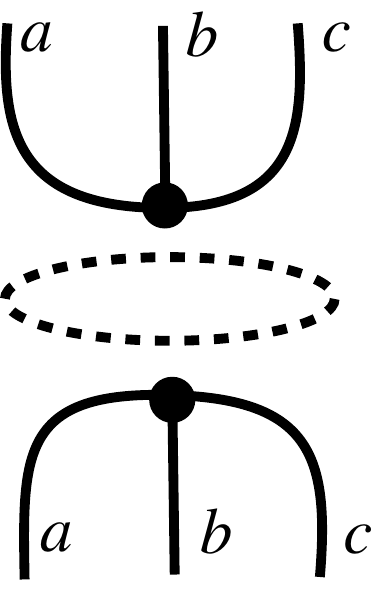}\end{minipage}, & \text{ if } (a,b,c) \text{ admissible} \\
 0, & \text{otherwise.}
 \end{array}\right.$$

If a sphere intersects a skein element in exactly $n > 3$ arcs, then
$$\begin{minipage}{1.3in}\includegraphics[height=.9in]{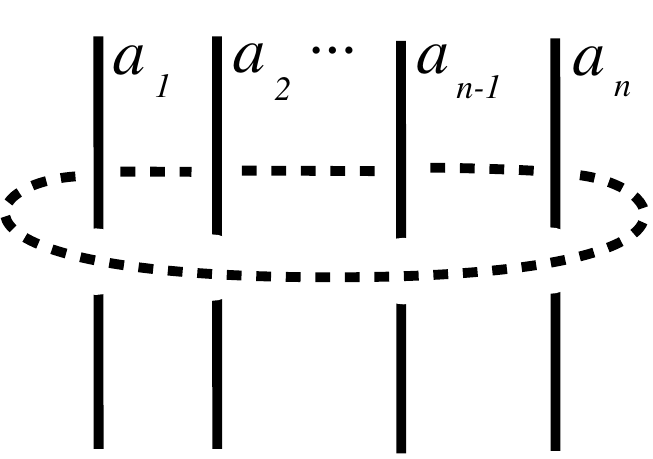}\end{minipage}=  \ds\sum \frac{1}{\begin{minipage}{1.3in}\includegraphics[height=.9in]{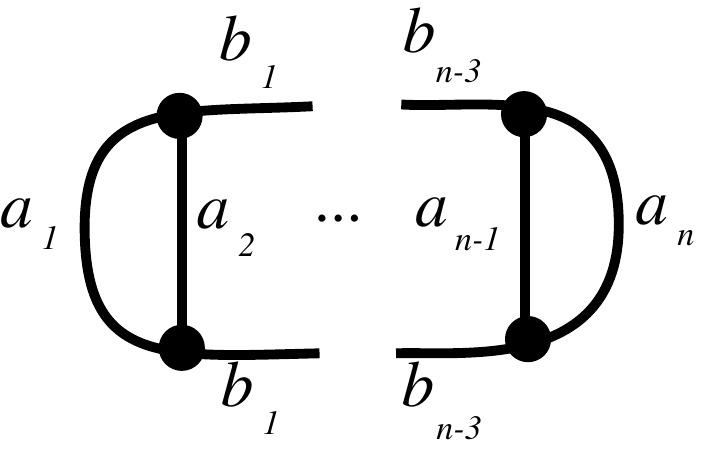}\end{minipage}} \hspace{.25in} \begin{minipage}{1.35in}\includegraphics[height=1.3in]{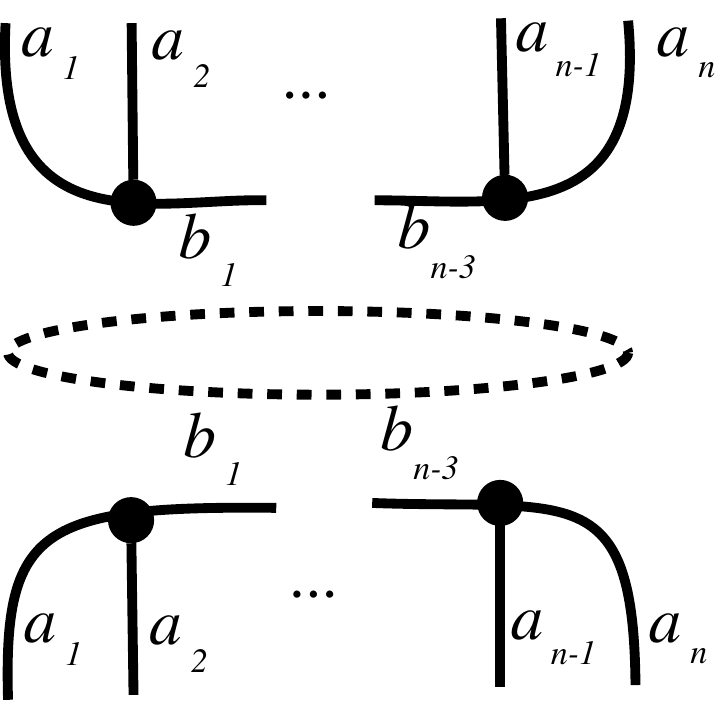}\end{minipage},$$
where the sum is over all admissible labellings.
\end{theorem}


\subsection{Defining the graph basis of $K_R(\storus,2)$}\label{subsection:graphbasis}

Given a pair of non-negative integers $(i,\vareps)$, let \[g_{i,\vareps} =\hspace{.15in}\begin{minipage}{1in}\includegraphics[height=.7in]{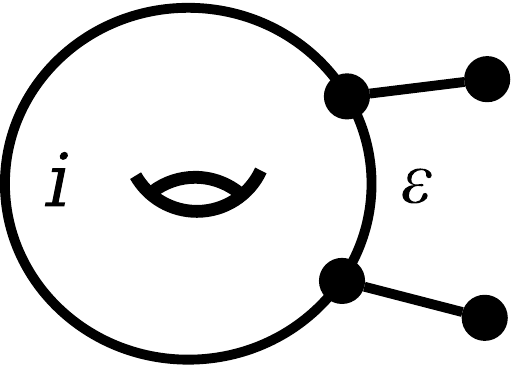}\end{minipage}.\]  Note that this definition implies that the triple $(1,i,\vareps)$ must be admissible.  Therefore, either $\vareps = i+1$ or $\vareps = i-1$.

Since we can write any skein element as a linear combination of these $g_{i,\vareps}$'s using the fusion formula from Theorem \ref{thm:fusion} and Formulas \ref{eq:bubble} - \ref{eq:removeloop}, we have that the $g_{i,\vareps}$'s form a generating set for $K_R (\storus, 2)$.

Making use of work of Hoste-Przytycki, we see that $K_R(S^1 \times S^2)/\text{torsion} = R$ via an isomorphism which sends the empty link to one (see \cite[Theorem 2.3 (d)]{pr}).  We define a pairing $\langle \text{ , }\rangle_D: K_R (\storus, 2) \times K_R (\storus, 2) \rightarrow K_R(S^1 \times S^2)/\text{torsion} = R$ as follows.  First, perform a radial twist on the second solid torus (see Fig. \ref{fig:radialtwist}), then identify the boundaries of the two solid tori via an orientation-reversing homeomorphism to obtain $S^1 \times S^2$.  We view the result as an element of $K_R(S^1 \times S^2)/\text{torsion}$. This pairing is symmetric and can be represented pictorially on the $g_{i,\vareps}$'s as follows, where the loop denotes a 0-surgery:
$$\langle g_{i,\vareps},g_{i^\prime,\vareps^\prime} \rangle_D = \begin{minipage}{1in}\includegraphics[height=.8in]{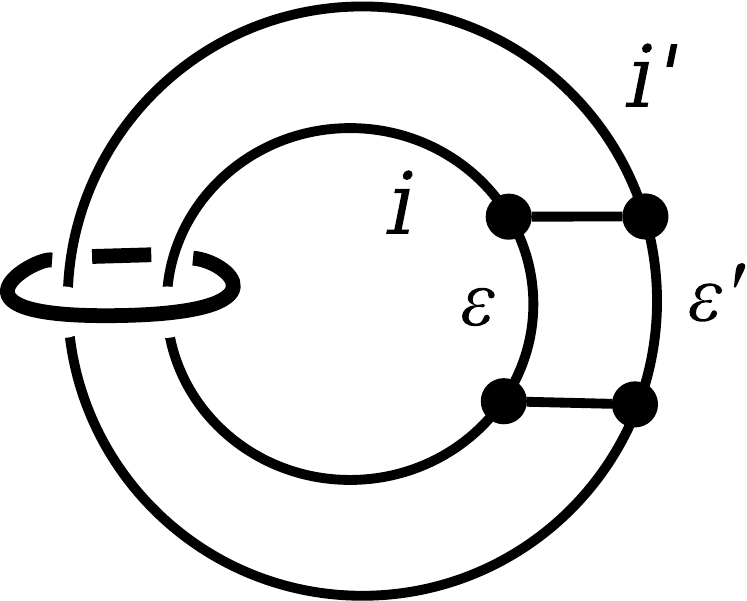}\end{minipage}.$$

\begin{figure}
\includegraphics[height=.85in]{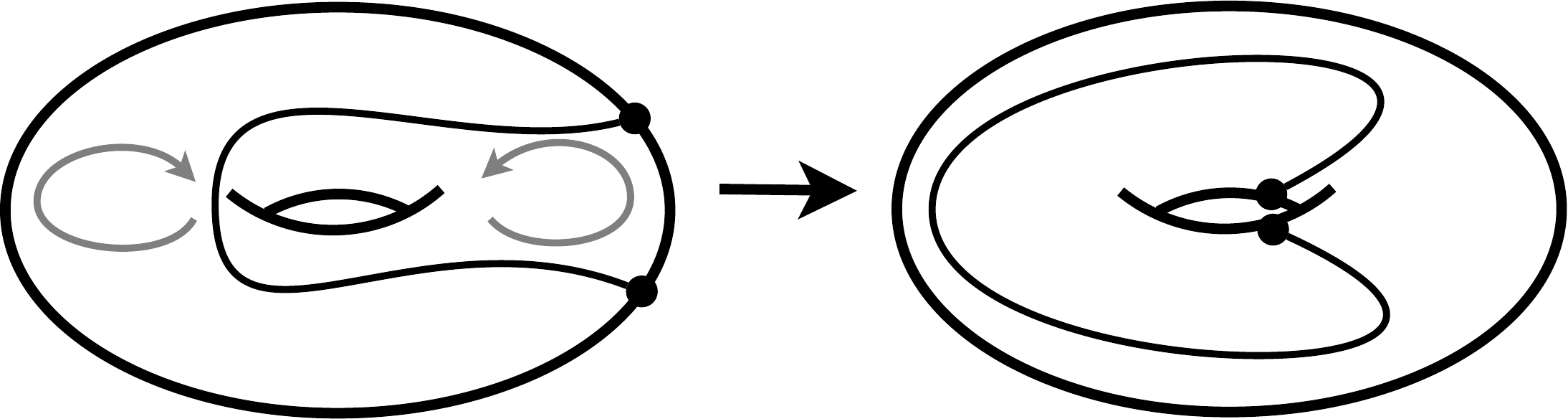}
\caption{A homeomorphism given by twisting the solid torus.}\label{fig:radialtwist}
\end{figure}

We call this the doubling pairing.  We have the following theorem which shows that the $g_{i,\vareps}$'s are orthogonal with respect to the doubling pairing and are therefore linearly independent.  So they form a basis for $K_R (\storus, 2)$.

\begin{theorem}\label{thm:doublingpairing}
We have that \\$\ds\langle g_{i,\vareps}, g_{i^\prime,\vareps^\prime} \rangle_D =
\left\{\begin{array}{l l}
\ds\frac{\theta(1,i,\vareps)^2}{\Delta_i \Delta_\vareps}, & (i,\vareps) = (i^\prime,\vareps^\prime)\\
0, & \text{otherwise}\\
\end{array}\right. $

\end{theorem}

\begin{proof}
According to Theorem \ref{thm:gh} and Formula \ref{eq:bubble}, we have

$$\begin{array}{r c l}
\langle g_{i,\vareps}, g_{i^\prime,\vareps^\prime} \rangle_D & = & \begin{minipage}{1.3in}\includegraphics[height=1in]{graphbasispairing1-eps-converted-to.pdf}\end{minipage}  =  \ds\frac{\delta^i_{i^\prime}}{\Delta_i} \hspace{.1in} \begin{minipage}{.8in}\includegraphics[height=.8in]{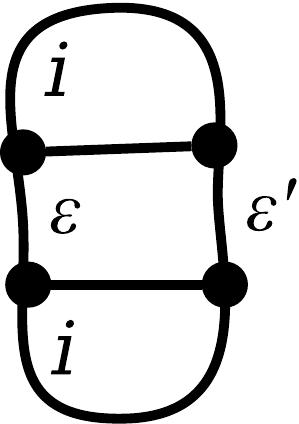}\end{minipage} \\
& = & \ds\frac{\delta^i_{i^\prime}\delta^\vareps_{\vareps^\prime}\theta(1,i,\vareps)}{\Delta_i \Delta_\vareps} \hspace{.1in} \begin{minipage}{.6in}\includegraphics[height=.6in]{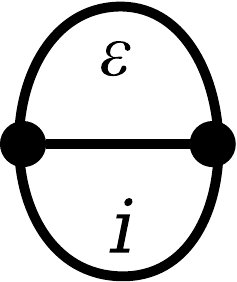}\end{minipage}\\
& = & \ds\frac{\delta^i_{i^\prime}\delta^\vareps_{\vareps^\prime}\theta(1,i,\vareps)^2}{\Delta_i \Delta_\vareps}.
\end{array}$$

\end{proof}


\section{The almost-orthogonal basis}\label{section:almostorthbasis}

Recall, several bases for $\K$ are defined in \cite{bhmv92}.  Let $e_i$ be a non-contractable loop in $\K$ colored $i$.  The set of all such elements is a basis for $\K$.  In particular, $e_1$ is also denoted by $z$.  The set $\{1, z, z^2, \ldots\}$ also forms a basis for $\K$, and furthermore, $\K = \laurent[z]$ as an algebra.  Finally, $\{Q_n\}$ for $n > 0$ is a basis where $Q_n = (z - \phi_0)(z-\phi_1)\ldots(z-\phi_{n-1})$ and $\phi_n = -A^{2n+2}-A^{-2n-2}$.  Each of these bases is related to the others by a unimodular triangular basis change.
 
Recall the Hopf pairing on $\K$ defined in \cite{bhmv92}.  Choose an orientation-preserving embedding of two disjoint solid tori into $S^3$ such that each of the standard bands is sent to one component of the banded Hopf link where each component has writhe zero.  Then let $\langle \text{ , } \rangle$ be given by the induced map $\K \times \K \rightarrow K(S^3) = \laurent$.  Note that $K(S^3)$ is isomorphic to $\laurent$ via the isomorphism which sends the empty link to one.  We express this pairing pictorially as follows:  $$\langle a, b \rangle =  \left\langle\hspace{.05in}\begin{minipage}{.5in}\includegraphics[width=.5in]{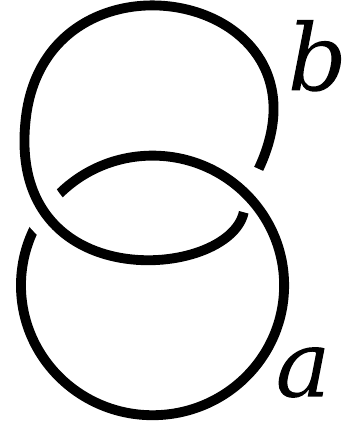}\end{minipage}\right\rangle.$$ 

\begin{lemma}[BHMV]\label{Qorth}
The set $\{Q_n\}$ is a basis for $\K$ which is orthogonal with respect to the bilinear form $\langle\text{ , }\rangle$.  
\end{lemma}

The following formula for $\langle Q_n, Q_n \rangle$ is stated in \cite[Section 2]{gm2}.

\begin{lemma}
For all $n\geq 0$, we have $\langle Q_n,Q_n\rangle = \Delta_n\displaystyle \prod _{i=0}^{n-1}(\phi_n - \phi_i)$.
\end{lemma}


By adapting the definition of the Hopf pairing, we define an analagous pairing $\langle\text{ , }\rangle:K_R(S^1\times D^2, 2) \times K_R(S^1\times D^2, 2)\rightarrow K_R(S^3) = R$ as follows:   $$\langle a, b \rangle =  \left\langle\hskip.05in\begin{minipage}{1in}\includegraphics[width=1in]{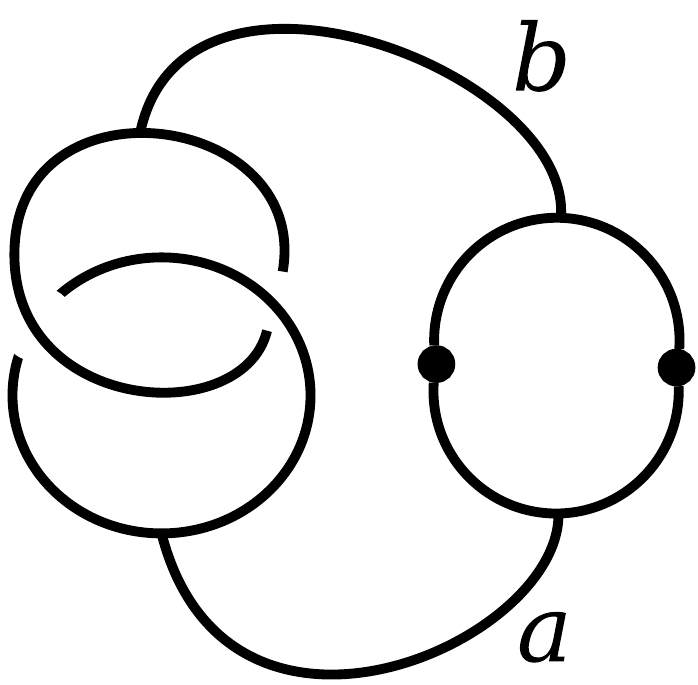}\end{minipage}\hspace{.04in}\right\rangle$$ where $a$ and $b$ are skein elements lying in a regular neighborhood of the trivalent graphs pictured.  Again, we note that $K_R(S^3)$ is isomorphic to $R$ via the isomorphism that sends the empty link to one.  We call this the relative Hopf pairing.  We use the same notation for this pairing as for the Hopf pairing, but the context should make it clear which pairing is being used.

It is easy to see that this pairing is a symmetric bilinear form on $K_R(S^1\times D^2, 2)$.  Furthermore, this pairing restricted to $K(S^1\times D^2, 2)\times K(S^1\times D^2, 2)$ is a symmetric bilinear form which takes values in $K(S^3) = \laurent$.  For simplicity, we use the same notation for the restricted pairing.

We use the basis $\{Q_n\}$ to define a basis for $K(\storus, 2)$.  For $n \geq 0$, let  $$x_n =  \begin{minipage}{1in}\includegraphics[width=1in]{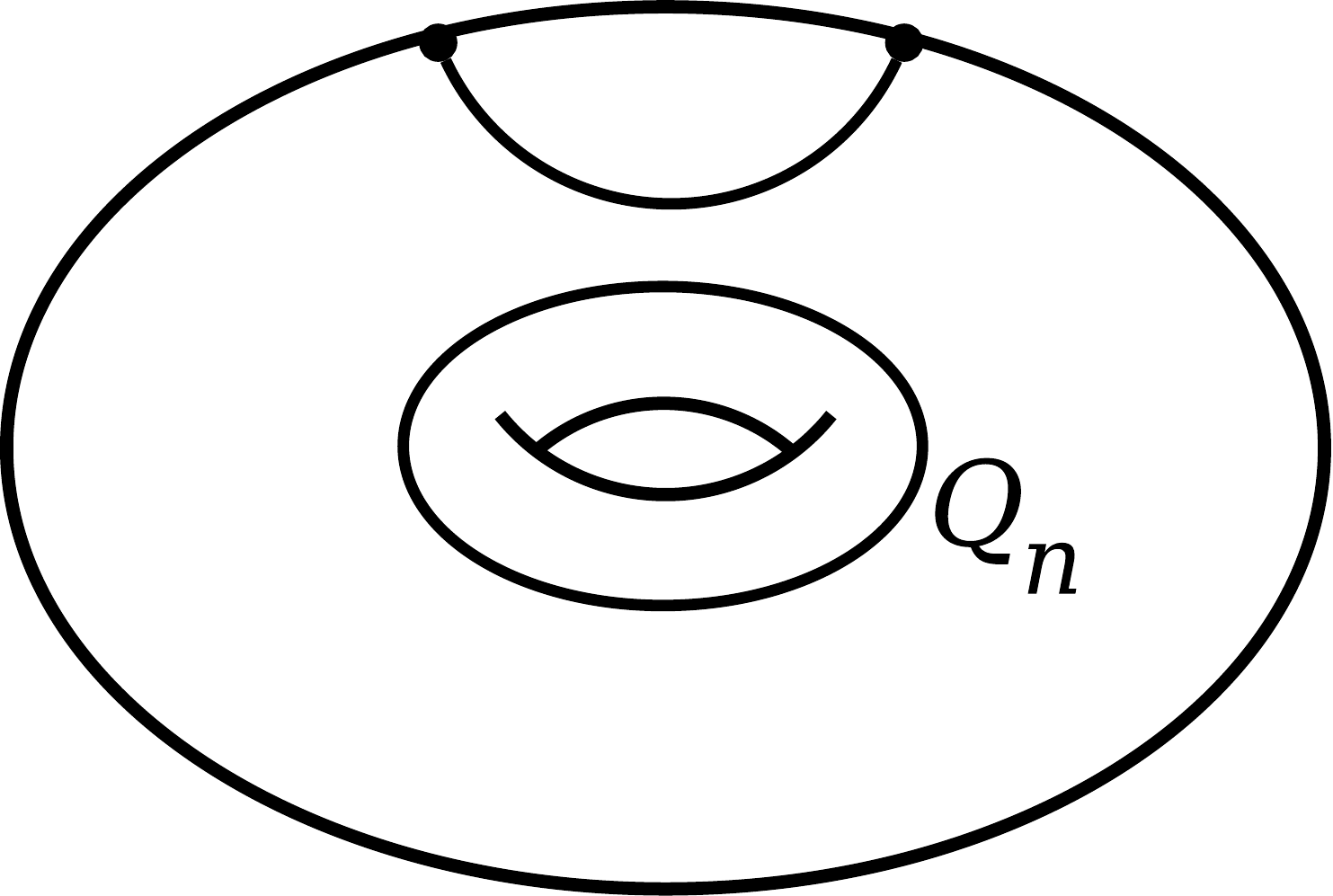}\end{minipage}\text{ and }y_n =  \begin{minipage}{1in}\includegraphics[width=1in]{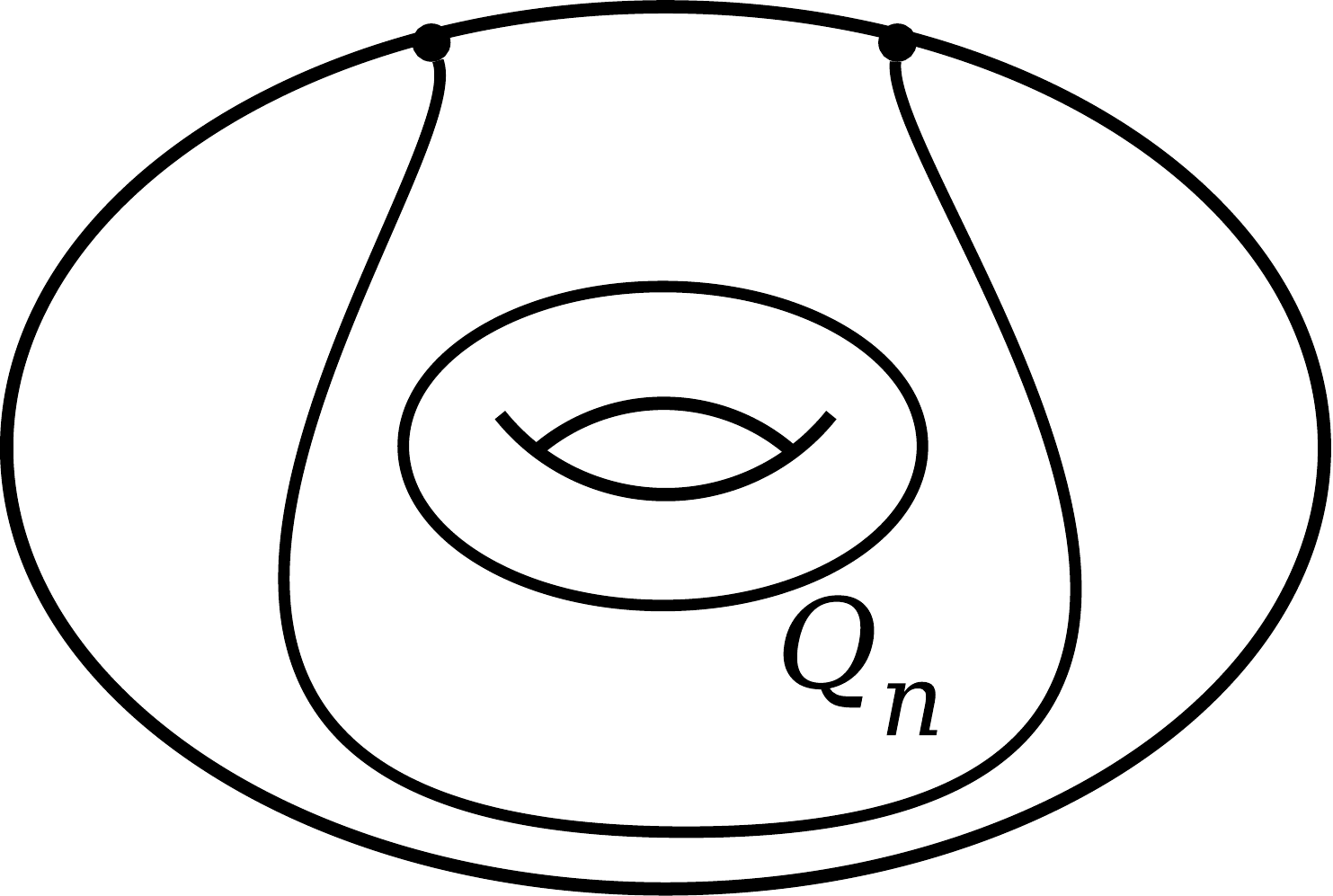}\end{minipage}.$$  Since $\{Q_n\}$ is a basis for $\K$, it is not hard to see that $\{x_n,y_n\}$ is a basis for $\Krel$.

Before discussing how this basis relates to the relative Hopf pairing, we must define a map on $K(S^1 \times D^2)$.  We let $\bar\tau$ denote the mirror image of the map $\tau$ from \cite{bhmv92}.  So, $\bar\tau (u)$ for $u \in K(\storus)$ is given by adding a single loop as in Fig. \ref{taubar}.  The next two lemmas follow directly from \cite[Lemmas 3.2 and 4.9]{bhmv92} respectively.

\begin{figure}
\includegraphics[height=1in]{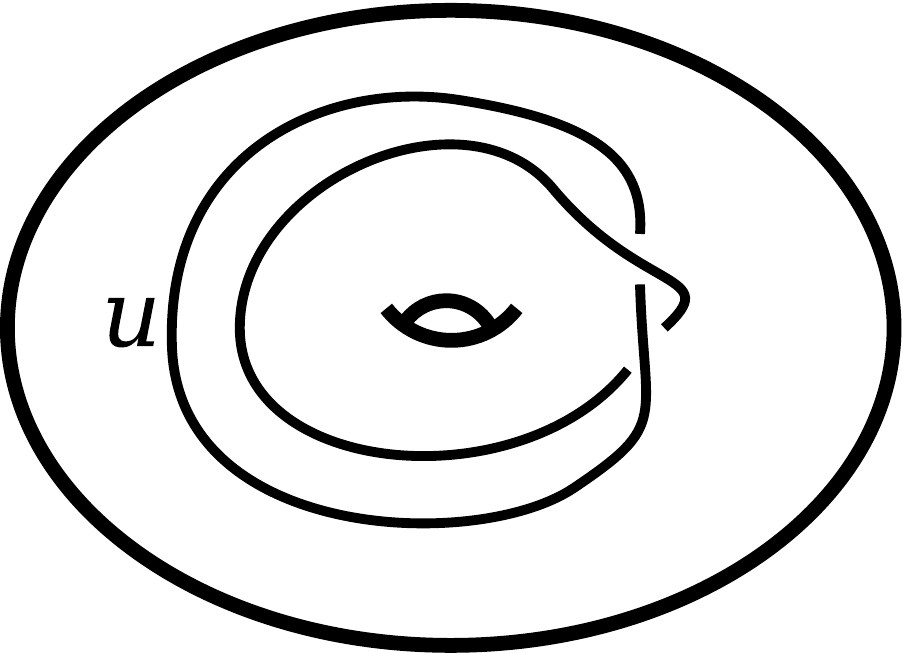}
\caption{$\bar\tau(u)$}\label{taubar}
\end{figure}

\begin{lemma}\label{taubarz}
$\bar\tau(Q_{n-1}) = A^{-2n+2}Q_n + \ldots$, where the dots indicate lower order terms; that is, terms in which the index of each $Q_i$ appearing is at most $n-1$.
\end{lemma}

\begin{lemma}\label{taubarQQ}
One has that $\langle\bar\tau Q_n,Q_n\rangle = (A^{-2n}\sigma_n-A^{-2n+2}\sigma_{n-1})\langle Q_n, Q_n\rangle$ for all $n\geq 0$ where $\sigma_n = \sum_{i=0}^n \phi_i$.
\end{lemma}

Using these results, we prove the following lemma which states that the basis $\{x_n,y_n\}$ is almost orthogonal with respect to the relative Hopf pairing.  In view of this lemma, we refer to $\{x_n,y_n\}$ as the almost orthogonal basis.


\begin{lemma}\label{basisorth}  We have the following formulas for pairings of elements of $\{x_n,y_n\}$:\smallskip
\begin{enumerate}[(i)]

\item $\displaystyle \langle x_m, x_n\rangle = \left\{
\begin{array}{ll}
\delta\langle Q_m,Q_m\rangle, &\text{ if } m=n\\
0, & \text{otherwise}.
\end{array}
\right.$\medskip

\item $\displaystyle \langle x_m, y_n\rangle = \left\{
\begin{array}{ll}
\langle Q_m, Q_m \rangle, & \text{ if }m=n+1\\
 \phi_m \langle Q_m, Q_m \rangle,& \text{ if }m = n\\
 0, & \text{otherwise}.
\end{array}
\right.$\medskip

\item $\displaystyle \langle y_m, y_n\rangle = \left\{
\begin{array}{ll}
A^{-2k-4}\langle Q_k, Q_k \rangle \text{ where }k=\max\{m,n\}, & \text{ if }|n-m|=1\\
A^{-6}(A^{-2m}\sigma_m - A^{-2m+2}\sigma_{m-1})\langle Q_m, Q_m \rangle,& \text{ if }m = n\\
 0, & \text{otherwise}.
\end{array}
\right.$
\end{enumerate}

\end{lemma} 

\begin{proof}
\begin{enumerate}[(i)]
\item We first consider pairing two $x$-elements together. Since the $Q_i$'s form an orthogonal basis according to Lemma \ref{Qorth}, we have 

$ \displaystyle\langle x_m,x_n \rangle= 
 \hskip.1in
 \begin{minipage}{1in}\includegraphics[width=1in]{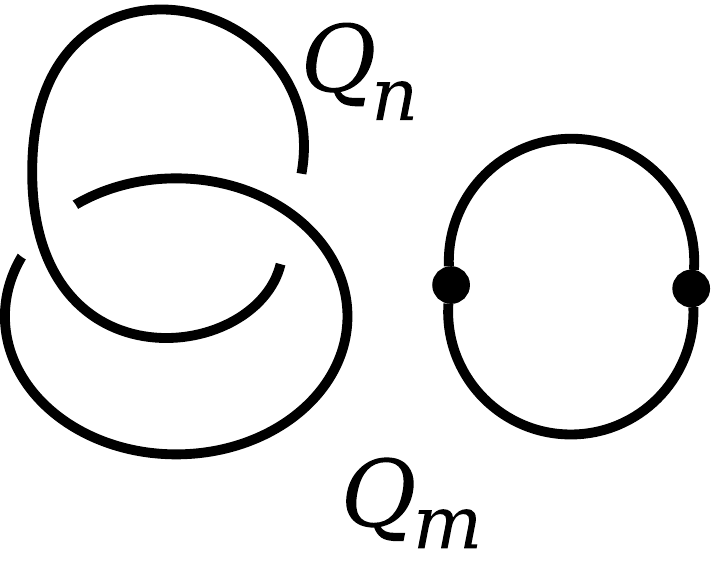}\end{minipage} = \delta\langle Q_m,Q_n\rangle = \left\{\begin{array}{ll}
\delta\langle Q_m,Q_m\rangle, & \text{ if }m=n\\
0, & \text{otherwise}\hspace{.1in}.
\end{array}\right.
$

\item When pairing $x_m$ with $y_n$, we see that
 $$ \displaystyle\langle x_m,y_n \rangle = 
 \hskip.1in
 \begin{minipage}{1in}\includegraphics[width=1in]{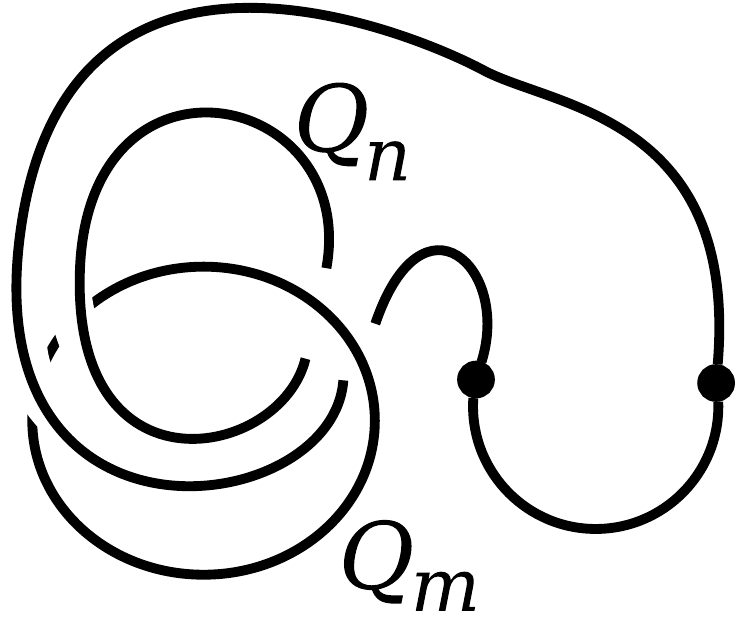}\end{minipage} = \langle Q_m,zQ_n\rangle.$$
Then, since $Q_{n+1} = (z-\phi_n)Q_n$, we have that
$$\langle Q_m,zQ_n\rangle = \langle Q_m, Q_{n+1}+\phi_{n}Q_n \rangle = \left\{
\begin{array}{ll}
\langle Q_{m}, Q_{m} \rangle, & \text{ if } m=n+1\\
 \phi_m \langle Q_m, Q_m \rangle,& \text{ if } m = n\\
 0, & \text{otherwise}\hspace{.1in}.
\end{array}
\right.
$$

\item Finally, we have that

$$\displaystyle \langle y_m,y_n\rangle = \hskip.1in
 \begin{minipage}{1in}\includegraphics[width=1in]{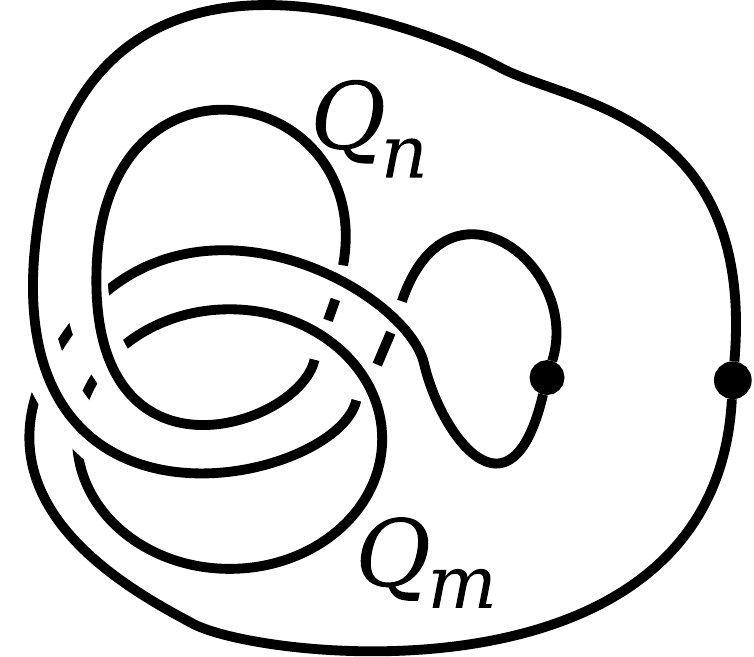}\end{minipage} = 
A^{-6}\hspace{.1in}\begin{minipage}{.7in}\includegraphics[width=.7in]{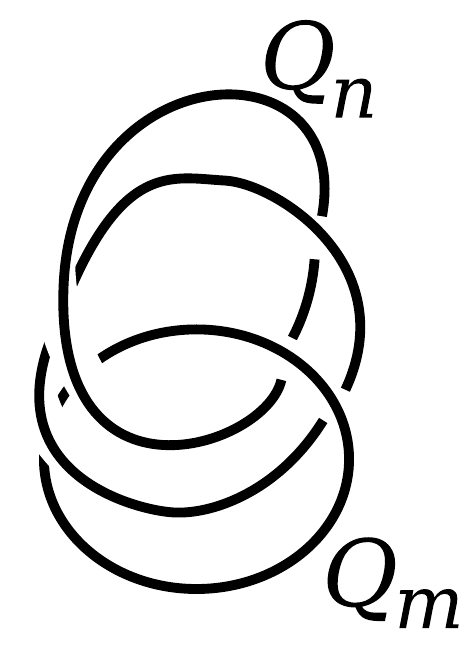}\end{minipage} = A^{-6}\langle Q_m,\bar\tau Q_n\rangle.$$

If $m=n$, then according to Lemma \ref{taubarQQ}, we have that $\langle y_m, y_n \rangle = A^{-6}\langle Q_m,\bar\tau Q_m \rangle = A^{-6}(A^{-2m}\sigma_m-A^{-2m+2}\sigma_{m-1})\langle Q_m, Q_m \rangle.$

Suppose $m = n + 1$.  Then Lemma \ref{taubarz} implies $ \langle y_m, y_n \rangle = A^{-6}\langle Q_m,\bar\tau Q_{m-1}\rangle = A^{-6}A^{-2m+2}\langle Q_m,Q_m\rangle$.

If $n = m + 1$, we have $\langle y_m,y_n\rangle = A^{-6}A^{-2n+2}\langle Q_n, Q_n \rangle$ since the relative Hopf pairing is symmetric.

Suppose $m \geq n + 2$.  Then $\langle y_m, y_n \rangle =A^{-6}\langle Q_m,\bar\tau Q_{n}\rangle = A^{-6}\langle Q_m,A^{-2n}Q_{n+1} + \ldots \rangle = 0$ since $m$ is greater than $n+1$ and the index of each lower order term.  Because the relative Hopf pairing is symmetric, we also have $\langle y_m,y_n \rangle = 0$ if $n \geq m + 2$.  In this way, (iii) follows.

\end{enumerate}

\end{proof}



\section{A finite set of generators for the Kauffman bracket ideal}\label{sec:fingen}

In this section, we outline an algorithm for computing a finite list of generators for the Kauffman bracket ideal $I_\G$ of a genus-1 tangle. However, we must first discuss how the graph basis and the almost-orthogonal basis relate to one another.

\begin{lemma}\label{lemma:removingq}
$$\begin{minipage}{.75in}\includegraphics[height=.8in]{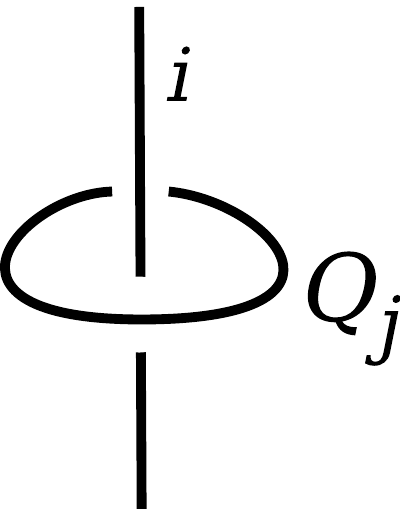}\end{minipage} = \ds\prod_{k=0}^{j-1}(\phi_i - \phi_k) \hspace{.04in}\begin{minipage}{.3in}\includegraphics[height=.65in]{strandcoloredi-eps-converted-to.pdf}\end{minipage}$$
where $\phi_n = -A^{2n+2}-A^{-2n-2}$.  If $j=0$, we let $\ds\prod_{k=0}^{j-1}(\phi_i - \phi_k)=1$.
\end{lemma}

\begin{proof}
From the definition of $Q_j$ and Formula \ref{eq:removeloop}, we see that
$$\begin{array}{r c c c l}
\begin{minipage}{.6in}\includegraphics[height=.8in]{removingq-eps-converted-to.pdf}\end{minipage} & = & \begin{minipage}{.55in}\includegraphics[height=1in]{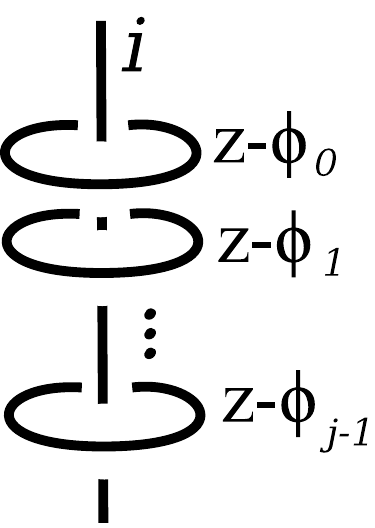}\end{minipage} & = &  \ds\prod_{k=0}^{j-1}(\phi_i - \phi_k) \hspace{.04in}\begin{minipage}{.15in}\includegraphics[height=.65in]{strandcoloredi-eps-converted-to.pdf}\end{minipage}.\\
\end{array}$$
\end{proof}

We can now compute the relative Hopf pairings of graph basis elements with almost-orthogonal basis elements, which we need to compute the generators of $I_\G$.

\begin{proposition}\label{prop:basespairings}
We have that $\langle g_{i,\vareps}, x_j\rangle = \ds\theta(1,\vareps,i)\prod_{k=0}^{j-1}(\phi_i - \phi_k)$ and $\langle g_{i,\vareps}, y_j \rangle = \ds\theta(1,\vareps,i)(\lambda^{1\text{ } i}_\vareps)^{-1}(\lambda^{i\text{ } 1}_\vareps)^{-1}\prod_{k=0}^{j-1}(\phi_i - \phi_k)$ for all non-negative $i$, $\vareps$, and $j$.  Again, if $j=0$, we let $\prod_{k=0}^{j-1}(\phi_i - \phi_k)=1$.  Note that if $j >i$, then both of these pairings are zero.
\end{proposition}

\begin{proof}
First, consider pairing a graph basis element with $x_j$.  From Lemma \ref{lemma:removingq}, we have that 
$$ \langle g_{i,\vareps}, x_j\rangle = \hspace{.05in}\begin{minipage}{1in}\includegraphics[height=.75in]{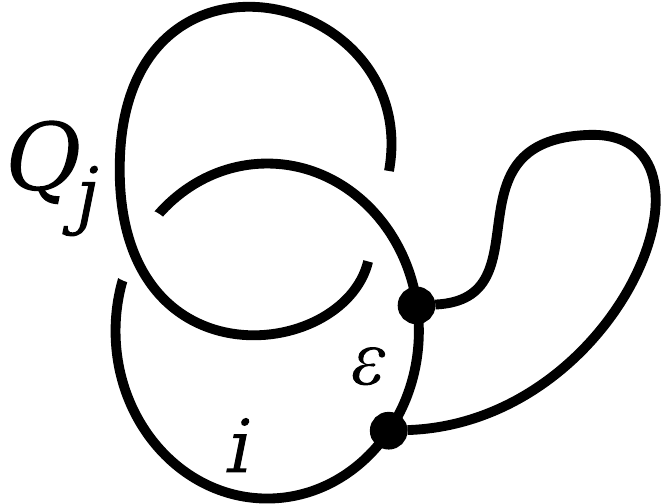}\end{minipage} = \ds\prod_{k=0}^{j-1}(\phi_i - \phi_k)\hspace{.05in} \begin{minipage}{.65in}\includegraphics[height=.5in]{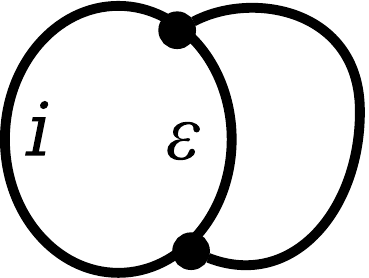}\end{minipage} =  \ds\theta(1,\vareps,i)\prod_{k=0}^{j-1}(\phi_i - \phi_k).$$

For the second case, we have from the Formula \ref{eq:lambdatwist} that

$$\begin{array}{r c  l}
\langle g_{i,\vareps}, y_j\rangle & = & \begin{minipage}{.75in}\includegraphics[height=.85in]{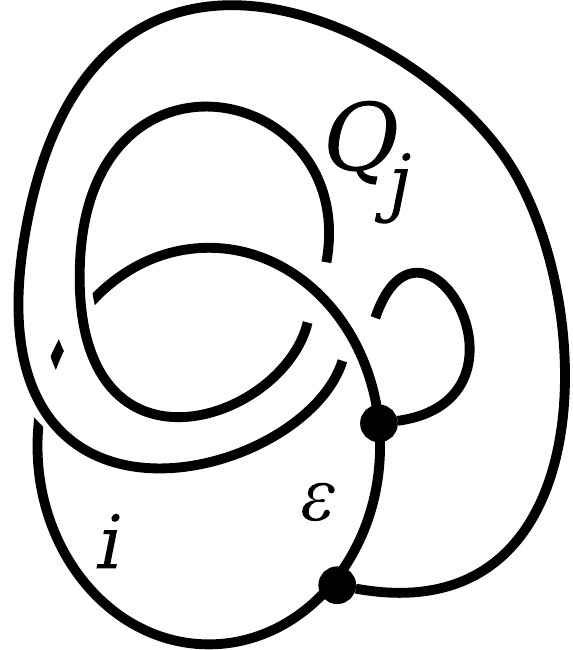}\end{minipage}  = \hspace{.05in}  \ds\prod_{k=0}^{j-1}(\phi_i - \phi_k)\hspace{.05in} \begin{minipage}{.75in}\includegraphics[height=.85in]{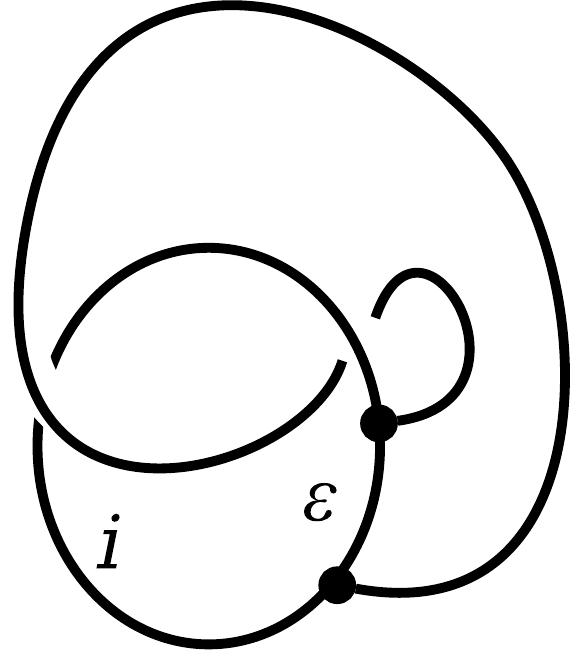}\end{minipage}\\
& & \\
& = & \ds(\lambda^{1\text{ } i}_\vareps)^{-1}(\lambda^{i\text{ } 1}_\vareps)^{-1}\prod_{k=0}^{j-1}(\phi_i - \phi_k)\hspace{.05in} \begin{minipage}{.6in}\includegraphics[height=.5in]{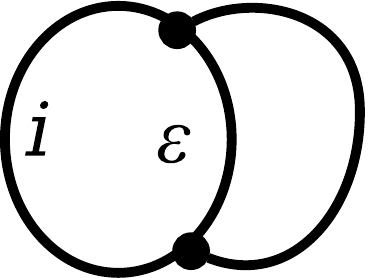}\end{minipage}\\
& = & \ds\theta(1,\vareps,i)(\lambda^{1\text{ } i}_\vareps)^{-1}(\lambda^{i\text{ } 1}_\vareps)^{-1}\prod_{k=0}^{j-1}(\phi_i - \phi_k).
\end{array}$$

\end{proof}

We are now able to outline an algorithm for explicitly computing a finite list of generators for the Kauffman bracket ideal of a genus-1 tangle $\G$.  Let $L$ be a closure of $\G$.  Then $L$ may be viewed as the relative Hopf pairing of $\G$ with some complementary genus-1 tangle $\mathcal{H}$.  Since the set $\{x_n,y_n\}$ is a basis for $K(\storus, 2)$,  we have that $\mathcal{H}$ can be written as a linear combination of elements of $\{x_n,y_n\}$.  Thus, $\langle L \rangle = \langle \G, \mathcal{H}\rangle$ is a linear combination of pairings $\langle \G, x_j \rangle$ and $\langle \G, y_j \rangle$.  Since we are considering only non-empty links, we have that $\langle L \rangle^\prime = \langle L \rangle/\delta = \langle \G, \mathcal{H} \rangle/\delta \in \laurent$.  So, $\langle \G, x_j \rangle /\delta$ and $\langle \G, y_j \rangle /\delta$ form a generating set for $I_\G$.

To compute these generators, we view $\G$ as an element of $K_R(\storus,2)$ which allows us to write $\G$ as a linear combination $\sum c_{i,\vareps} g_{i,\vareps}$ of graph basis elements.  We do this by computing the doubling pairing $\langle \G, g_{i,\vareps}\rangle_D$ for each $(i,\vareps)$ using Theorem \ref{thm:gh}, along with the fusion formula given in Theorem \ref{thm:fusion} and Formulas \ref{eq:bubble} - \ref{eq:removeloop}.  Since the graph basis is orthogonal with respect to the doubling pairing, we have that $c_{i,\vareps} = \langle \G, g_{i,\vareps}\rangle_D/\langle g_{i,\vareps}, g_{i,\vareps}\rangle_D$ for any $(i,\vareps)$.

Then, $\langle \G, x_j \rangle /\delta = \langle \sum c_{i,\vareps} g_{i,\vareps}, x_j \rangle /\delta = \sum (c_{i,\vareps}/\delta)\langle g_{i,\vareps}, x_j \rangle$ and $\langle \G, y_j \rangle/\delta = \langle \sum c_{i,\vareps} g_{i,\vareps} ,y_j\rangle/\delta =  \sum (c_{i,\vareps}/\delta)\langle g_{i,\vareps}, y_j \rangle$.  We compute these pairings using Proposition \ref{prop:basespairings} which states that $\langle g_{i,\vareps}, x_j \rangle$ and $\langle g_{i,\vareps}, y_j \rangle$ are non-zero only if $i \geq j$.  There are only finitely many $j$ less than or equal to a given $i$, and there are finitely many non-zero terms in the linear combination $\G = \sum c_{i,\vareps} g_{i,\vareps}$.  Therefore, the set of all non-zero $\langle \G, x_j \rangle/\delta$ and $\langle \G, y_j \rangle/\delta$ is a finite generating set for the Kauffman bracket ideal $I_\G$.


\section{The Kauffman bracket ideal of $\f$}\label{section:1057}


\begin{proof}[Proof of Theorem \ref{thm:1057}]

\begin{lemma}
$I_\f$ is generated by the following elements:
$$\pair{\f,x_i}/\delta =  \begin{minipage}{.95in}\includegraphics[height=1.3in]{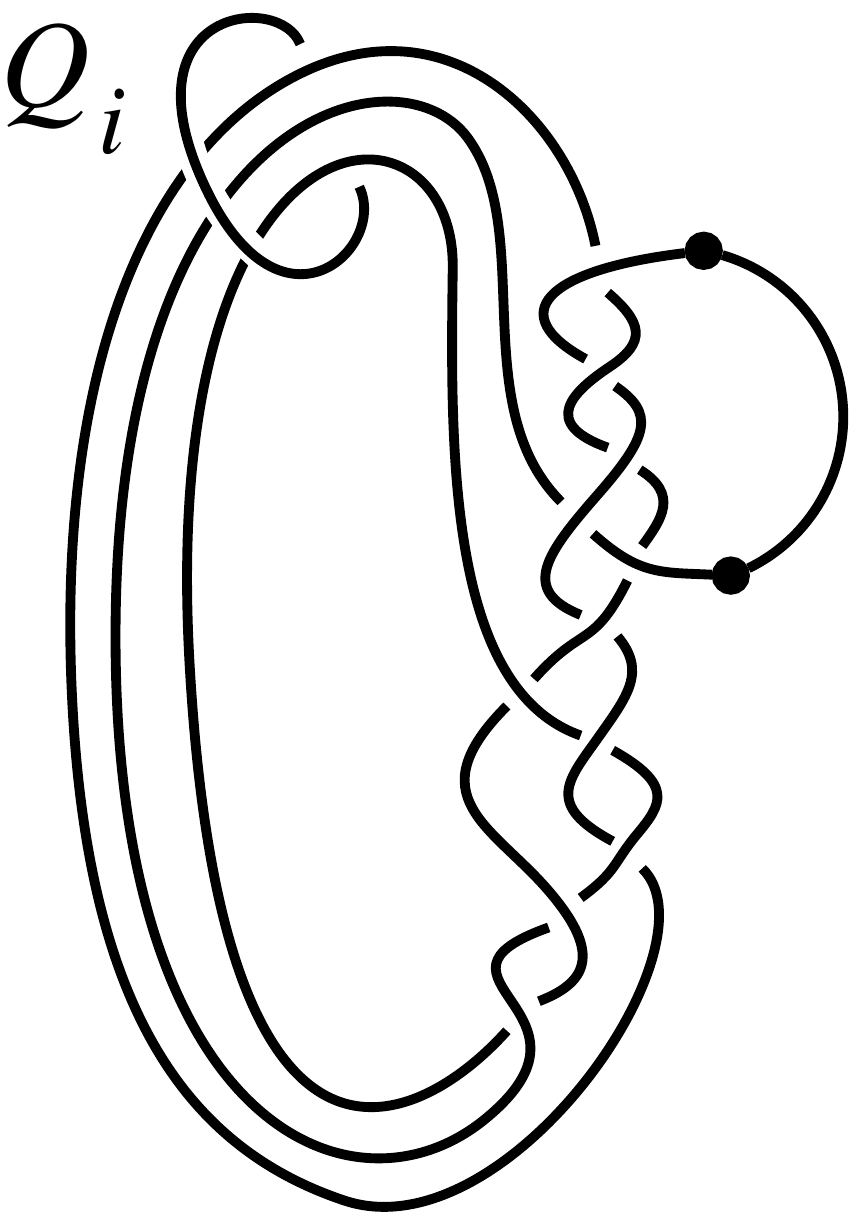}\end{minipage}/\delta \text{ and }
\pair{\f,y_i} /\delta =  \begin{minipage}{.95in}\includegraphics[height=1.3in]{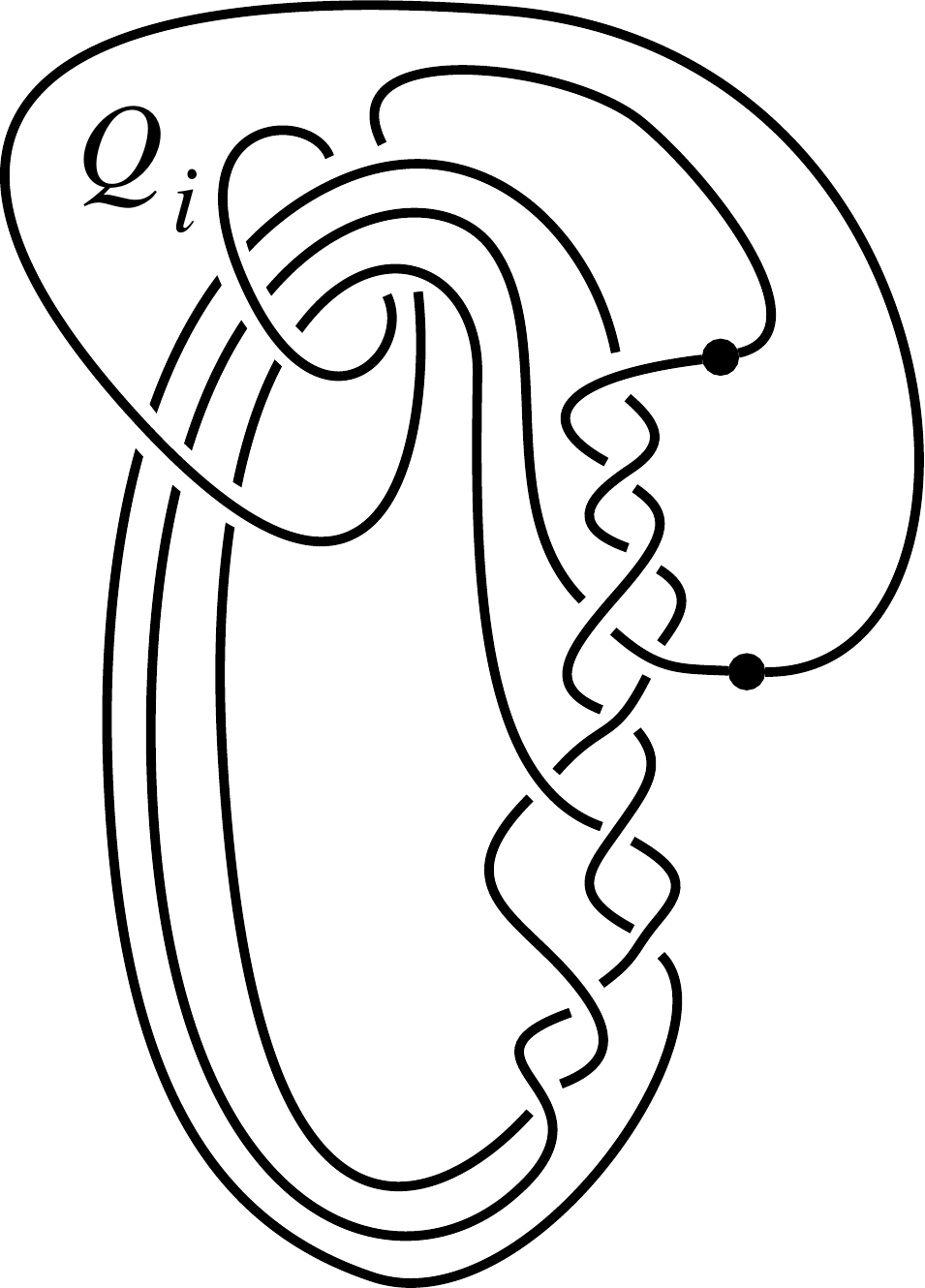}\end{minipage}/\delta\\$$
where $0 \leq i \leq 3$. 
\end{lemma}

\begin{proof}
As described in Section \ref{sec:fingen}, we write $\f$ as a linear combination $\f = \sum c_{i,\vareps}g_{i,\vareps}$ of graph basis elements, and we have that \begin{equation}\label{eqn:cieps} c_{i,\vareps} = \langle \f, g_{i,\vareps}\rangle_D/\langle g_{i,\vareps}, g_{i,\vareps}\rangle_D.\end{equation}  The formula for $\langle g_{i,\vareps}, g_{i,\vareps}\rangle_D$ is given in Theorem \ref{thm:doublingpairing}.  We need to compute $\langle \f, g_{i,\vareps}\rangle_D$. 

We use Theorems \ref{thm:fusion} and \ref{thm:gh} and Formulas \ref{eq:bubble} - \ref{eq:removeloop} to find a general formula for $\langle \f, g_{i,\vareps}\rangle_D$. The computation of this formula is given in Appendix \ref{app:linearcombo}.  From the second line of that computation, one can see using admissibility that $\langle \f, g_{i,\vareps}\rangle_D = 0$ unless $i = 1$ or $i = 3$.  So, we need only compute four coefficients: $c_{1,0}$, $c_{1,2}$, $c_{3,2}$, and $c_{3,4}$.  Using some code from \cite{h}, we implemented the formula derived in Appendix \ref{app:linearcombo} in Mathematica to find explicit expressions for these coefficients.  As $\langle \f, g_{1,0}\rangle_D$ turned out to be zero, we have that $c_{1,0} = 0$.  The remaining three coefficients are as follows:\\
\begin{align*}
c_{1,2} &  =  \ds \frac{\langle \f, g_{1,2}\rangle_D}{\langle g_{1,2},g_{1,2}\rangle_D} =  \begin{minipage}{.9in}\includegraphics[height=1.3in]{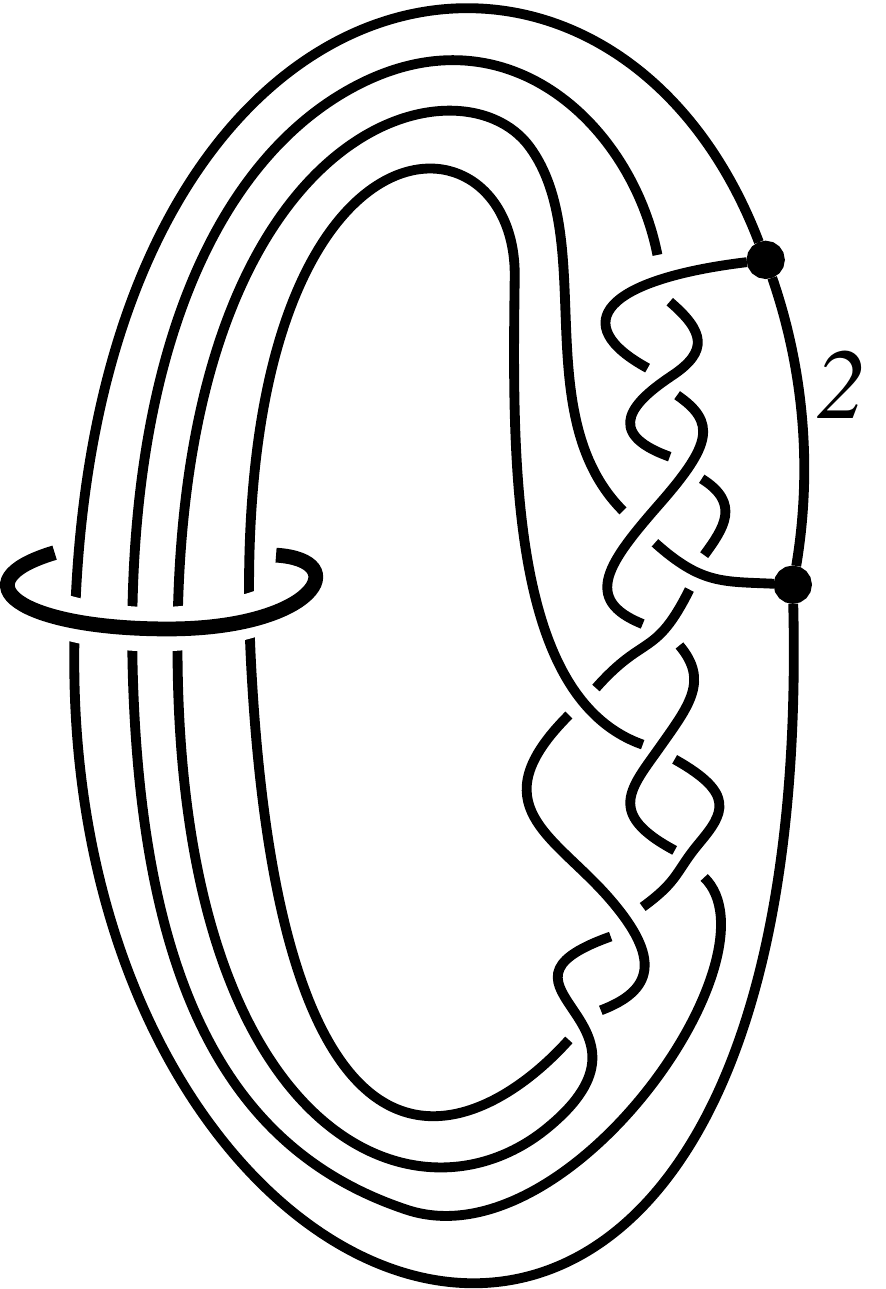}\end{minipage}/\langle g_{1,2},g_{1,2}\rangle_D\\
& = (1 + A^4 + A^8)^{-1} (-A^{-21} + 2A^{-17} -4A^{-13} +4A^{-9} -3A^{-5}+\\
& \hspace{.2in}  2A^{-1} +A^3-4A^7+4A^{11}-4A^{15} + 2A^{19} - A^{23}),\\
\end{align*}
\begin{align*}
c_{3,2} &   =  \ds \frac{\langle \f, g_{3,2}\rangle_D}{\langle g_{3,2},g_{3,2}\rangle_D} =  \begin{minipage}{.9in}\includegraphics[height=1.3in]{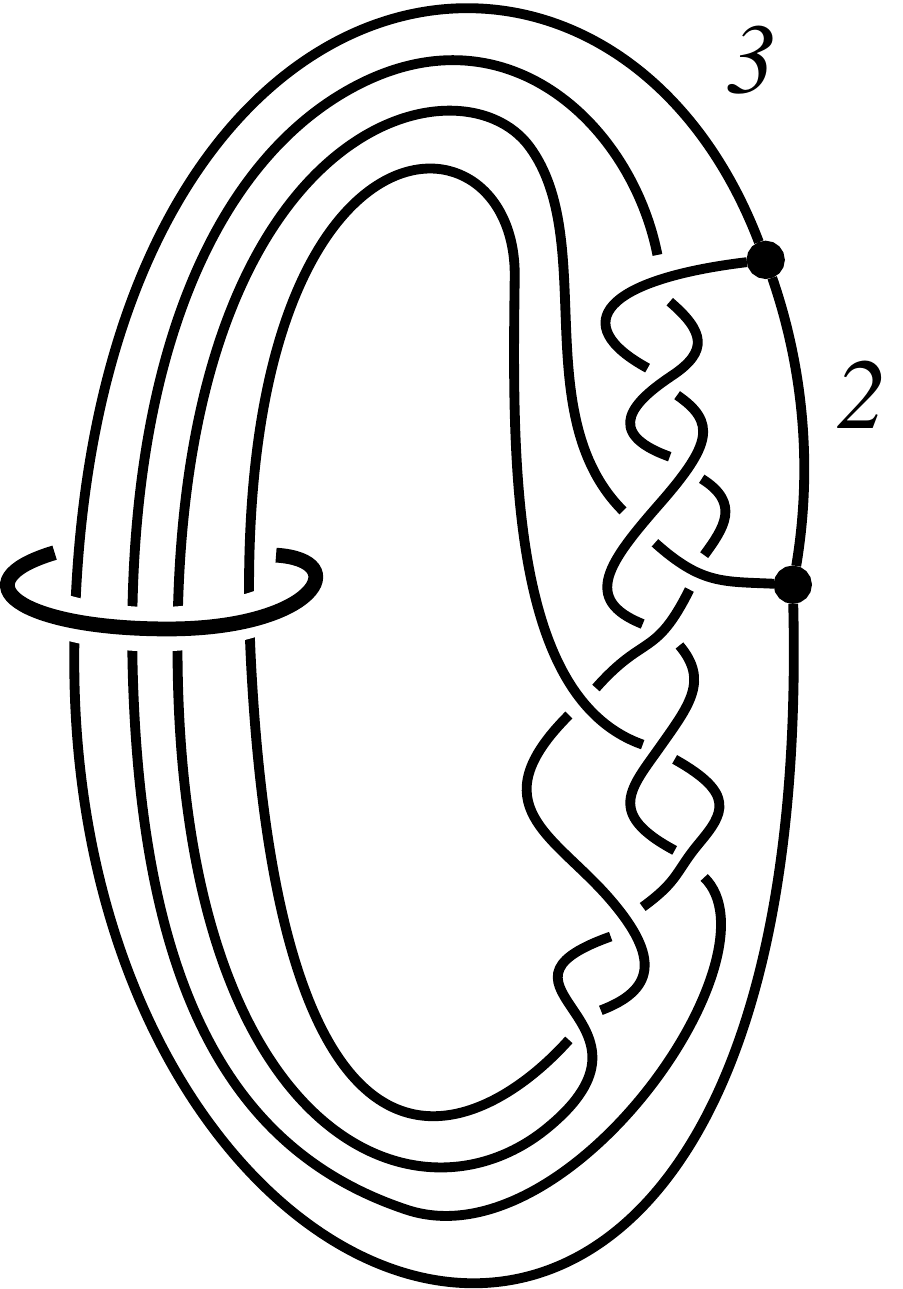}\end{minipage}/\langle g_{3,2},g_{3,2}\rangle_D\\
& =  (A^7 + A^{11} + A^{15} + A^{19})^{-1}(1-A^4+A^8-A^{16}+A^{20}), \text{ and}\\
c_{3,4} &  =  \ds \frac{\langle \f, g_{3,4}\rangle_D}{\langle g_{3,4},g_{3,4}\rangle_D} =  \begin{minipage}{.9in}\includegraphics[height=1.3in]{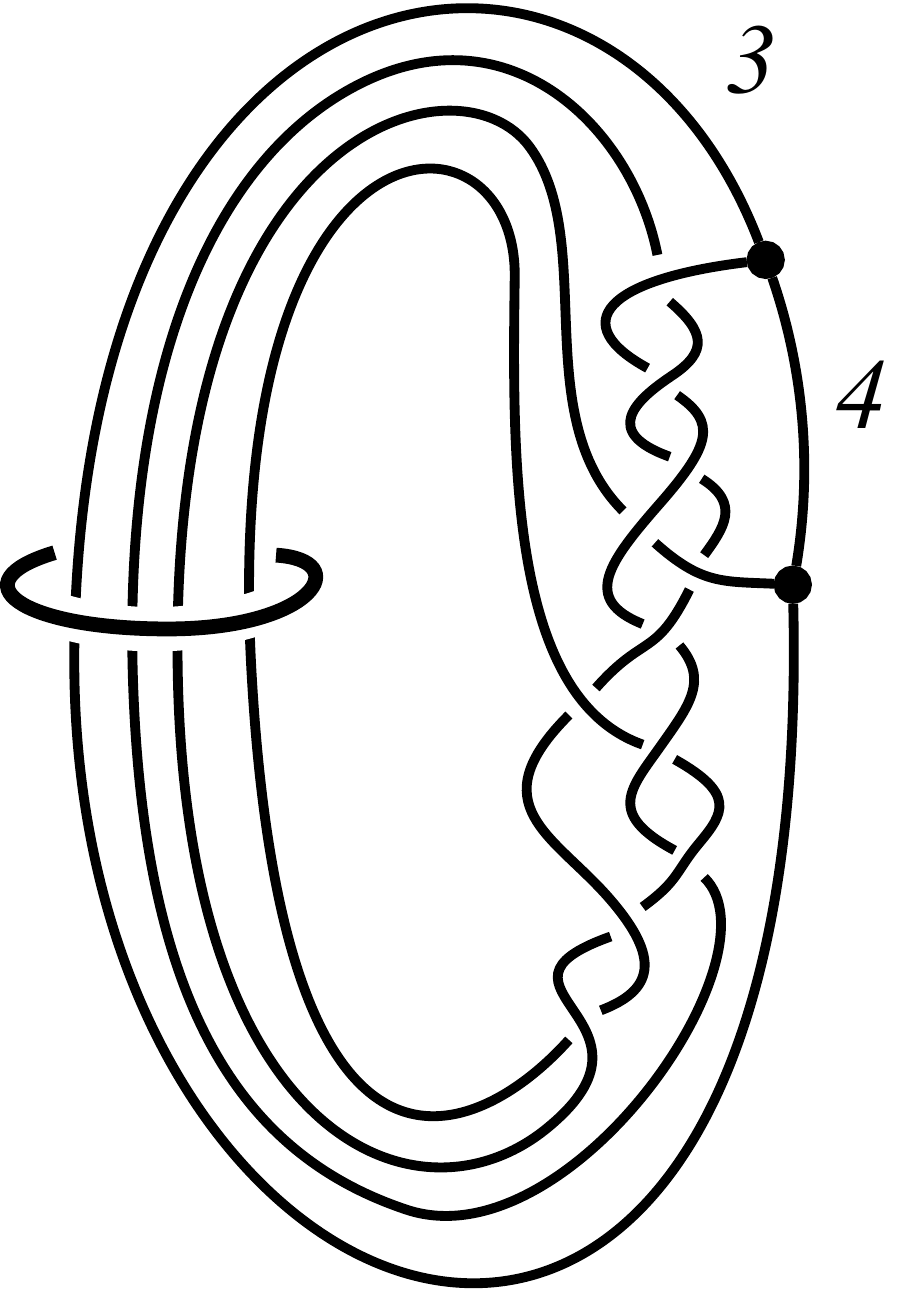}\end{minipage}/\langle g_{3,4},g_{3,4}\rangle_D = A^3.
\end{align*}
So $\f = c_{1,2}g_{1,2} + c_{3,2} g_{3,2} + c_{3,4} g_{3,4}$.

Now, we use Lemma \ref{lemma:removingq} and Proposition \ref{prop:basespairings} to compute $\langle \f, x_i\rangle/\delta$ and $\langle \f, y_i\rangle/\delta$ for any $i$ to obtain a list of generators for $I_\f$.  According to Proposition \ref{prop:basespairings}, since $\f = c_{1,2}g_{1,2} + c_{3,2} g_{3,2} + c_{3,4} g_{3,4}$, we have that $\langle \f, x_i\rangle/\delta$ and $\langle \f, y_i\rangle/\delta$ are zero if $i > 3$.   So, $I_\f$ is generated by $\langle \f, x_i\rangle/\delta$ and $\langle \f, y_i\rangle/\delta$ where $i \leq 3$.  An explicit expression for each generator is given in Appendix \ref{app:generators}.
\end{proof}

Let $g_1, \ldots g_8$ denote these eight generators rescaled by the power of $A$ necessary to make the lowest degree term a constant.\footnote{Actually, as seen in Appendix \ref{app:generators}, $g_4$ is a multiple of $g_3$, and $g_8$ is a multiple of $g_7$.  Thus, $g_4$ and $g_8$ are not needed in the list of generators.}  Since $A$ is a unit in $\laurent$, we have that the ideal $\langle g_1, \ldots, g_8 \rangle_{\laurent} = I_\f$.  Using the GroebnerBasis command in Mathematica, we see that $\{11, 4-A^4\}$ is a generating set for the ideal $\langle g_1, \ldots, g_8\rangle_{\mathbb{Z}[A]}$.

\begin{lemma}
The Kauffman bracket ideal $I_\f = \langle 11, 4-A^4\rangle_{\laurent}$, and $I_\f$ is non-trivial in $\laurent$.
\end{lemma}
\begin{proof}

We first show that $I_\f \subseteq \langle 11, 4-A^4\rangle_{\laurent}$.  Let $f \in I_\f$.  Then $f = f_1g_1 + \ldots f_8 g_8$ for some $f_i \in \laurent$.  Since $\langle g_1, \ldots, g_8 \rangle_{\mathbb{Z}[A]} = \langle 11, 4-A^4\rangle_{\mathbb{Z}[A]}$, we have for each $i$ that $g_i = 11r_i + (4-A^4)s_i$ for some $r_i$ and $s_i$ in $\mathbb{Z}[A]$.  Then,
\[\begin{array}{r c l }
f &  = & f_1 g_1 + \ldots f_8 g_8 \\
& = & f_1 (11r_1 + (4-A^4)s_1) + \ldots + f_8 (11r_8 + (4-A^4)s_8)\\
& = & 11(f_1 r_1 + \ldots f_8 r_8) + (4-A^4) (f_1 s_1 + \ldots + f_8 s_8)
\end{array}\] which is an element of $\langle 11, 4-A^4\rangle_{\laurent}$.

It is easy to see that $\langle 11, 4 - A^4 \rangle \subseteq I_\f$.  Since $11$ and $4-A^4$ are elements of $\langle g_1, \ldots, g_8 \rangle_{\mathbb{Z}[A]}$, it follows immediately they are both in $\langle g_1, \ldots, g_8 \rangle_{\laurent} = I_\f $.  Therefore, $I_\f = \langle 11, 4-A^4 \rangle_{\laurent}$.

It remains only to show that $I_\f = \langle 11, 4-A^4\rangle_{\laurent}$ is non-trivial.  Let $\rho :\laurent \rightarrow \mathbb{Z}_{11}$ be the map which sends $A$ to $3$.  It is easy to see that $\rho$ is a ring homomorphism.  The image of $I_\f$ under $\rho$ is the ideal $\langle 11, 4 - 81 \rangle = \langle 11 \rangle = \langle 0 \rangle$ in  $\mathbb{Z}_{11}$.  So, $I_\f \subseteq \ker \rho$.  Since $\rho$ is not the trivial homomorphism, this implies that $I_\f \neq \laurent$.
\end{proof} 

Recall that the Jones polynomial of an oriented link $L$ is defined to be $J_L(\sqrt{t}) = A^{-3\omega(D)} \langle D \rangle^\prime$ where $D$ is an oriented diagram of $L$ with writhe $\omega(D)$ and $t = A^{-4}$.

We show that if $L$ is a closure of $\f$, then $J_L(\sqrt{t})$ evaluated at $\sqrt{t} = 5$ is $0 \pmod{11}$.  Let $D$ be an oriented diagram for $L$.  Then $\langle D \rangle^\prime \in I_\f$ and thus $A^{-3\omega(D)}\langle D \rangle^\prime \in I_\f$.  So, $\rho (A^{-3\omega(D)}\langle D \rangle^\prime) = 0$ in $\mathbb{Z}_{11}$.  Note that $\sqrt{t} = 5$ implies $t = 25 = 3 \pmod{11}$ and $3^{-4} = \frac{1}{81} = \frac{1}{4} = 3 \pmod{11}$.  Therefore, $J_L(\sqrt{t})|_{\sqrt{t}=5} = \rho (A^{-3\omega(D)}\langle D \rangle^\prime) = 0 \pmod{11}$.

\end{proof}

\begin{proposition}
The genus-1 tangle $\f$ contains no local knots.
\end{proposition}

\begin{proof}
The simplest closure of $\f$ is the knot $10_{57}$ which is prime.  Therefore, $\f$ has no local knots unless $10_{57}$ itself is a local knot.  In that case, any closure $L$ of $\f$ may be written as the connect sum of $10_{57}$ with some knot $K \subset S^3$, and we have that $\langle L \rangle ^\prime = \langle 10_{57} \rangle ^\prime \langle K \rangle ^\prime$.  Thus, $I_\f$ is the principal ideal generated by $\langle 10_{57} \rangle ^\prime$.  However, since $11 \in I_\f$, this means that $11$ is a multiple of $\langle 10_{57} \rangle ^\prime$ which is impossible.
\end{proof}


\section{Partial closures}\label{section:partialclosures}

Recall, the partial closure of a $(B^3,2n)$-tangle $\T$ is a genus-1 tangle obtained from $\T$ by gluing a copy of $D^2 \times I$ containing $n-1$ properly embedded arcs to $B^3$ as indicated in Fig. \ref{fig:partialclosure}.  We denote the partial closure by $\hat \T$.  We can describe this more colloquially as partially closing off $\T$ with $n-1$ simple arcs and placing the hole of the solid torus as indicated in Fig. \ref{fig:partialclosure}.

Theorem \ref{thm:partialclosures} states that in the case of a $(B^3,4)$-tangle whose partial closure has a single component, the Kauffman bracket ideal of the partial closure is exactly the Kauffman bracket ideal of the original $(B^3,4)$-tangle.  Before proving this result, we need the following lemma.

\begin{lemma}\label{lemma:partialclosure}
Let $\T$ be a $(B^3,4)$-tangle and let $\hat \T$ denote the genus-1 tangle which is the partial closure of $\T$.  Then the Kauffman bracket ideal $I_{\hat\T} = \langle \langle d(\T)\rangle^\prime,  (1-A^{-4})\langle n(\T) \rangle^\prime \rangle$.\\
\end{lemma}

\begin{proof}
Recall, we can think of the operation of taking closures of $(B^3,2n)$-tangles as a symmetric bilinear pairing on $K(B^3,2n)$ as follows: \[\left\langle\text{ } \begin{minipage}{.7in}\includegraphics[width=.6in]{tangleS-eps-converted-to.pdf}\end{minipage}, \begin{minipage}{.7in}\includegraphics[width=.6in]{tangleR-eps-converted-to.pdf}\end{minipage}\right\rangle = \left\langle \begin{minipage}{1.5in}\includegraphics[width=1.45in]{3ballpairing-eps-converted-to.pdf}\end{minipage}\right\rangle.\]

For a given closure $L$ of $\hat \T$, $L$ has $k$ strands passing through the hole of the solid torus for some non-negative integer $k$.  We can think of $\langle L\rangle$ as the pairing of the $(B^3, 2k+2)$-tangle in Fig. \ref{fig:tk}, denoted by $\T_k$, with some complementary $(B^3, 2k+2)$-tangle.  Since the Catalan tangles form a basis for $K(B^3, 2k+2)$, we can write $\langle L\rangle$ as a linear combination of $\langle \T_k, \C\rangle$ where $\C$ is a $(2k+2)$-Catalan tangle.  See Fig. \ref{fig:pcexample} for an example where $k = 2$.

\begin{figure}
\includegraphics[width=.925in]{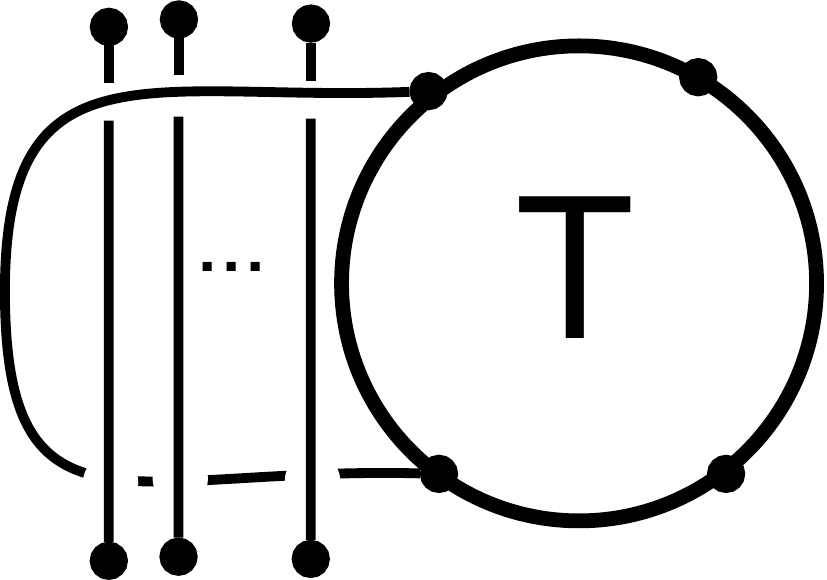}
 \caption{Given a $(B^3, 4)$-tangle $\T$, the tangle consisting of $\hat\T$ with an additional $k$ strands placed as above is denoted by $\T_k$.}\label{fig:tk}
\end{figure}

For any non-negative integer $k$ and Catalan tangle $\C$, we show that $\langle \T_k,\C\rangle = f \langle d(\T)\rangle + g(1-A^{-4})\langle n(\T) \rangle$ for some $f$ and $g$ in $\laurent$.  Hence, $\langle L \rangle$ is also a linear combination of $ \langle d(\T)\rangle$ and $ (1-A^{-4})\langle n(\T) \rangle$, and we have that the reduced Kauffman bracket polynomial $\langle L \rangle ^\prime$ is generated by $\langle d(\T)\rangle^\prime$ and $(1-A^{-4})\langle n(\T) \rangle^\prime$.

\begin{figure}
$\begin{minipage}{1in}\includegraphics[width=1in]{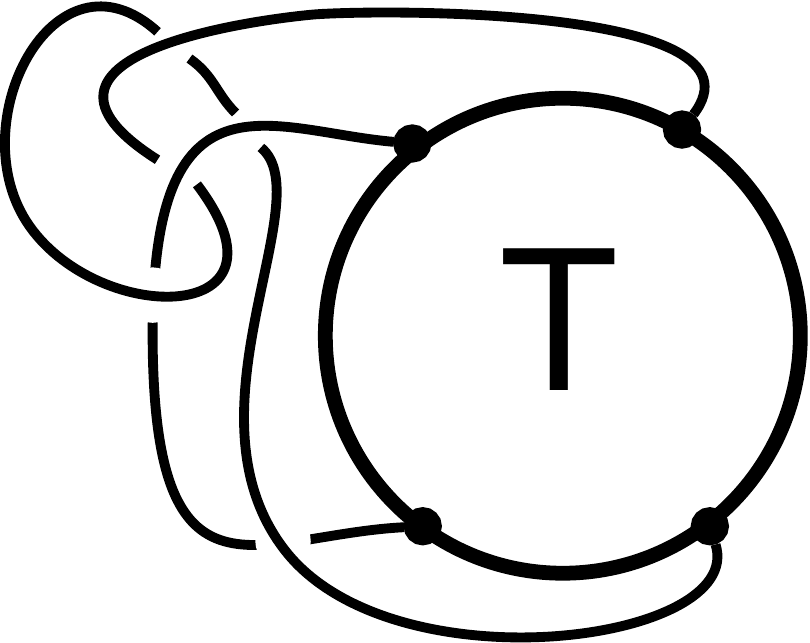}\end{minipage} =  \left\langle\text{ } \begin{minipage}{.9in}\includegraphics[width=.85in]{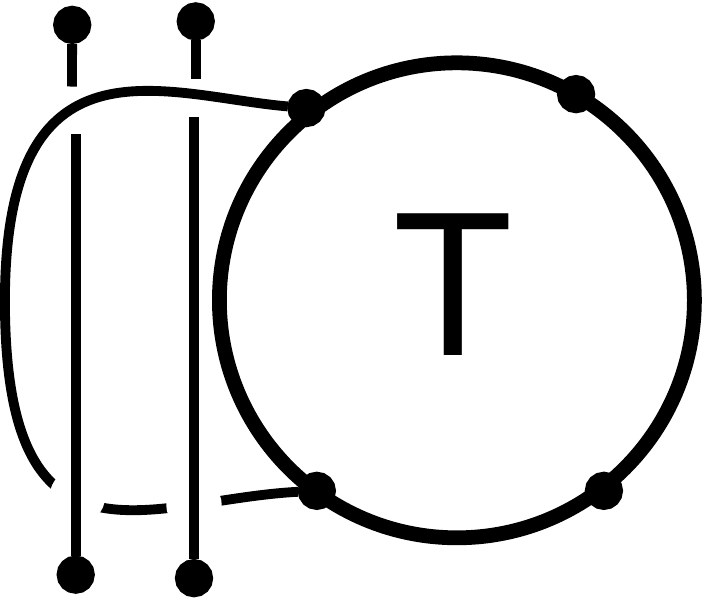}\end{minipage}, \begin{minipage}{.7in}\includegraphics[width=.6in]{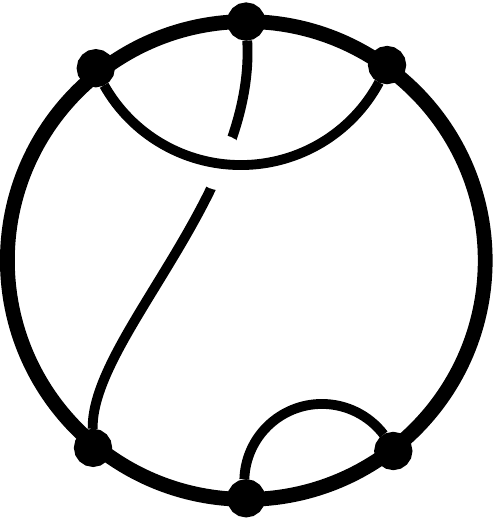}\end{minipage}\right\rangle$\\

\smallskip

$  =  A \left\langle\text{ } \begin{minipage}{.9in}\includegraphics[width=.85in]{pcex3balltangle-eps-converted-to.pdf}\end{minipage}, \begin{minipage}{.7in}\includegraphics[width=.6in]{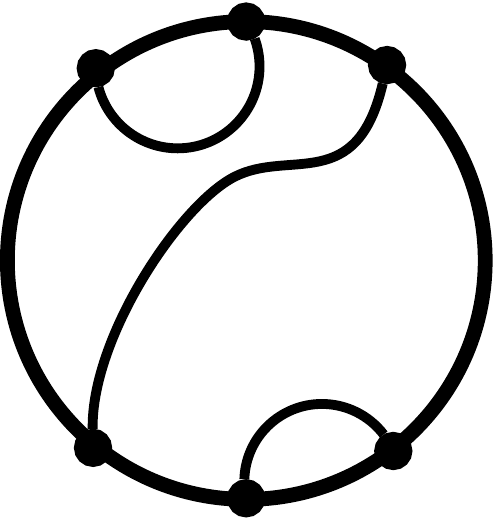}\end{minipage}\right\rangle +  A^{-1} \left\langle\text{ } \begin{minipage}{.9in}\includegraphics[width=.85in]{pcex3balltangle-eps-converted-to.pdf}\end{minipage}, \begin{minipage}{.7in}\includegraphics[width=.6in]{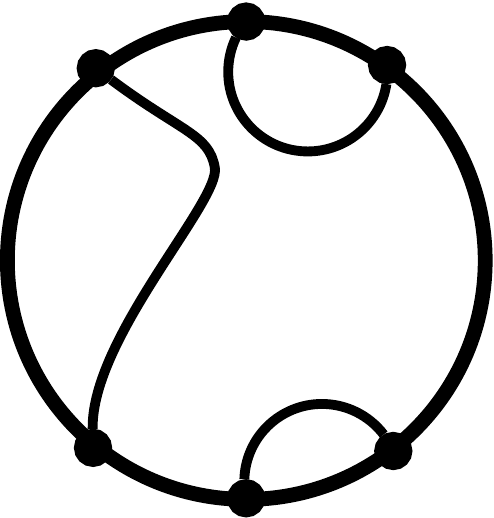}\end{minipage}\right\rangle$
\caption{The partial closure $L$ of $\hat \T$ represented as a linear combination of pairings of $(B^3,6)$-tangles.}\label{fig:pcexample}
\end{figure}

We proceed by induction on $k$.  If $k=0$, then $L = \langle \T_0 , \mathcal{U}\rangle = \langle \hat\T , \mathcal{U}\rangle$ for some $(B^3,2)$-tangle $\mathcal{U}$ and it is easy to see that $\langle L\rangle$ is a multiple of $\langle d(\T) \rangle$.

If $k = 1$, then there are two possibilities for $\langle \T_1, \C\rangle$:\\
 $\begin{array}{r c l} \left\langle \T_1,\begin{minipage}{.575in}\includegraphics[width=.55in]{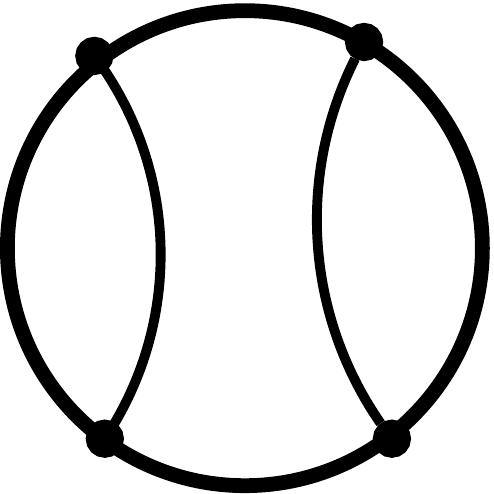}\end{minipage}\right\rangle &  = & \left\langle\text{ } \begin{minipage}{.725in}\includegraphics[width=.7in]{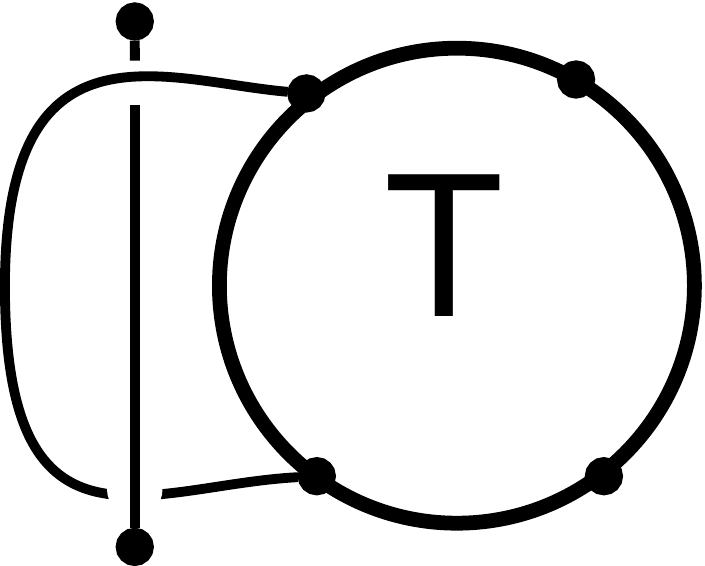}\end{minipage}, \begin{minipage}{.575in}\includegraphics[width=.55in]{lemma5102-eps-converted-to.pdf}\end{minipage}\right\rangle = \left\langle \begin{minipage}{.9in}\includegraphics[width=.85in]{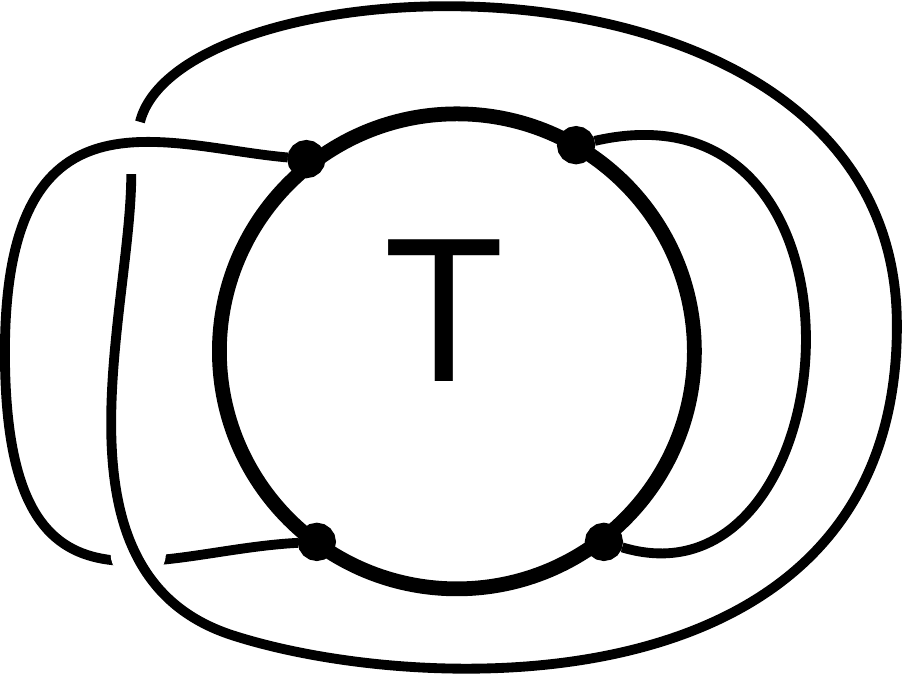}\end{minipage} \right\rangle\\
 & & \\
& = & (-A^4-A^{-4})\langle d(\T)\rangle,\text{ and}\\
\end{array}$

\begin{gather}
\begin{aligned}\label{eq:num}
 \left\langle \T_1, \begin{minipage}{.55in}\includegraphics[width=.525in]{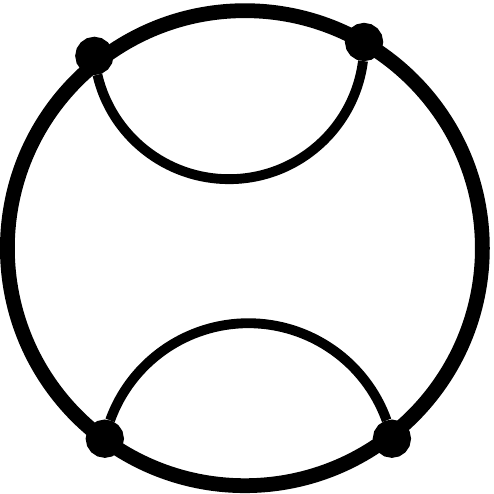}\end{minipage}\right\rangle &=\left\langle\text{ } \begin{minipage}{.7in}\includegraphics[width=.675in]{lemma5101-eps-converted-to.pdf}\end{minipage}, \begin{minipage}{.55in}\includegraphics[width=.525in]{lemma5108-eps-converted-to.pdf}\end{minipage}\right\rangle = \left\langle \begin{minipage}{.8in}\includegraphics[width=.75in]{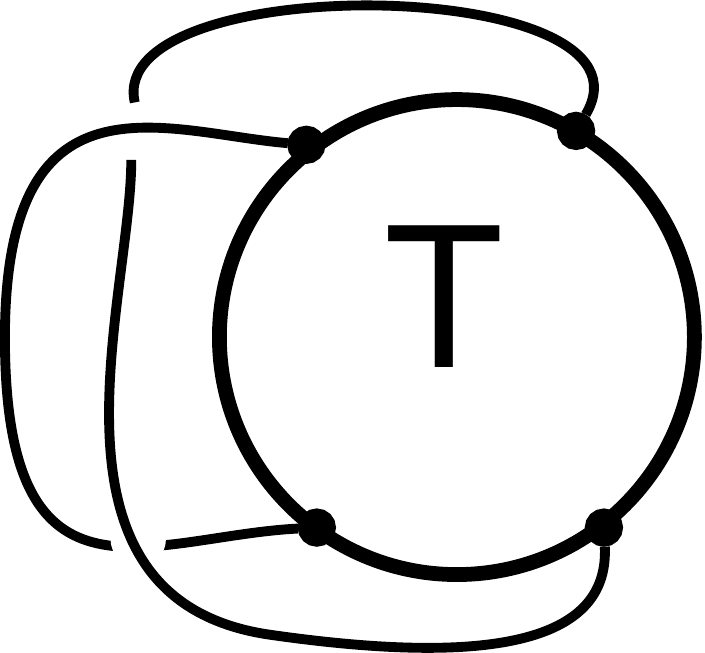}\end{minipage} \right\rangle
 \\ & = A \left\langle \begin{minipage}{.725in}\includegraphics[width=.7in]{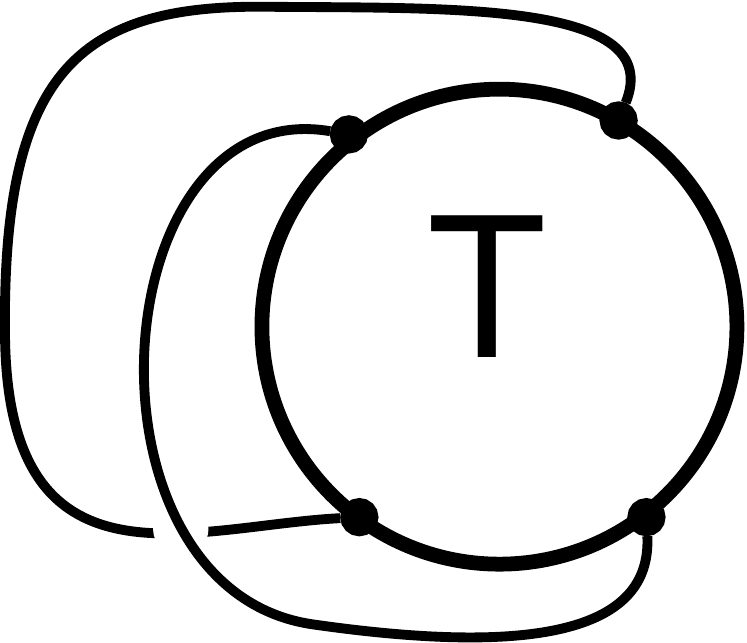}\end{minipage}\right\rangle + A^{-1} \left\langle \begin{minipage}{.725in}\includegraphics[width=.7in]{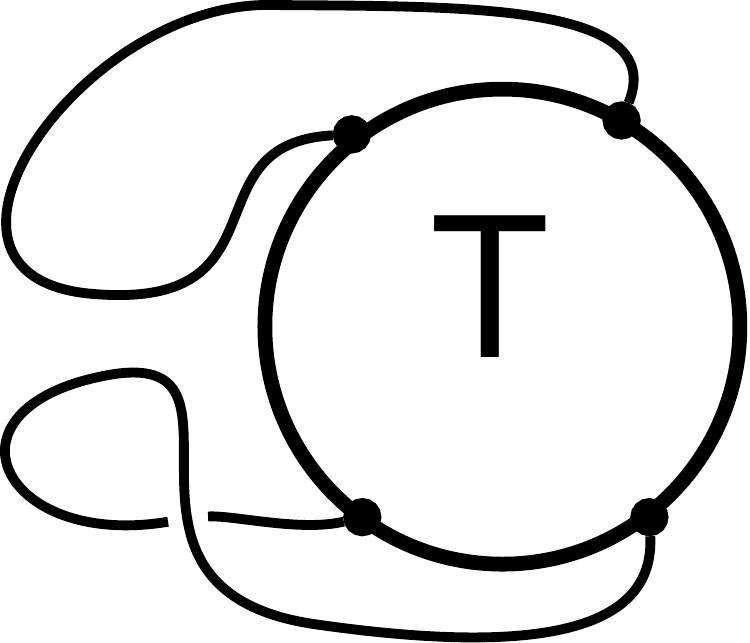}\end{minipage}\right\rangle
\\ &= A^2 \left\langle \begin{minipage}{.725in}\includegraphics[width=.7in]{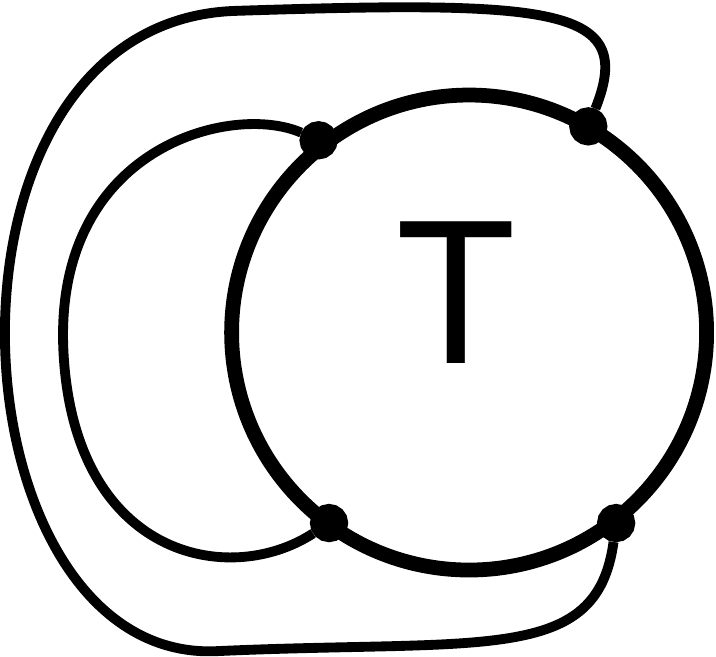}\end{minipage}\right\rangle + \left\langle \begin{minipage}{.55in}\includegraphics[width=.525in]{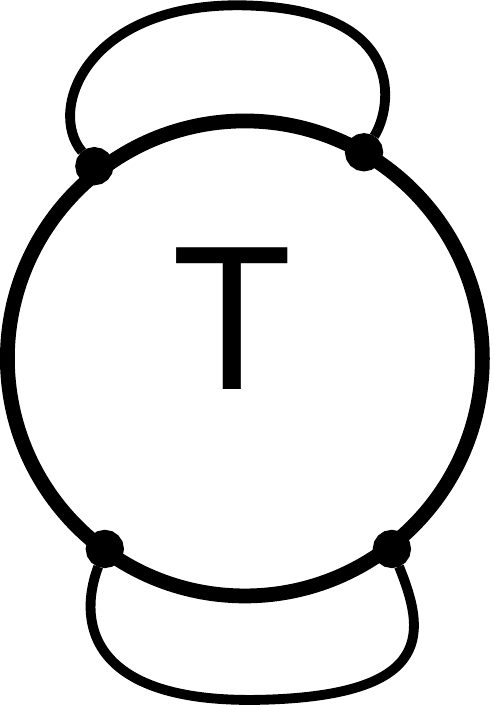}\end{minipage}\right\rangle - A^{-4} \left\langle \begin{minipage}{.55in}\includegraphics[width=.525in]{lemma51num-eps-converted-to.pdf}\end{minipage}\right\rangle
\\ &= (1-A^{-4})\langle n(\T) \rangle + A^2 \langle d(\T)\rangle.
\end{aligned}
\end{gather}

Suppose the property holds for $k > 1$, and consider $\langle \T_{k+1}, \C\rangle$.  We have two cases to consider.  The first case is that the Catalan tangle $\C$ connects two of the $k+1$ strands in $\T_{k+1}$ which are adjacent.  Then, $$\langle \T_{k+1}, \C\rangle = \left\langle\begin{minipage}{.85in}\includegraphics[width=.825in]{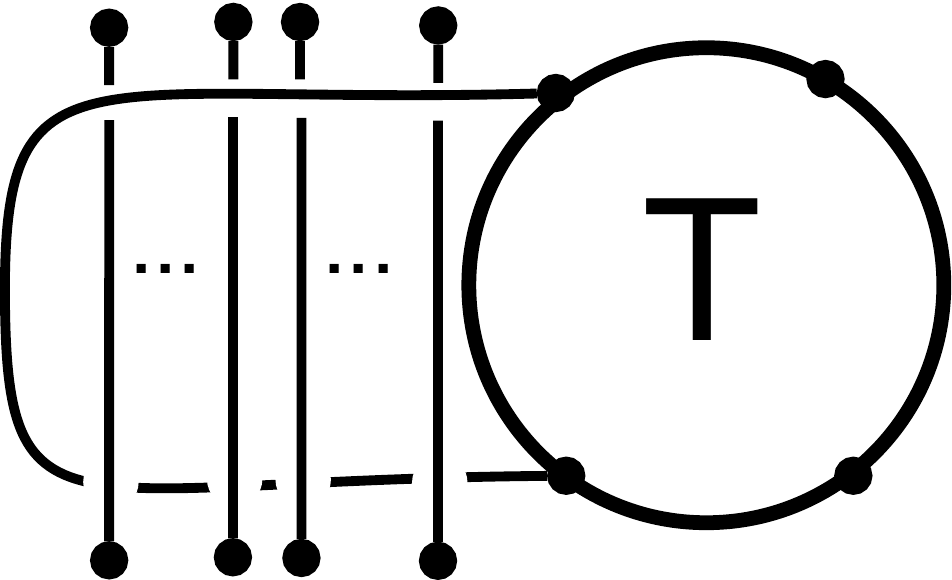}\end{minipage},\C\right\rangle = \left\langle \begin{minipage}{.85in}\includegraphics[width=.825in]{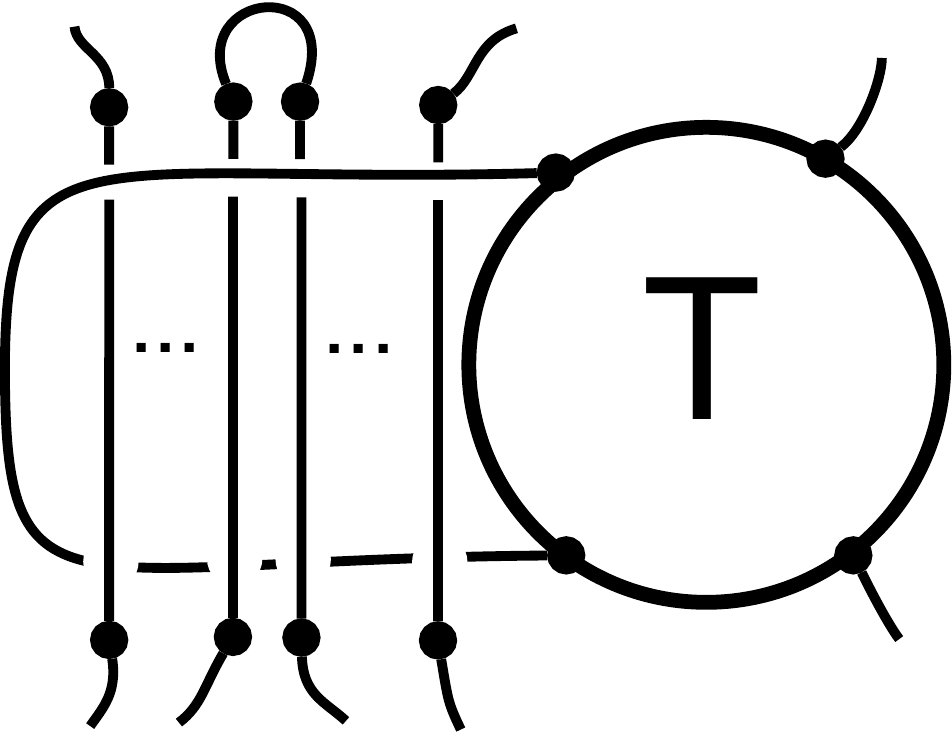}\end{minipage}\right\rangle$$ and we may perform a Reidemeister II move as follows:
$$\left\langle \begin{minipage}{.85in}\includegraphics[width=.825in]{pcproofadjacentstrands2-eps-converted-to.pdf}\end{minipage}\right\rangle = \left\langle \begin{minipage}{.85in}\includegraphics[width=.825in]{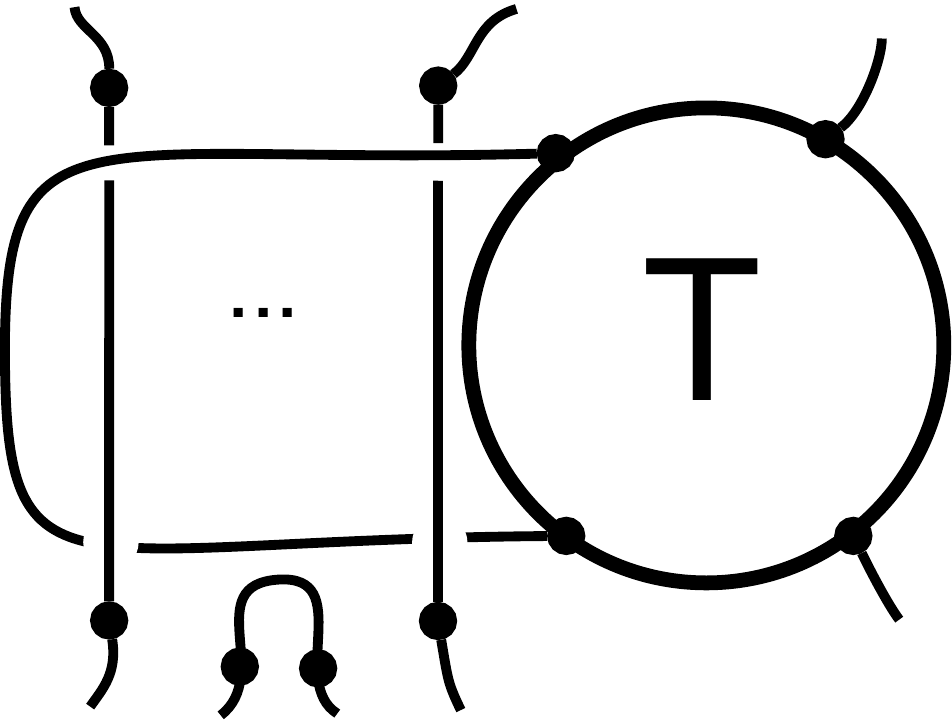}\end{minipage} \right\rangle = \left\langle\begin{minipage}{.85in}\includegraphics[width=.825in]{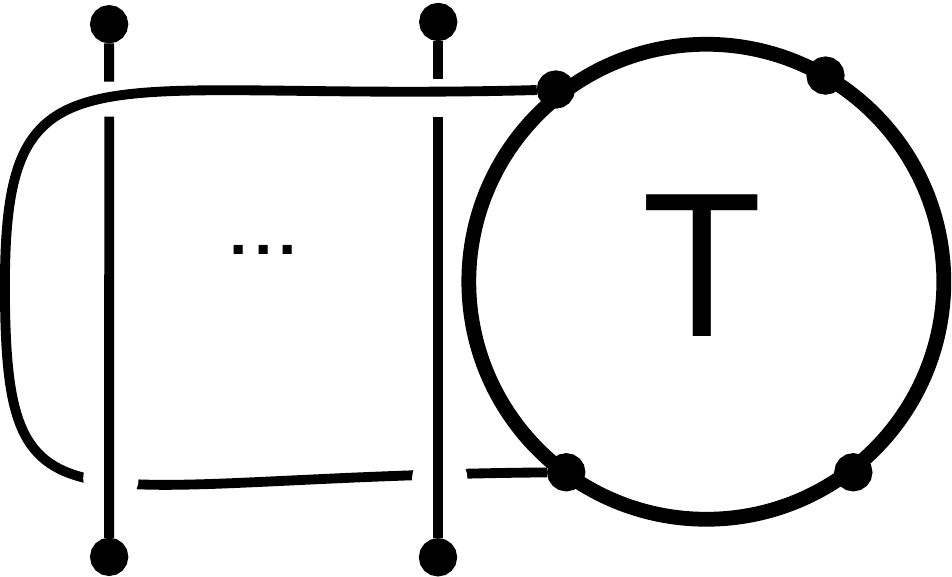}\end{minipage},\C^\prime\right\rangle = \langle \T_{k-1}, \C^\prime\rangle$$ where $\C^\prime$ is some $(2k-2)$-Catalan tangle.  So, $\langle \T_{k+1}, \C\rangle = \langle \T_{k-1}, \C^\prime\rangle = f \langle d(\T)\rangle + g(1-A^{-4})\langle n(\T) \rangle$ for some $f$ and $g$ in $\laurent$.

If no adjacent strands in $\T_{k+1}$ via pairing with the Catalan tangle, then $$\begin{array}{c}
\langle \T_{k+1}, \C\rangle = \left\langle\begin{minipage}{.95in}\includegraphics[width=.925in]{pcproof3strands-eps-converted-to.pdf}\end{minipage},\C\right\rangle = \left\langle \begin{minipage}{1.2in}\includegraphics[width=1.15in]{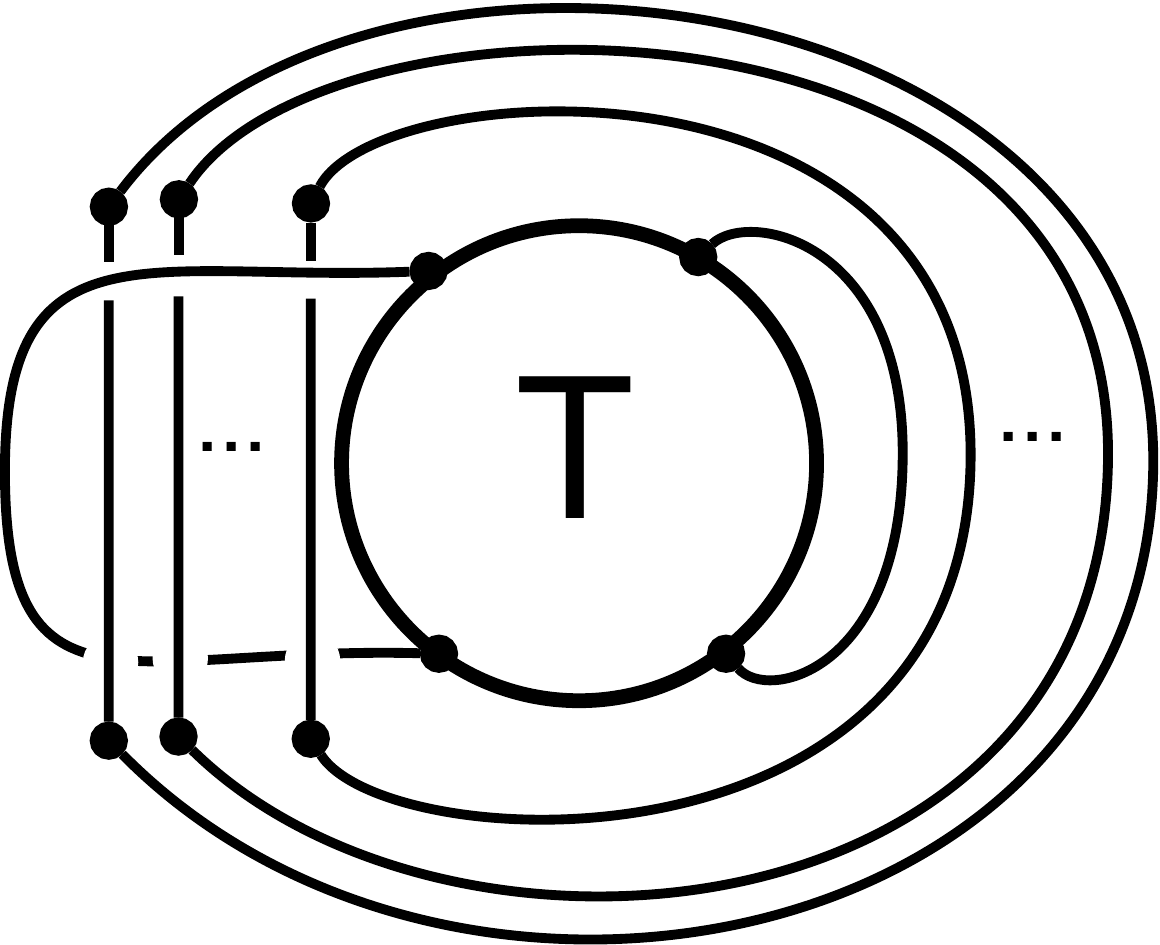}\end{minipage}\right\rangle \\
 = (-A^4-A^{-4})\left\langle \begin{minipage}{1.2in}\includegraphics[width=1.15in]{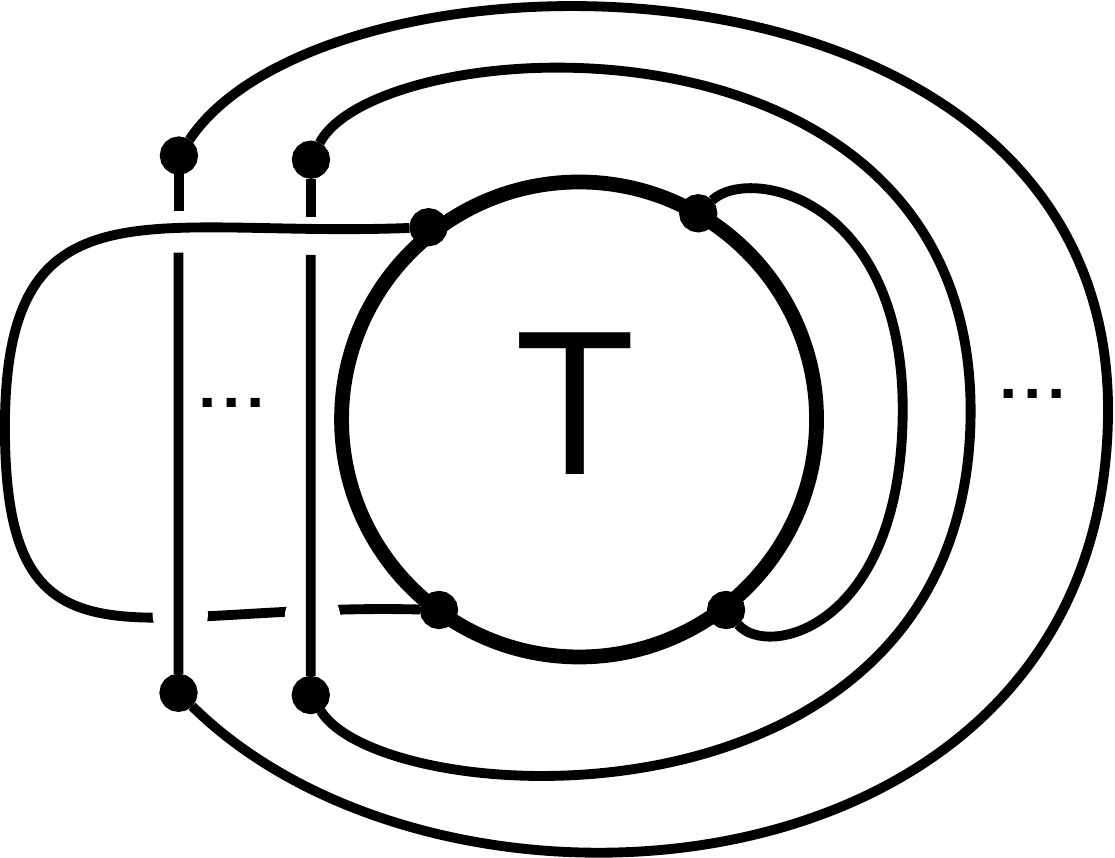}\end{minipage} \right\rangle = (-A^4-A^{-4}) \langle \T_k, \C^\prime\rangle
\end{array}$$
where $\C^\prime$ is some $(2k+2)$-Catalan tangle.

So, given any non-negative integer $k$ and any $(2k+2)$-Catalan tangle, we have that $\langle \T_k,\C\rangle = f \langle d(\T)\rangle + g(1-A^{-4})\langle n(\T) \rangle$ for some $f$ and $g$ in $\laurent$.  Thus, for any closure $L$ of $\hat\T$, we see that $\langle L \rangle$ is a linear combination of $\langle d(\T)\rangle$ and $ (1-A^{-4})\langle n(\T) \rangle$.  This implies that $\langle L\rangle^\prime$ is a linear combination of $\langle d(\T)\rangle^\prime$ and $ (1-A^{-4})\langle n(\T) \rangle^\prime$, and so $I_{\hat \T} \subseteq  \langle\langle d(\T)\rangle^\prime,(1-A^{-4})\langle n(\T) \rangle^\prime\rangle$.

Since the denominator $d(\T)$ is clearly a closure of $\hat \T$, we have that $\langle d(\T)\rangle^\prime \in I_{\hat\T}$.  Let $\s$ denote the tangle $\begin{minipage}{.325in}\includegraphics[width=.3in]{lemma5108-eps-converted-to.pdf}\end{minipage}$.  We have from Equation \eqref{eq:num} that $(1-A^{-4})\langle n(\T) \rangle =  \langle \T_1, \s\rangle - A^2\langle d(\T)\rangle$.  Since $\left\langle \T_1, \s\right\rangle$ is the Kauffman bracket polynomial of a closure of $\hat\T$, we see that $\langle \T_1, \s\rangle/\delta \in I_{\hat\T}$.  So, $(1-A^{-4})\langle n(\T) \rangle^\prime$ is the difference of two elements of $I_{\hat\T}$ and is therefore an element of $I_{\hat\T}$ itself.  Hence, the Kauffman bracket ideal $I_{\hat\T} = \langle\langle d(\T)\rangle,(1-A^{-4})\langle  n(\T) \rangle\rangle$.
\end{proof}

We now prove Theorem \ref{thm:partialclosures}.

\begin{proof}[Proof of Theorem \ref{thm:partialclosures}]
Let $\T$ be a  $(B^3, 4)$-tangle with partial closure $\hat\T$ which has a single component. Since any closure of $\hat\T$ is also a closure of $\T$, we have that $I_{\hat\T} \subseteq I_\T =\langle \langle n(\T)\rangle^\prime, \langle d(\T)\rangle^\prime \rangle$ according to Theorem \ref{thm:psw}.  According to Lemma \ref{lemma:partialclosure}, $I_{\hat\T} = \langle \langle d(\T)\rangle^\prime,(1-A^{-4})\langle n(\T) \rangle^\prime \rangle$, so it remains only to show that $\langle n(\T)\rangle^\prime \in I_{\hat\T}$ to prove equality of the two ideals.

Since $\hat \T$ has a single component, the denominator $d(\T)$ is a knot.  Then, its Jones polynomial $J_{d(\T)}(t)$ evaluated at one is one by \cite[Theorem 15]{jo} and so $J_{d(\T)} (t)= (1-t)f(t) + 1$ for some $f(t) \in \mathbb{Z}[t,t^{-1}]$.  Since $J_{d(\T)} (t) = A^{-3\omega} \langle d(\T) \rangle^\prime$ where $A^{-4} = t$ and $\omega$ is the writhe of an oriented diagram of the denominator, we have that \begin{equation}\label{eq:ntdt}\langle d(\T) \rangle^\prime = A^{3\omega}(1-A^{-4})f(A^{-4}) + A^{3\omega}.\end{equation}

Then, $\langle n(\T)\rangle^\prime \langle d(\T)\rangle^\prime \in I_{\hat\T}$ since $\langle d(\T)\rangle^\prime \in I_{\hat\T}$. We also have from  Equation \ref{eq:ntdt} that $\langle n(\T)\rangle^\prime \langle d(\T)\rangle^\prime  =  A^{3\omega}f(A^{-4})(1-A^{-4})\langle n(\T)\rangle^\prime + A^{3\omega}\langle n(\T)\rangle^\prime$.  Clearly, $A^{3\omega}f(A^{-4})(1-A^{-4})\langle n(\T)\rangle^\prime \in I_{\hat \T}$.  So, $A^{3\omega}\langle n(\T)\rangle^\prime$ and thus $\langle n(\T)\rangle^\prime$ are elements of $I_{\hat \T}$ as well. This concludes the proof.
\end{proof}

\section*{Acknowledgements}
The author thanks Patrick Gilmer whose advice, corrections, and suggestions significantly contributed to this paper.

\appendix
\section{}\label{app:linearcombo}
Here we give the computation illustrating how to write $\f$ as a linear combination of graph basis elements $g_{i,\vareps}$.  We first find a general formula for the pairing $\langle \f, g_{i,\vareps}\rangle_D$ for any $(i,\vareps)$.  Using this formula and Mathematica code from \cite{h}, we were able to compute this pairing and find the explicit formulas for non-zero $c_{i,\vareps}$ given in Section \ref{section:1057}.  The Mathematica notebook we used to do this is available on the author's website.  Each sum in the following computation ranges over all admissible colorings of the corresponding graph.  Using Theorems \ref{thm:fusion} and \ref{thm:gh} along with Formulas \ref{eq:bubble} - \ref{eq:removeloop}, we have that:
\begin{align*}
\langle \f, g_{i,\vareps}\rangle_D & =  \begin{minipage}{1in}\includegraphics[width=1.1in]{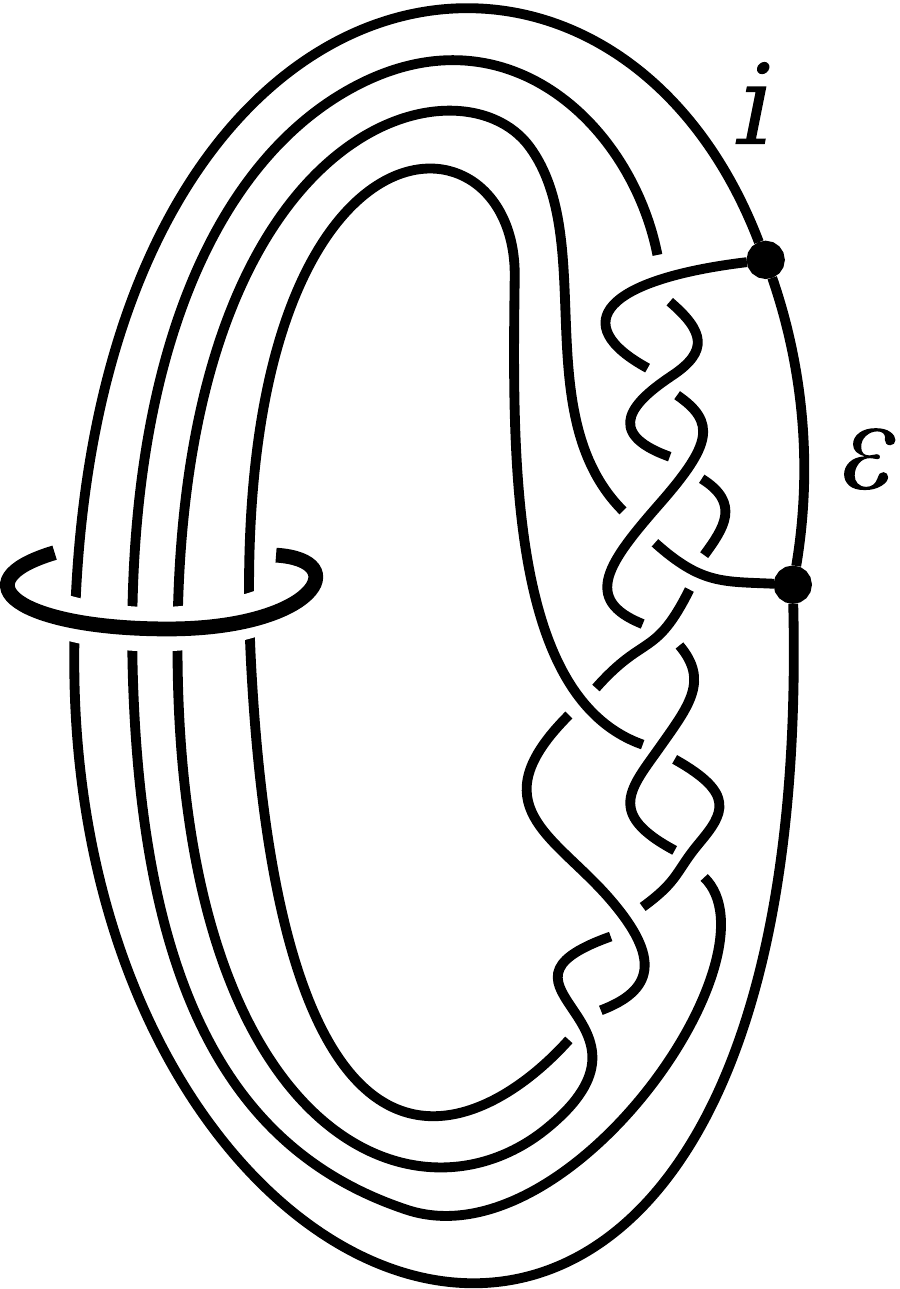}\end{minipage} =  \ds \sum_j c_j  \begin{centering}\begin{minipage}{.7in}\includegraphics[width=.6in]{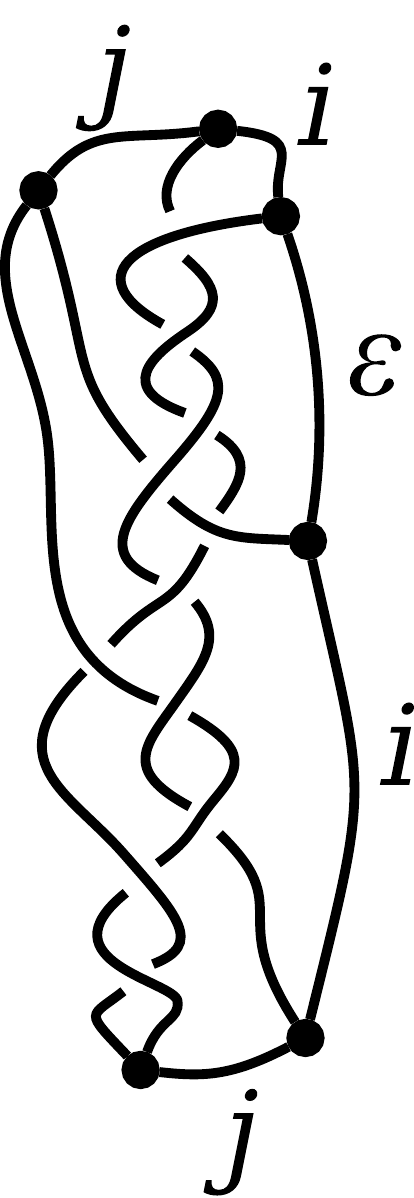}\end{minipage}\end{centering}\\
& \\
& \text{ where } c_j = \ds \frac{1}{\begin{centering}\begin{minipage}{.7in}\includegraphics[width=.65in]{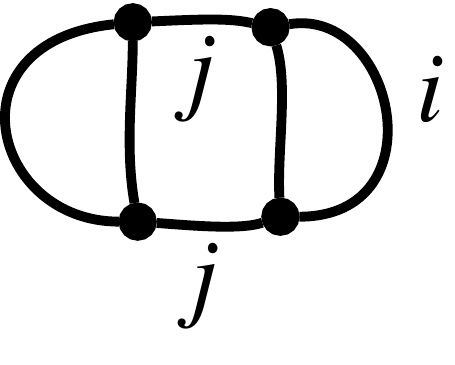}\end{minipage}\end{centering}} = \frac{\Delta_j}{\theta(1,1,j)\theta(1,i,j)}\\
\end{align*}
\begin{align*}
& = \ds \sum_j c_j^\prime  \begin{centering}\begin{minipage}{.8in}\includegraphics[width=.65in]{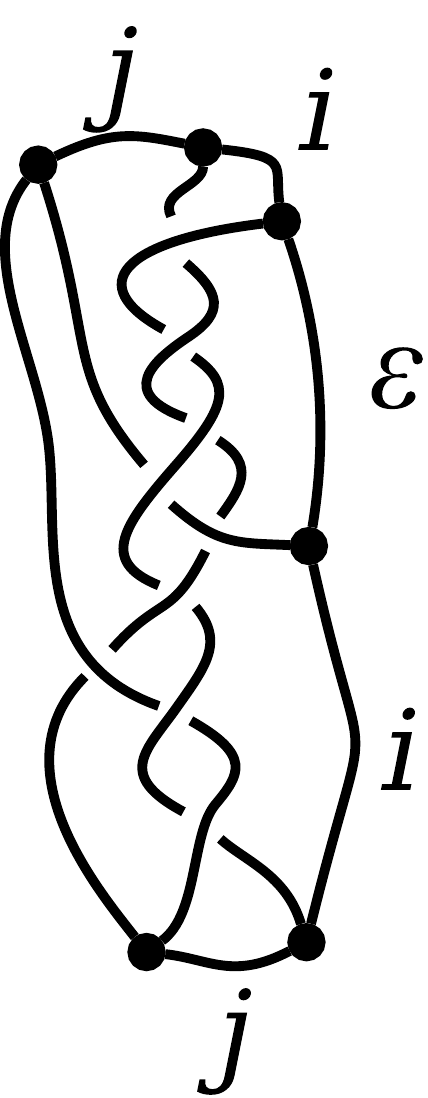}\end{minipage}\end{centering} \text{ where } c_j^\prime = c_j (\lambda_j^{11})^2\\
& = \ds \sum_{j,k} c_{j,k}  \begin{centering}\begin{minipage}{.8in}\includegraphics[width=.65in]{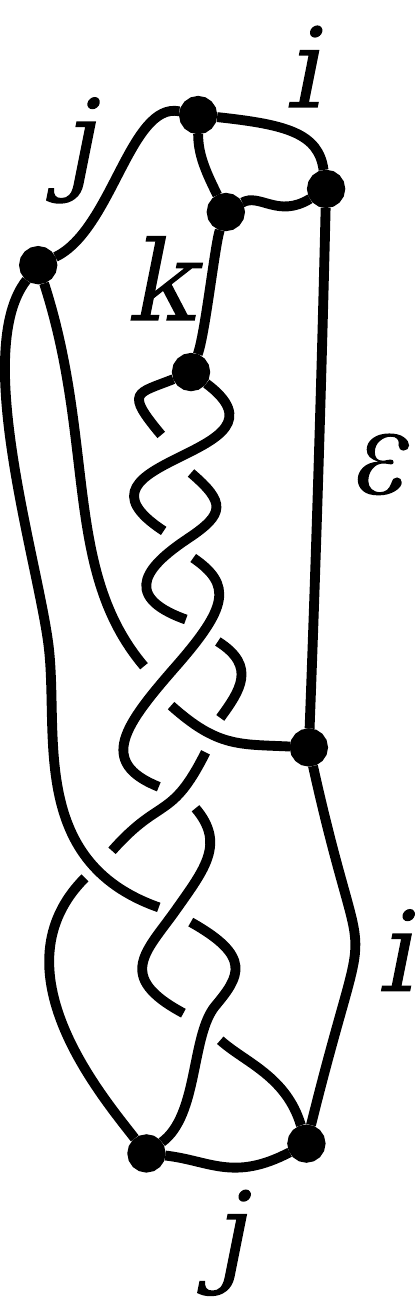}\end{minipage}\end{centering} \text{ where } c_{j,k} = c_j^\prime \frac{\Delta_k}{\theta(1,1,k)}\\
& \\
& = \ds \sum_{j,k} c_{j,k}^\prime  \hspace{.15in}\begin{minipage}{.8in}\includegraphics[width=.65in]{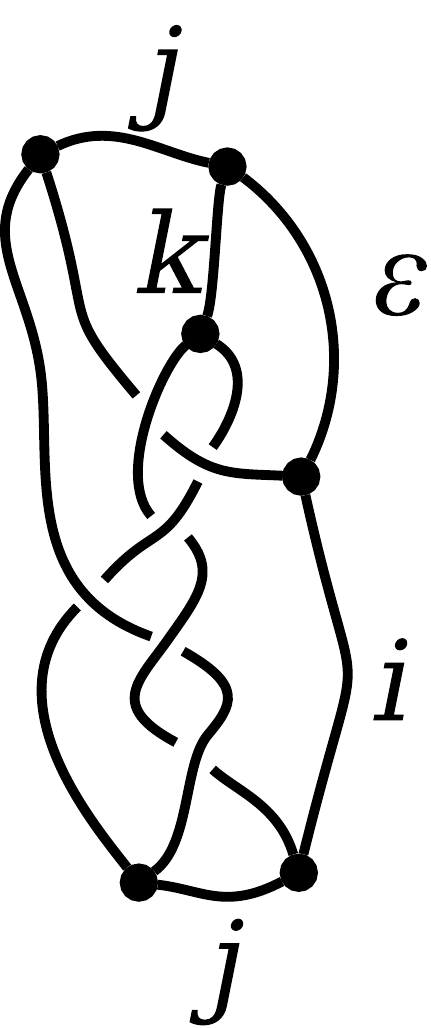}\end{minipage} \text{ where } c_{j,k}^\prime = c_{j,k} \frac{Tet \left[  \begin{array} {c c c} j & i & \vareps \\ 1 & k & 1 \\ \end{array}\right](\lambda_k^{11})^{-3}}{\theta(j,k,\vareps)}\\
& \\
& =  \ds \sum_{j,k,l} c_{j,k,l}  \hspace{.15in}\begin{minipage}{.8in}\includegraphics[width=.65in]{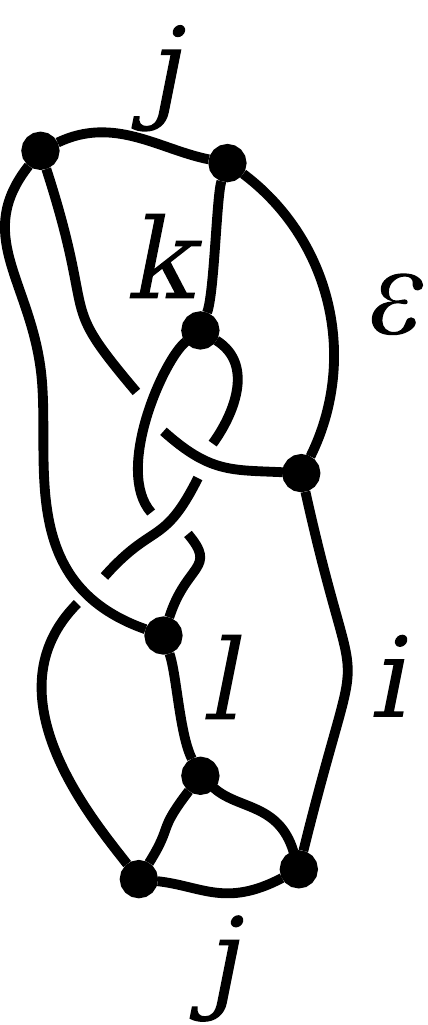}\end{minipage} \text{ where } c_{j,k,l} = c_{j,k}^\prime \frac{\Delta_l (\lambda_l^{11})^{-2}}{\theta(1,1,l)}\\
\end{align*}
\begin{align*}
& = \ds \sum_{j,k,l} c_{j,k,l}^\prime \hspace{.15in}\begin{minipage}{.8in}\includegraphics[width=.65in]{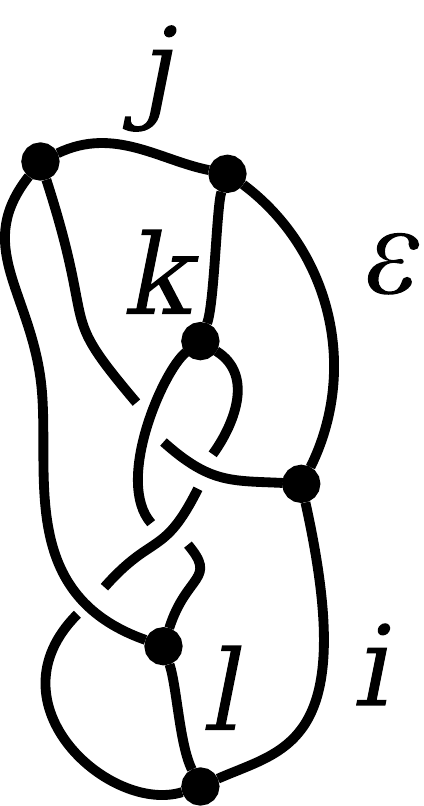} \end{minipage}\text{ where } c_{j,k,l}^\prime = c_{j,k,l} \frac{Tet \left[  \begin{array} {c c c} i & j & 1 \\ 1 & l & 1 \\ \end{array}\right]}{\theta(i,l,1)}\\
& =  \ds \sum_{j,k,l,m} c_{j,k,l,m}  \hspace{.15in}\begin{minipage}{.8in}\includegraphics[width=.7in]{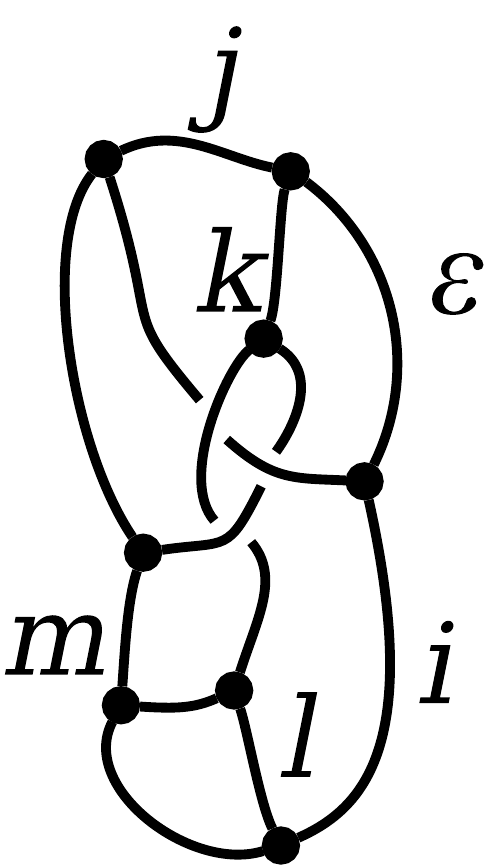}\end{minipage} \text{ where } c_{j,k,l,m} = c_{j,k,l}^\prime \frac{\Delta_m\lambda_m^{11}}{\theta(1,1,m)}\\
& \\
& = \ds \sum_{j,k,l,m} c_{j,k,l,m}^\prime \hspace{.15in} \begin{minipage}{.8in}\includegraphics[width=.7in]{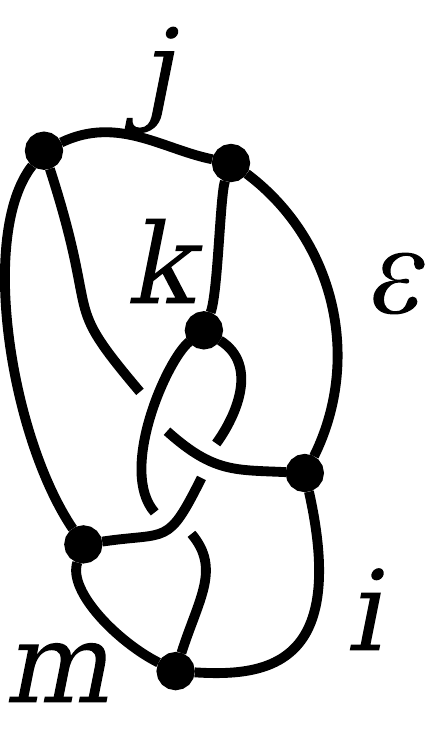}\end{minipage}\text{ where } c_{j,k,l,m}^\prime = c_{j,k,l,m} \frac{Tet \left[  \begin{array} {c c c} 1 & l & i \\ 1 & m & 1 \\ \end{array}\right]}{\theta(1,m,i)}\\
& \\
& =  \ds \sum_{j,k,l,m,n} c_{j,k,l,m,n}  \hspace{.15in}\begin{minipage}{.8in}\includegraphics[width=.7in]{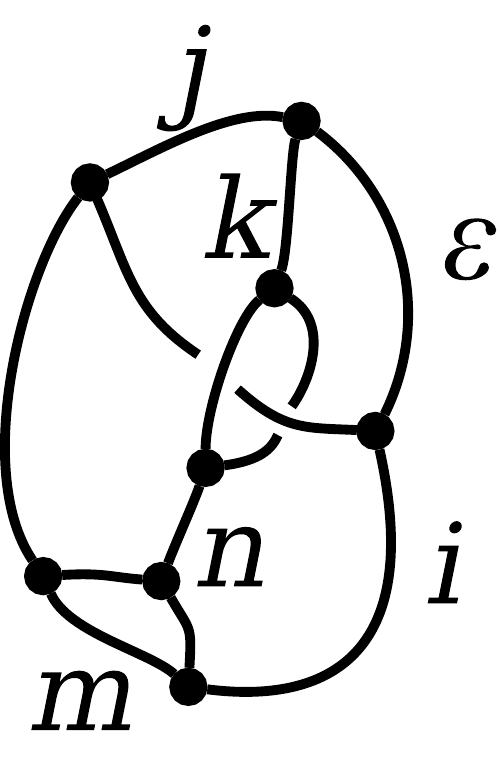}\end{minipage} \text{ where } c_{j,k,l,m,n} = c_{j,k,l,m}^\prime \frac{\Delta_n(\lambda_n^{11})^{-1}}{\theta(1,1,n)}\\
& \\
& = \ds \sum_{j,k,l,m,n} c_{j,k,l,m,n}^\prime \hspace{.15in} \begin{minipage}{.8in}\includegraphics[width=.7in]{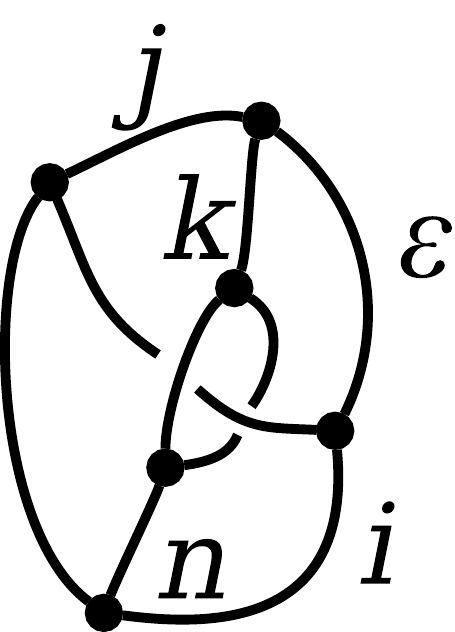}\end{minipage}\text{ where } c_{j,k,l,m,n}^\prime = c_{j,k,l,m,n} \frac{Tet \left[  \begin{array} {c c c} 1 & 1 & n \\ 1 & i & m \\ \end{array}\right]}{\theta(1,i,n)}\\
& \\
& =  \ds \sum_{j,k,l,m,n,p} c_{j,k,l,m,n,p}  \hspace{.15in}\begin{minipage}{.8in}\includegraphics[width=.7in]{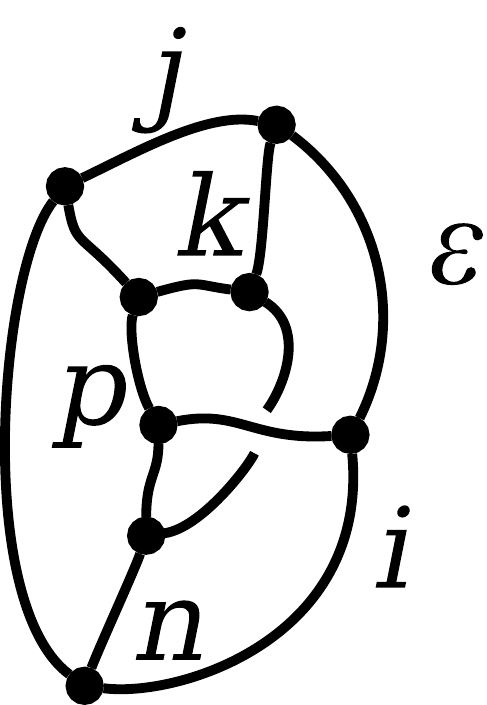}\end{minipage} \text{ where } c_{j,k,l,m,n,p} = c_{j,k,l,m,n}^\prime \frac{\Delta_p(\lambda_p^{11})^{-1}}{\theta(1,1,p)}\\
\end{align*}
\begin{align*}
& =  \ds \sum_{j,\ldots,p,q} c_{j,\ldots,p,q}  \hspace{.15in}\begin{minipage}{.9in}\includegraphics[width=.8in]{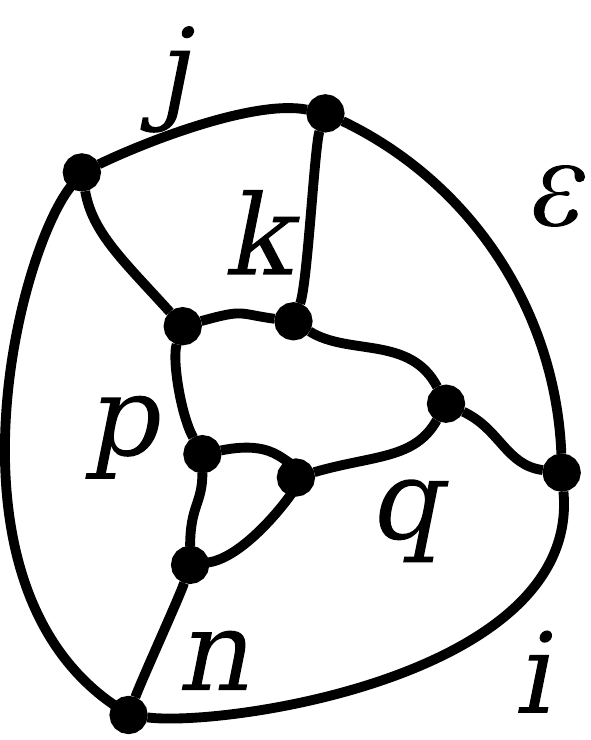}\end{minipage} \text{ where } c_{j,\ldots,p,q} = c_{j,k,l,m,n,p} \frac{\Delta_q(\lambda_q^{11})^{-1}}{\theta(1,1,q)}\\
& \\
& = \ds \sum_{j,\ldots,p,q} c_{j,\ldots,p,q}^\prime \hspace{.15in} \begin{minipage}{.9in}\includegraphics[width=.8in]{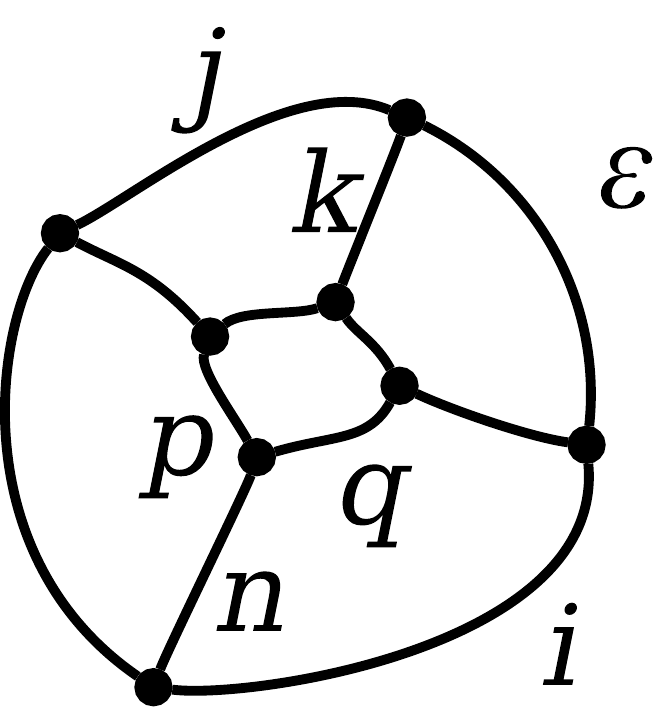}\end{minipage}\text{ where } c_{j,\ldots,p,q}^\prime = c_{j,\ldots,p,q} \frac{Tet \left[  \begin{array} {c c c} n & 1 & p \\ 1 & q & 1 \\ \end{array}\right]}{\theta(n,q,p)}\\
& \\
& = \ds \sum_{j,\ldots,p,q,r} c_{j,\ldots,p,q,r}  \hspace{.15in} \begin{minipage}{.9in}\includegraphics[width=.8in]{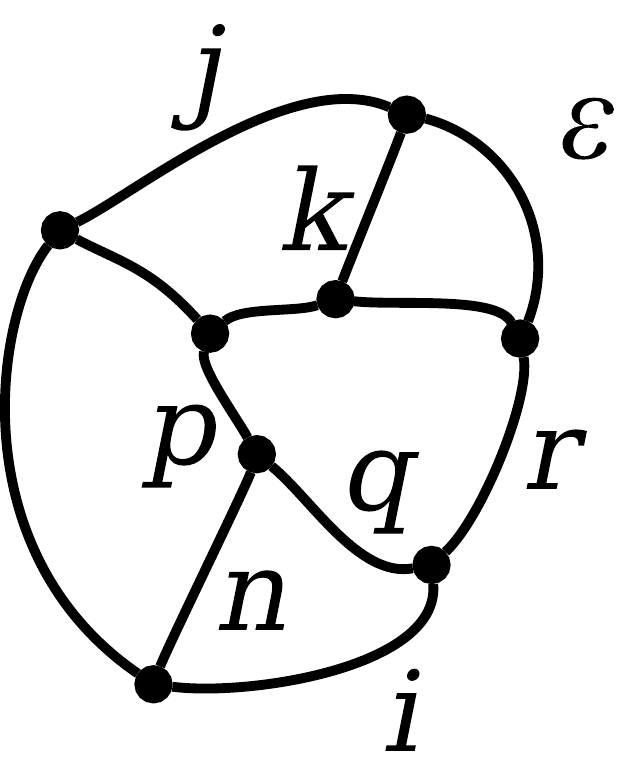}\end{minipage}\text{ where } c_{j,\ldots,p,q,r}  = c_{j,\ldots,p,q}^\prime \left\{\begin{array}{c c c}q & 1 & r\\ \vareps &i & 1 \end{array}\right\}\\
& \\
& = \ds \sum_{j,\ldots,p,q,r} c_{j,\ldots,p,q,r} ^\prime \hspace{.15in} \begin{minipage}{.7in}\includegraphics[width=.65in]{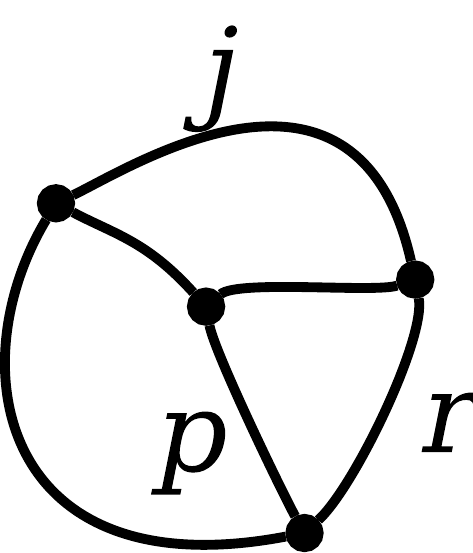}\end{minipage}\text{ where } c_{j,\ldots,p,q,r} ^\prime = c_{j,\ldots,p,q,r}  \frac{Tet \left[  \begin{array} {c c c} 1 & n & p \\ q & r & i \\ \end{array}\right]Tet \left[  \begin{array} {c c c} 1 & k & j \\ \vareps & r & 1 \\ \end{array}\right]}{\theta(1,r,p)\theta(1,r,j)}\\
& \\
& = \ds \sum_{j,\ldots,p,q,r}  Tet \left[  \begin{array} {c c c} 1 & p & 1 \\ 1 & j & r \\ \end{array}\right]c_{j,\ldots,p,q,r} ^\prime \hspace{.05in}.
\end{align*}

\section{}\label{app:generators}
We now give explicit expressions for the eight generators of $I_\f$.  We computed them using the same Mathematica notebook we used in Appendix \ref{app:linearcombo}.  The generators are as follows:
\[\begin{array}{r c l}
\pair{\f,x_0}/\delta & = & (c_{1,2}\theta(1, 2, 1) + c_{3,2} \theta(1,2,3) + c_{3,4}\theta(1,4,3))/\delta\\
& =& A^{-23}-3A^{-19}+7A^{-15}-10A^{-11}+12A^{-7}-14A^{-3}+12A-\\
& & 10A^5+6A^9-3A^{13}+A^{17} \\
& = & A^{-23}g_1,\\
\pair{\f,x_1}/\delta & = & (c_{1,2}\theta(1, 2, 1)(\phi_1 - \delta) + c_{3,2} \theta(1,2,3)(\phi_3 - \delta)\\
 & & + c_{3,4}\theta(1,4,3)(\phi_3 - \delta))/\delta\\
\end{array}\]
\[\begin{array}{r c l}
& =& -A^{-27} + A^{-25} + 3A^{-23}-2A^{-21}-9A^{-19} + 4A^{-17} + 16A^{-15}\\
& &  -3A^{-13} -23A^{-11} + 2A^{-9} + 30A^{-7} -2A^{-5} -31 A^{-3} -2A^{-1}\\
& & +30A+2A^3-24A^5-4A^7+17A^9+3A^{11}-9A^{13}-2A^{15}+\\
& & 4A^{17}+A^{19}-A^{21} \\
& = &  A^{-27}g_2,\\
& & \\
\pair{\f,x_2}/\delta & = &  (c_{3,2} \theta(1,2,3)(\phi_3 - \delta)(\phi_3 -\phi_1) + c_{3,4}\theta(1,4,3)(\phi_3 - \delta)(\phi_3 -\phi_1))/\delta\\
&= & A^{-27}-3A^{-23}-A^{-21}+4A^{-19}+2A^{-17}-6A^{-15}-A^{-13}+\\
& & 8A^{-11}+2A^{-9}-9A^{-7}-A^{-5}+10A^{-3}-A^{-1}-10A+9A^5-\\
& & 2A^7-7A^9+2A^{11}+6A^{13}-A^{15}-4A^{17}+A^{19}+2A^{21}-A^{25}\\
& =& A^{-27}g_3,\\
& & \\
\pair{\f,x_3}/\delta & = & (\phi_3 - \phi_2)\pair{\f,x_2}/\delta = A^{-35}g_4,\\
& & \\
\pair{\f,y_0}/\delta & = &( c_{1,2}(\lambda^{1\text{ } 1}_2)^{-2}\theta(1, 2, 1) + c_{3,2} (\lambda^{1\text{ } 3}_2)^{-1}(\lambda^{3\text{ } 1}_2)^{-1}\theta(1,2,3)\\
   & & + c_{3,4}(\lambda^{1\text{ } 3}_4)^{-1}(\lambda^{3\text{ } 1}_4)^{-1}\theta(1,4,3))/\delta\\
&= & 2A^{-21}-5A^{-17}+10A^{-13}-14A^{-9}+16A^{-5}-17A^{-1}+15A^3\\
& & -12A^7+7A^{11}-4A^{15}+A^{19} \\
& = & A^{-21}g_5,\\
& & \\
\pair{\f,y_1}/\delta & = & (c_{1,2}(\lambda^{1\text{ } 1}_2)^{-2}\theta(1, 2, 1)(\phi_1 - \delta) + c_{3,2} (\lambda^{1\text{ } 3}_2)^{-1}(\lambda^{3\text{ } 1}_2)^{-1}\theta(1,2,3)(\phi_3 - \delta)\\
   & & + c_{3,4}(\lambda^{1\text{ } 3}_4)^{-1}(\lambda^{3\text{ } 1}_4)^{-1}\theta(1,4,3)(\phi_3 - \delta))/\delta\\
& =& -A^{-29}+A^{-25}+2A^{-23}-3A^{-19}-5A^{-17}+5A^{-15}+11A^{-13}\\
& & -4A^{-11}-18A^{-9}+2A^{-7}+24A^{-5}-A^{-3}-25A^{-1}-2A+\\
& & 25A^3+3A^5-20A^7-5A^9+14A^{11}+3A^{13}-7A^{15}-3A^{17}+\\
& & 3A^{19}+A^{21} \\
& = &  A^{-29}g_6,\\
& & \\
\pair{\f,y_2}/\delta & = & (c_{3,2} (\lambda^{1\text{ } 3}_2)^{-1}(\lambda^{3\text{ } 1}_2)^{-1}\theta(1,2,3)(\phi_3 - \delta)(\phi_3 -\phi_1)\\
   & & + c_{3,4}(\lambda^{1\text{ } 3}_4)^{-1}(\lambda^{3\text{ } 1}_4)^{-1}\theta(1,4,3)(\phi_3 - \delta)(\phi_3 -\phi_1))/\delta\\
& = & A^{-37}-3A^{-33}-A^{-31}+5A^{-29}+2A^{-27}-7A^{-25}-2A^{-23}+\\
& & 9A^{-21}+2A^{-19}-11A^{-17}-A^{-15}+12A^{-13}-13A^{-9}+A^{-7}\\
& & +12A^{-5}-A^{-3}-10A^{-1}+2A+9A^3-2A^5-7A^7+5A^{11}\\
& & -A^{13}-3A^{15}+2A^{19}-A^{23}+A^{25}+A^{27}-A^{31}\\
& = & A^{-37}g_7, \text{ and }\\
& & \\
\pair{\f,y_3}/\delta & = & (\phi_3 - \phi_2)\pair{\f,y_2}/\delta = A^{-45}g_8.\\
\end{array}\]

\end{document}